\documentclass[11pt]{article}
\usepackage{graphicx}
\usepackage{stmaryrd}
\usepackage{amssymb}
\usepackage{amsthm}
\usepackage{dsfont}
\usepackage{hyperref}
\usepackage{mathrsfs}
\usepackage{stmaryrd}
\usepackage{stackrel}
\usepackage[lite,initials,msc-links]{amsrefs}
\usepackage{amsmath}
\usepackage{centernot}
\usepackage{enumitem}
\usepackage{verbatim}
\usepackage[margin=1.2in]{geometry}
\usepackage{xcolor}
\usepackage{mathtools}
\usepackage{color,soul}
\usepackage{sidecap,subcaption}
\usepackage[section]{placeins}
\usepackage{wrapfig}

\newtheorem{theorem}{Theorem}[section]
\newtheorem{lemma}[theorem]{Lemma}
\newtheorem{proposition}[theorem]{Proposition}
\newtheorem{corollary}[theorem]{Corollary}
\numberwithin{equation}{section}

\theoremstyle{remark} \newtheorem{remark}[theorem]{Remark} \numberwithin{equation}{section}
\numberwithin{figure}{section}

\newcommand\concel[2]{\ooalign{$\hfil#1\mkern0mu/\hfil$\crcr$#1#2$}}
\newcommand\nxlra{\mathrel{\mathpalette\concel{\longleftrightarrow}}}

\newcommand{\ep}{\varepsilon}
\newcommand{\set}{{\bf n}}
\newcommand{\m}{{\bf m}}
\newcommand{\sets}{\pmb\Omega}

\makeatletter
\newcommand*\bigcdot{\mathpalette\bigcdot@{.5}}
\newcommand*\bigcdot@[2]{\mathbin{\vcenter{\hbox{\scalebox{#2}{$\m@th#1\bullet$}}}}}
\makeatother

\newcommand{\n}{{\mathbf n}}

\definecolor{slightblue}{rgb}{.8, .8, 1}
\definecolor{hair}{RGB}{100,225,190}
\definecolor{ruby}{RGB}{220,50,120}
\definecolor{grass}{RGB}{150,220,110}

\title{Conformal invariance of double random currents II: tightness and properties in the discrete}

\author{Hugo Duminil-Copin\thanks{Institut des Hautes \'Etudes Scientifiques} \thanks{Universit\'e de Gen\`eve} \and Marcin Lis\thanks{Universit\"{a}t Wien} \and Wei Qian\thanks{CNRS and Laboratoire de Math\'{e}matiques d'Orsay, Universit\'{e} Paris-Saclay}}

\date{\today}

\begin{document}
\maketitle

\begin{abstract}
This is the second of two papers devoted to the proof of conformal invariance of the critical double random current on the square lattice. More precisely, we show convergence of loop ensembles obtained by taking the cluster boundaries in the sum of two independent critical currents (both for free and wired boundary conditions). 
The strategy is first to prove convergence of the associated height function to the continuum Gaussian free field, and then to characterize the scaling limit of the loop ensembles as certain local sets of this Gaussian Free Field.
In this paper, we derive crossing properties of the discrete model required to prove this characterization. 
\end{abstract}

\section{Introduction}

\subsection{Motivation}\label{sec:motivation}
Studying the large scale properties of discrete lattice models at criticality is one of the cornerstones of modern statistical physics. In the present papers we show conformal invariance of the critical double random current on the square lattice, i.e.~the percolation model obtained by summing two independent currents from the current representation of the Ising model. 

As often, there are two sides to the story when proving conformal invariance:
\begin{itemize}
\item one studies the discrete model to guarantee that subsequential scaling limits exist, and sometimes completes this result with the derivation of a few quantitative properties of the limit;
\item one characterizes any possible subsequential scaling limit, based either on the convergence of certain discrete holomorphic observables, or as in our case, thanks to the joint convergence of loops and a height function in a well-chosen coupling.
\end{itemize}
In this paper, we perform the first step. 
For the motivation related to the whole project and to the second item, we refer to the first paper \cite{DumLisQia21}. It will not come as a surprise that this article is mostly concerned with the so-called crossing estimates for the double random current model. The importance of precise crossing estimates became first evident in the study of Bernoulli percolation at the end of the seventies \cite{Rus78,SeyWel78}. In this context, the analysis of crossing estimates is  known under the coined name of Russo-Seymour-Welsh (RSW) theory. The RSW theory was developed extensively in the last ten years (see \cite{DumTas16} and references therein) for dependent percolation models, and this paper is another addition to the literature on the subject.
Compared to existing RSW results, the present framework presents two new interesting features: the model does not satisfy
\begin{itemize}
\item the FKG inequality;
\item uniform lower bounds on probabilities of crossing from boundary to boundary in fractal domains. This comes from the fact that the scaling limit is related to the Conformal Loop Ensemble CLE(4) whose loops are known not to touch the boundary (see \cite{DumLisQia21}). Let us mention that this is similar to what is expected for the critical random cluster model with cluster-weight $q=4$. 
\end{itemize}
 Studying these crossing estimates therefore requires new tools related to discrete harmonic measures and properties of the double random current model.

\subsection{Definition of the model}
A finite graph will be denoted by $G=(V,E)$, with vertex-set $V$ and edge-set $E$. We will often consider $G$ to be a subset of the square lattice $\mathbb Z^2$ with the vertex-set consisting of points $x=(x_1,x_2)$ with $x_1,x_2\in \mathbb Z$, and the edge-set consisting of unordered pairs $\{x,y\}\subset\mathbb Z^2$ with $\|x-y\|_1=1$. For a subgraph $G=(V,E)$  of $\mathbb Z^2$, let $\partial G$ be the set of $x\in V$ such that there exists an edge $\{x,y\}$ of the square lattice that does not belong to $E$. A {\em domain} $\Omega$ is a graph $G$ whose boundary $\partial G$ is a self-avoiding polygon of $\mathbb Z^2$.

 For two integers $n\le N$, set $\Lambda_n:=[-n,n]^2$ and ${\rm Ann}(n,N):=\Lambda_N\setminus\Lambda_{n-1}$. 
 We also write $\Lambda_n(x)$ and ${\rm Ann}(x,n,N)$ for the translates by $x$ of $\Lambda_n$ and ${\rm Ann}(n,N)$.
 
 In some (rare) occasions, we will also refer to the dual graph of a graph $G\subset \mathbb Z^2$. The dual graph $(\mathbb Z^2)^*$ of $\mathbb Z^2$ is $(\tfrac12,\tfrac12)+\mathbb Z^2$. For each edge $e$ of $\mathbb Z^2$, we write $e^*$ for the unique edge of $(\mathbb Z^2)^*$ intersecting it in its middle. The dual graph $G^*=(V^*,E^*)$ of $G=(V,E)$ is defined as follows: $E^*:=\{e^*:e\in E\}$ and $V^*$ is the set of endpoints of the vertices in $E^*$.
 
 \paragraph{Definition of the Ising model}

 Consider the Ising model with free boundary conditions on $G$ defined as follows. For spin configurations $\sigma \in\{-1,+1\}^V$ (the variable $\sigma_x$ is called the (Ising) {\em spin} at $x$), introduce the nearest-neighbor ferromagnetic Ising Hamiltonian with {\em free boundary conditions}
\[
H_G(\sigma):=-\sum_{\{x,y\}\in E}\sigma_x\sigma_y,
\]
and the Gibbs measure $\langle\cdot\rangle_{G,\beta}$ on $G$ given by 
\[
\langle X\rangle_{G,\beta}:=\frac{1}{Z_{G,\beta}^{\rm Ising}}\sum_{\sigma\in\{\pm1\}^V}X(\sigma)\exp[-\beta H_G(\sigma)],\qquad\forall X:\{-1,+1\}^V\longrightarrow \mathbb C,
\]
where $Z_{G,\beta}^{\rm Ising}:=\sum_{\sigma\in\{\pm1\}^V}\exp[-\beta H_G(\sigma)]$ is the {\em partition function} of the model.

In this paper, $\beta$ is always fixed to be equal to the critical inverse temperature 
\[
\beta_c:=\tfrac12\log(\sqrt 2+1)
\] 
of the Ising model on the square lattice, and we drop it from the notation. 


\paragraph{Definition of the random current and the double random current} A {\em current} $\n$ on  $G=(V,E)$ is an integer-valued function defined on the edges $E$.  
The current's set of {\em sources} is defined as the set
\begin{equation}
\partial\n \, := \, \{ x\in V :\, \sum_{y\in V:\{x,y\}\in E} \n_{\{x,y\}}\text{ is odd} \}.
\end{equation}
Let $\sets^B$ be the set of all currents with sources $B$. For a current $\n$ on $G$, we define the {\em critical weight} 
\begin{equation} 
w_{G}(\n) :=\prod_{\{x,y\}\in E}\frac{\beta_c ^{\n_{\{x,y\}}}}{\n_{\{ x,y\}}!}\,. 
\end{equation}
Currents are useful because of the following relation between their weighted sums and Ising spin correlations: if one defines, for a set $B\subset V$, the quantity ${Z}^B(G):=\sum_{\n\in \sets^B}w_{G}(\n)$, and one writes $\sigma_B= \prod_{x\in B}$, then
\begin{equation}
\langle \sigma_B\rangle_G=\frac{Z^B(G)}{Z^\emptyset(G)}.
\end{equation}
We introduce a probability measure on currents with sources $B\subset V$ (with $|B|$ even) by 
\begin{align}\label{eq:rc}
\mathbf P^B_G(\set) := \frac {w_{G}(\n)}{{Z}^B(G)} \qquad \text{for all }\set \in \sets^B.
\end{align}
The random variable $\n$ is called the {\em critical random current with free boundary conditions and sources $B$}.
When $B=\emptyset$, we speak of {\em sourceless} currents.
We will also write $\mathbf P^{A,B}_{G,H}$ for the law of $(\n_1,\n_2)$, where $\n_1$ and $\n_2$ are two independent currents drawn according to $\mathbf P^A_G$ and $\mathbf P^B_H$ respectively. 
Under this law, the sum $\n_1+\n_2$ is called the \emph{critical double random current}.
Finally we note that random currents can be defined in the infinite volume, as explained in \cite{AizDumSid15}. In this case we denote the measure by $\mathbf P^{A,B}_{\mathbb Z^2,\mathbb Z^2}$.

Let us mention that the double random current model has proved to be a very powerful tool in the study of the Ising model. Its applications range from  correlation inequalities \cite{GHS}, exponential decay in the off-critical regime \cite{AizBarFer87,DumTas15,DumGosRao18}, classification of Gibbs states \cite{Rao17}, etc. Even in two dimensions, where a number of other tools are available, new developments have been made possible via the use of this representation \cite{DumLis,LisT,ADTW,LupWer}. For a more exhaustive account of random currents, we refer the reader to \cite{Dum16}.

We write $A\stackrel{\n}\longleftrightarrow B$ if there exists $v_0,\dots,v_k$ with $v_0\in A $, $v_k\in B$, with $\{v_i,v_{i+1}\}\in E$ and $\n_{\{v_i,v_{i+1}\}}>0$ for every $0\le i<k$. We call a {\em cluster} of $\n$ a connected component of the graph with vertex-set $V$ and edge-set $E(\n):=\{e\in E:\n_e>0\}$. 
We will say that a subgraph of $H\subset G$ is a {\em $H$-cluster} if it is a cluster of $\n$ when {\em restricted to the edges in} $H$. Note that the $H$-clusters of $\n$ are not necessarily equal to the restrictions of the clusters of $\n$ to $H$ (as several $H$-clusters may be connected to each other outside of $H$ and therefore belong to the restriction to $H$ of the same cluster in $\n$).

\subsection{Main results}

In order to implement the scheme described in the first of our papers \cite{DumLisQia21}, several properties of the model need to be derived.
We start by mentioning the Aizenman-Burchard criterion, which is in fact fairly straightforward to obtain.
For an integer $k\ge1$, let $A_{2k}(r,R)$ be the event\footnote{The subscript $2k$ instead of $k$ is meant to illustrate that there are $k$ ${\rm Ann}(r,R)$-clusters from inside to outside separated by $k$ ``dual'' clusters.} that there are $k$ distinct  ${\rm Ann}(r,R)$-clusters in $\n_1+\n_2$ that are crossing ${\rm Ann}(r,R)$. 

\begin{theorem}[Aizenman-Burchard criterion for the double random current model]
\label{prop:tight number crossings}
There exist sequences $(C_k)_{k\geq1}$ and $(\lambda_k)_{k\geq 1}$, with the latter tending to infinity as $k\to \infty$, such that  for every domain $\Omega$, every $k\ge1$ and all $r,R$ with $1\le r\le R$,

\begin{align}
\label{eq:tight number XOR}{\bf P}_{\Omega,\Omega}^{\emptyset,\emptyset}[A_{2k}(r,R)]&\le (C_k\tfrac rR)^{\lambda_k}.\end{align}
\end{theorem}
Here, we do not a priori assume that $\Omega$ contains $\Lambda_R$.

Contrarily to previously known results about other dependent percolation models, crossing probabilities for the critical double random current do not remain bounded away from zero uniformly in the domain $\Omega$. This is a feature which makes the following theorem very interesting (it may somehow look surprising to be able to derive it without referring to the scaling limit of the double random current).
For a set $\Omega$, let $\partial_r\Omega$ be the set of vertices in $\Omega$ that are within a distance $r$ from $\partial\Omega$. 
\begin{theorem}[Connection probabilities close to the boundary for double random current]\label{thm:crossing free}
There exists $c>0$ such that for all $r,R$ with $1\le r\le R$ and every $R$-centred domain $\Omega$, 
$$\frac{c}{\log(R/r)}\le{\bf P}_{\Omega,\Omega}^{\emptyset,\emptyset}[\Lambda_R\stackrel{\n_1+\n_2}\longleftrightarrow\partial_r\Omega]\le \epsilon(\tfrac rR),$$
where $x\mapsto \epsilon(x)$ is an explicit function tending to 0 as $x$ tends to 0. 
\end{theorem}

We note that the assumption that $\Omega$ does not contain $\Lambda_{3R}$ is only necessary for the lower bound. The proof also gives the same lower bound for the probability that $\Lambda_R$ is connected to $\partial_r\Omega\cap \Lambda_{4R}$. 

We predict that the upper bound should be true for $\epsilon(x):=C/\log(1/x)$ 
but we do not need such a precise estimate here. The result is coherent with the fact that the scaling limit of the outer boundary of large clusters in $\n_1+\n_2$ is given by CLE(4) (see \cite{DumLisQia21}), which is known to be made of simple loops that do not intersect the boundary of the domain. Interestingly, to derive the convergence to the continuum object it will be necessary to first prove this result at the discrete level.

The lower bound is to be compared with recent estimates \cite{DumSidTas17,DumManTas20} obtained for another dependent percolation model, namely the critical random cluster model with cluster-weight $q\in[1,4)$. There, it was proved that the crossing probability is bounded from below by a constant $c=c(q)>0$ uniformly in $r/R$. On the other hand, we expect that the behaviour of the critical random cluster model with cluster weight $q=4$ is comparable to the behaviour presented here: large clusters do not come close to the boundary of domains when the boundary conditions are ``free''.

We conclude this paper with a series of results that are both important as intermediary steps in the proof of our main results, and also play an essential role in \cite{DumLisQia21}.

The first one deals with the possibility of two large clusters of the double random current coming close to each other. More formally, let
\begin{align*}
A_4^\square(r,R)&:=\{\text{there exist two $\Lambda_R$-clusters crossing $\mathrm{Ann}(r,R)$}\}\end{align*}
and let $A_4^\square(x,r,R)$ be the translate of $A_4^\square(r,R)$ by $x$.

\begin{theorem}\label{thm:absence closed pivotal expectation white}
There exists $C>0$ such that for all $r,R$ such that $1\le r\le R$,
\begin{align}
\label{eq:closed pivotal0}{\bf P}_{\mathbb Z^2,\mathbb Z^2}^{\emptyset,\emptyset}[A_4^\square(r,R)]&\le C(r/R)^2.
\end{align}
Furthermore, for every $\ep>0$, there exists $\eta=\eta(\ep)>0$ such that for all $r,R$ such that $1\le r\le \eta R$ and every domain $\Omega\supset\Lambda_{2R}$,
\begin{align}
\label{eq:closed pivotal existence0}{\bf P}_{\Omega,\Omega}^{\emptyset,\emptyset}[\exists x\in \Lambda_R:A_4^\square(x,r,R)]&\le \ep.
\end{align}
\end{theorem}

Equation \eqref{eq:closed pivotal0} implies that the expected number of $x\in r\mathbb Z^2\cap \Lambda_R$ such that $A_4^\square(x,r,R)$ occurs is $O(1)$. This is to be compared, for instance in the case of $A_4^\square$, with random cluster models with $1\le q<4$, for which the expected number of so-called pivotal boxes is polynomially large in $R/r$. In this case, it is also proved that with positive probability, there exists a pivotal box. Here, we see from \eqref{eq:closed pivotal existence0} that this is not true and that the probability of seeing a pivotal box is tending to 0.

The second theorem deals with another event of interest. For a current $\n$, let $\n^*$ be the set of dual edges $e^*$ with $\n_e=0$. For a dual path $\gamma=(e_1^*,e_2^*,\dots,e_k^*)$, call the $\n$-flux through $\gamma$ to be the sum of the $\n_{e_i}$. 
Call a {\em ${\rm Ann}(r,R)$-hole in $\n_1+\n_2$} a connected component of $(\n_1+\n_2)^*$ in ${\rm Ann}(r,R)^*\cap \Omega^*$ (note that it can be seen as a collection of faces). A ${\rm Ann}(r,R)$-hole is said to be {\em crossing} ${\rm Ann}(r,R)$ if it intersects $\partial\Lambda_r^*$ and $\partial\Lambda_R^*$.  Consider the event
\begin{align*}
A_4^\blacksquare(r,R)&:=\Big\{\begin{array}{c}\text{there exist two ${\rm Ann}(r,R)$-holes crossing ${\rm Ann}(r,R)$ and the}\\
\text{shortest dual path between them has even $(\n_1+\n_2)$-flux}\end{array}\Big\}\end{align*}
(see Fig.~\ref{fig:events_pivotal})
and its translate by $x$, denoted by $A_4^\blacksquare(x,r,R)$.

\begin{theorem}\label{thm:absence closed pivotal expectation black}
There exists $C>0$ such that for all $r,R$ such that $1\le r\le R$,
\begin{align}
\label{eq:open pivotal}{\bf P}_{\mathbb Z^2,\mathbb Z^2}^{\emptyset,\emptyset}[A_4^\blacksquare(r,R)]&\le C(r/R)^2.
\end{align}
Furthermore, for every $\ep>0$, there exists $\eta=\eta(\ep)>0$ such that for all $r,R$ such that $1\le r\le \eta R$ and every domain $\Omega\supset\Lambda_{2R}$,
\begin{align}
\label{eq:open pivotal existence}{\bf P}_{\Omega,\Omega}^{\emptyset,\emptyset}[\exists x\in \Lambda_R:A_4^\blacksquare(x,r,R)]&\le \ep.
\end{align}
\end{theorem}

At this stage, we want to highlight the fact that the condition on the $(\n_1+\n_2)$-flux is important as otherwise the bound is wrong for the probability of the existence of two holes coming close to each other.

\begin{figure}
\begin{center}
			\includegraphics[scale=1.00]{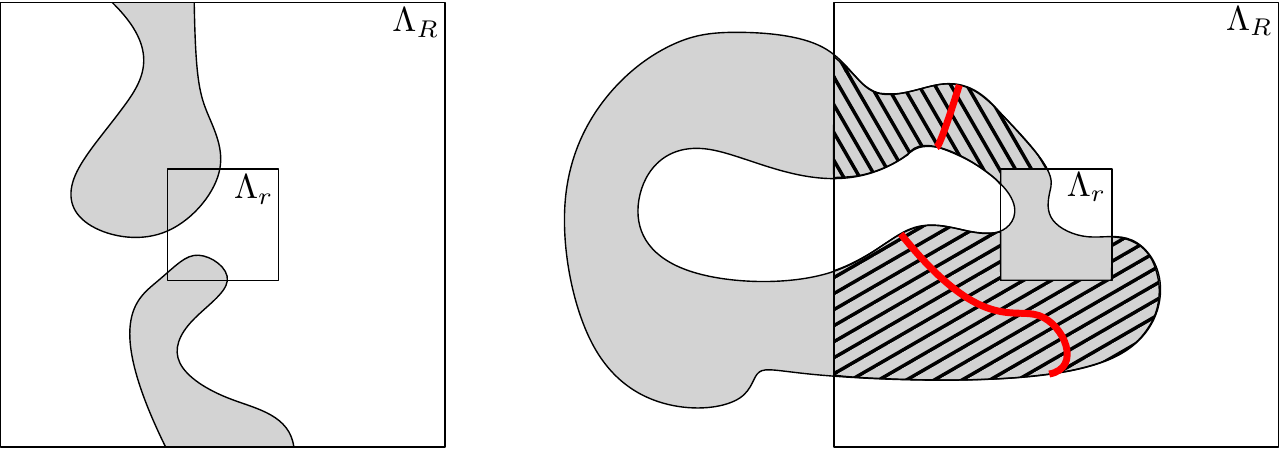}
		\end{center}
\caption{A depiction of the events $A_4^\square(r,R)$ and $A_4^\blacksquare(r,R)$. The $(\n_1+\n_2)-$flux across each one of the red paths must be even. }\label{fig:events_pivotal}
\end{figure}

\paragraph{Acknowledgements} The first author was supported by the NCCR SwissMap from the FNS. This project has received funding from the European Research Council (ERC) under the European Union's Horizon 2020 research and innovation programme (grant agreement No.~757296). The beginning of the project involved a number of people, including Gourab Ray, Benoit Laslier, and Matan Harel. We thank them for inspiring discussions. The project would not have been possible without the numerous contributions of Aran Raoufi. We are very thankful to him.

\paragraph{Organization} In Section~\ref{sec:background}, we review some background on the random cluster and random current models. The results  can be found in the literature and are briefly mentioned, without proofs.
In Section~\ref{sec:preliminaries hugo}, we present several new results that are of general interest. This includes a mixing property for random currents and some monotonicity properties of the double random current. While these results are not the most difficult results of this paper, we think that they may be of independent interest for the study of the planar Ising model. 
In Section~\ref{sec:crossing estimate}, we prove Theorem~\ref{thm:crossing free}. Section~\ref{sec:pivotal points} is devoted to the proofs of Theorems~\ref{thm:absence closed pivotal expectation white} and \ref{thm:absence closed pivotal expectation black}.
Finally, in Section~\ref{sec:tightness}, we prove Theorem~\ref{prop:tight number crossings}.

\begin{remark}
The previous version of this article also contained results on the critical XOR Ising model that are no longer relevant 
for the current scope of the article but that should still be useful for proving Wilson's conjecture on the XOR Ising model~\cite{Wilson}.
\end{remark}

\section{Background}\label{sec:background}

\subsection{The switching lemma for the double random current}

We will repeatedly use the following classical property of the double random current, see e.g.~\cite{Aiz82} or \cite{AizDumSid15} for the proof of the statement below. 
\begin{lemma}[Switching lemma]\label{lem:switching lemma}
Consider two graphs $H\subset G$ and two sets $A$ and $B$ in $G$ and $H$ respectively. For every functional $F$ from currents on $G$ with source-set $A\Delta B$ into $\mathbb C$, we have
\[
\mathbf E^{A,B}_{G,H}[F(\n_1+\n_2)]=\frac{\langle\sigma_A\sigma_B\rangle_G}
{\langle\sigma_B\rangle_G\langle\sigma_A\rangle_H}\mathbf E^{A\Delta B,\emptyset}_{G,H}[F(\n_1+\n_2)\mathbf 1_{(\n_1+\n_2)_{|H}\in \mathcal F_A}],
\]
where $\n\in \mathcal F_A$ is the event that every cluster of $\n$ intersects an even number of vertices in $A$ (it may be none).
\end{lemma}

We will sometimes refer to a generalization of the switching lemma, referred to as the switching principle, given in \cite[Lemma 2.1]{ADTW}. In order to state it, we introduce a representation in which  
 a current configuration $\n$ is presented  as a (multi-)graph $\mathcal N$ obtained by replacing each edge  $e$ by $\n_e$ edges, all linking the endpoints of $e$.   By default, we shall denote the multigraph corresponding to $\n$ or $\m$ by the appropriate calligraphic script symbol  
$\mathcal N$ or $\mathcal M$.  
We extend the above correspondence to the  weight and source notation, so that 
$\partial \mathcal N := \partial \mathcal \n$ and $w(  \mathcal N)  := w( \mathcal \n)$, and similarly for $\mathcal M$ in relation to $\m$.  

 The switching principle is stated as follows.

\begin{lemma}[Switching principle] \label{lem:switch}
 For any set $A$ of vertices on $G$, any multigraph $\mathcal M$ such that there exists  $\mathcal K \subset \mathcal M$ with  $ \partial \mathcal K = A$, and any function $ f $ of a current:
\begin{equation} \label{eq:switch} 
  \sum_{\substack {\mathcal N \subset \mathcal M \\  \partial \mathcal N = A } }  f(\mathcal{N})  \ = \  
       \sum_{\substack {\mathcal N \subset \mathcal M \\  \partial \mathcal N = \emptyset } }    f(\mathcal{N} \Delta \mathcal{K}). 
\end{equation} 
\end{lemma} 

This result will be used as follows. Consider a current $\m$ that contains a cluster separating two faces $u$ and $v$ in $\mathbb Z^2$. Then, if we consider the $\n$-flux between $u$ and $v$, i.e.~the $\n$-flux of $\n$ through a shortest dual path going from $u$ to $v$, we have that 
\[
\mathbf P_{\Omega,\Omega}^{\emptyset,\emptyset}[\n_1-\text{flux is odd}|\n_1+\n_2=\m]=\tfrac12.
\]
Indeed, when interpreting the currents in terms of multigraphs, we can rephrase the previous identity as follows: for every $\mathcal M$, if $f$ is the indicator function that $\mathcal N$ has an odd flux between the faces $u$ and $v$,
\begin{equation} \label{eq:switch2} 
  \sum_{\substack {\mathcal N \subset \mathcal M \\  \partial \mathcal N = \emptyset } }  f(\mathcal{N})  \ = \    \sum_{\substack {\mathcal N \subset \mathcal M \\  \partial \mathcal N = \emptyset } }  (1-f(\mathcal{N})).
\end{equation} 
Yet, if $\m$ disconnects $u$ from $v$, that means that there exists $\mathbf k\le \m$ such that the associated multigraph $\mathcal K$ is a simple loop going either around $u$ but not $v$, or the opposite. In particular, $\mathcal K$ has an odd flux between $u$ and $v$. We deduce that 
\[
1-f(\mathcal{N})=f(\mathcal N\Delta\mathcal K),
\]
and \eqref{eq:switch2} is a direct consequence of \eqref{eq:switch}.

\subsection{Definition and basic properties of the random cluster model}

We will use extensively the random cluster model and its basic properties that we now recall.
\paragraph{Definition} A percolation configuration $\omega$ on a graph $G=(V,E)$ is a function from $E$ into $\{0,1\}$. It is most of the time seen as a subgraph of ${G}$ with vertex-set $V$ and edge-set $\{e\in E:\omega_e=1\}$. The set of percolation configurations on $G$ is denoted by $\mathcal E({G})$. A boundary condition $\xi$ on ${G}$ is a partition of the vertices in ${G}$. Here and below, we do not require that $\xi$ is restricted to the actual boundary of the graph, as we will use these boundary conditions to merge vertices together later on in the paper. When the boundary condition is wiring vertices on $\partial{G}$ only, we speak of a {\em boundary condition on }$\partial{G}$.

We will also use the notation $A\stackrel{\omega}\longleftrightarrow B$ for the existence of a path $v_0,\dots,v_k$ with $v_0\in A$, $v_k\in B$ and $\{v_i,v_{i+1}\}\in E$ with $\omega_{\{v_i,v_{i+1}\}}=1$ for every $0\le i<k$. We call {\em cluster}  a connected component of the graph $\omega$.

The random cluster measure with edge-weight $p$, cluster-weight $q=2$, and boundary condition $\xi$ on ${G}$ will be denoted by 
\begin{align}\label{eq:rc}
\phi^\xi_{G,p}(\omega) := \frac 1{{Z}_{\rm RCM}^\xi(G)} 2^{k(\omega^\xi)}p^{\sum\omega_e}(1-p)^{|E|-\sum\omega_e} , \qquad \text{for all }\omega\in \mathcal E(G),
\end{align}
where $k(\omega^\xi)$ is the number of clusters in the configuration $\omega^\xi$ obtained by merging all the vertices that are wired together in $\xi$. When the boundary condition is made of singletons only, we refer to it as the {\em free boundary condition} and write $0$ instead of $\xi$. We will also consider {\rm wired} boundary conditions on part of $\partial G$, which corresponds to wiring all the vertices of this part into one element of the partition $\partial G$.

Below, we will always fix the parameter $p$ to be equal to the critical parameter $p_c:=\sqrt 2/(1+\sqrt 2)$ and we drop $p$ from the notation.

\paragraph{Spatial Markov property} Let us start by mentioning that the random cluster model satisfies the {\em spatial Markov property}: for every graph ${G}=(V,E)$ and $F\subset E$, let ${G}'$ be the graph induced by $F$ (i.e.~the graph with edge-set $F$ and vertices made of the endpoints of these edges). For every boundary condition $\xi$ on ${G}$ and every percolation configuration $\psi$  on $E\setminus F$, we have that 
\[
\phi_{G}^\xi[\cdot_{|F}|\omega_{|E\setminus F}=\psi_{|E\setminus F}]=\phi_{{G}'}^{\psi^\xi}[\cdot],
\]
where $\psi^\xi$ is the boundary condition for which $x$ is wired with $y$ if and only if 
the two vertices are connected in the graph $\psi^\xi$.

\paragraph{Positive association} We will be using a number of other properties of this model, among which are the consequences of its positive association. Below, a random variable $X:\mathcal E(G)\rightarrow\mathbb R$ is said to be increasing if $X(\omega)\le X(\omega')$ for every $\omega\le\omega'$, where $\le$ is the partial order on functions from $\mathcal E(G)$ to $\mathbb R$. Then, we have
\begin{itemize}
\item (FKG inequality) for every $X$ and $Y$ increasing, 
\begin{equation}\label{eq:FKG}
\phi^\xi_{G}[XY]\ge \phi_{G}^\xi[X]\phi_{G}^\xi[Y].
\end{equation}
\item (monotonicity in boundary conditions) for every $X$ increasing, every ${G}$, and every $\xi\le \xi'$, where $\xi\le \xi'$ means that every two vertices that are wired in $\xi$ are also wired in~$\xi'$,\begin{equation}
\phi_{G}^\xi[X]\le \phi_{{G}}^{\xi'}[X].
\end{equation}
\end{itemize}
\begin{remark}
This will often be  used in the following context: we will create a new graph by merging vertices of $\Omega$. This will be equivalent to wiring them in the sense of boundary conditions above and therefore will increase averages of increasing random variables.
\end{remark}

\paragraph{Crossing estimates}
 We will repeatedly use the following theorem, which was first proved in \cite{DumHonNol11}. \begin{theorem}[Crossing estimates for the random cluster model]\label{thm:RSW}
 For every $\kappa\in(0,\infty)$, there exists $c=c(\kappa)>0$ such that for every quad $(D,a,b,c,d)$ with $\ell_D[(ab),(cd)]\in(\kappa,1/\kappa)$ and every boundary condition $\xi$ on $\partial D$,
 \begin{equation}\label{eq:RSW}
 c\le \phi_D^\xi[(ab)\stackrel\omega\longleftrightarrow(cd)]\le 1-c.\tag{RSW}
 \end{equation}
 \end{theorem}
 
 We will use this result extensively in the next sections.  The recent literature contains numerous applications of such estimates. 
 We refer to these papers for details.
 
\paragraph{Couplings} The random cluster model is directly related to the random current model and the Ising model in the following ways. We do not mention all the properties of the corresponding coupling as we will only be using them sporadically.

 \begin{proposition}[Random current -- Random cluster coupling \cite{LupWer,ADTW}]\label{prop:coupling random cluster}
Consider a graph $G$ and a set $B\subset G$ of even cardinality. Let $\omega$ be the configuration constructed from $\n\sim\mathbf P^B_G$ as $\omega_e=1$ if either $\n_e>0$ or $\alpha_e=1$, where $\alpha$ is an independent family of Bernoulli random variables of parameter $1-\sqrt{1-p_c}$. Then,
 \[
 \omega\sim \phi_G^0[\cdot|\mathcal F_B],
 \]
 where $\mathcal F_B$ is the event that every cluster of $\omega$ contains an even number of vertices in $B$ (it may be none).
 \end{proposition}
 
  \begin{proposition}[Edwards--Sokal coupling \cite{EdwSok}]\label{prop:coupling ES} Consider a graph $G$. Let $\sigma$ be the configuration constructed from $\omega\sim \phi_G^0$ by assigning to each cluster $\mathcal C$ of $\omega$ a spin $\sigma_\mathcal C$ uniformly at random, and by writing $\sigma_x=\sigma_\mathcal C\in\{-1,+1\}$ for every $x\in \mathcal C$. Then, 
 \[
 \sigma\sim \langle\cdot\rangle_G.
 \]
 \end{proposition}
 
 Let us remark that we immediately obtain from this construction that for every $B\subset G$,
 \begin{equation}
 \langle\sigma_B\rangle_G=\phi_G^0[\mathcal F_B].
 \end{equation}
 We now mention a coupling between the odd part of a random current and interfaces in a dual Ising model.

 \begin{proposition}[Kramers--Wannier duality \cite{KraWan41}]\label{prop:coupling Kramers Wannier}Consider a subgraph $G$ of $\mathbb Z^2$ where $\partial G$ is a self-avoiding polygon and $B$ a subset of $\partial G$. Let $\n\sim\mathbf P^B_G$ with $B\subset\partial G$ and $\eta=\eta(\n)$ be the set of edges for which $\n$ is odd. The configuration $\eta$ has the same law as the edges $e$ bordering faces of different sign for an Ising model at inverse-temperature $\beta_c^*:=\beta_c$ on the dual graph $G^*$, where the boundary conditions on $\partial G^*$ are such that the spin change along edges that are incident to the primal vertices in $B$. \end{proposition}

 \subsection{Harmonic estimates for the random cluster model}

We will need precise estimates for the random cluster model with so-called mixed boundary conditions. More precisely, consider a quad $(D,a,b,c,d)$  and let $\phi_D^{(ab),(cd)}$ be the measure with wired boundary condition on $(ab)$, wired on $(cd)$, and free elsewhere. We are seeking estimates on $\phi_D^{(ab),(cd)}[(ab)\longleftrightarrow(cd)]$ that are written in terms of discrete harmonic estimates. We refer to \cite{CheDumHon16} for details.

Fix a domain $D$ and attach to each edge  a conductance $w_e$ equal to 1 for edges between vertices of $D$, $2(\sqrt 2-1)$ for edges exiting $D$ along $(bc)$ and $(da)$, and 0 otherwise. We also write $m_x:=\sum_{y}w_{xy}$ for the sum of the conductances around a vertex. Below, let
\[
Z_{D}[x,y]:=\sum_{\gamma\subset D:x\mapsto y}m_y^{-1}\prod_{1\le i<k}\frac{w_{\gamma_i\gamma_{i+1}}}{m_{\gamma_i}},\]
where the sum runs over paths $\gamma=(\gamma_i)_{0\le i\le k}$ of vertices in $D$ going from $x$ to $y$. 
\begin{remark}
The quantity is related, up to an explicit constant, to the discrete Green function $G_D(x,y)$ associated to the conductances $w_e$, or equivalently to the expected number of visits of a random walk associated to the conductances above starting from $x$.
\end{remark}
We introduce, for $X,Y\subset D$,
\[
Z_{D}[X,Y]:=\sum_{x\in X}\sum_{y\in Y}Z_{D}[x,y].\]
We also define the same quantities on the dual graph $D^*$, and refer to them as $Z_{D^*}[u,v]$ and $Z_{D^*}[U,V]$. 

\begin{remark}
The following observation will be convenient: for every $\kappa\in(0,\infty)$, there exists $\kappa'=\kappa'(\kappa)\in(0,\infty)$ such that $\ell_D[(ab),(cd)]\in(\kappa,1/\kappa)
$ implies that $Z_D[(ab),(cd)]$ and $Z_{D^*}[(bc)^*,(da)^*]$ belong to $(\kappa',1/\kappa')$.\end{remark}

We will  use the following estimate, see Fig.~\ref{fig:harmonic}. 

\begin{corollary}\label{cor:bound harmonic measure}
There exist $c,C\in(0,\infty)$ such that for all $r,R$ and every  domain $\Omega$ that contains $\Lambda_{2R}$ but not $\Lambda_{12R}$, and $\Lambda_{2r}(x)$ but not $\Lambda_{12r}(x)$,
\begin{equation}\label{eq:bound harmonic measure}
cZ_{\Omega^*}[x^*,0^*]\le {\bf P}_{\Omega^{\bullet\bullet},\Omega^{\bullet\bullet}}^{\emptyset,\emptyset}[\Lambda_r(x)\stackrel{\n_1+\n_2}\longleftrightarrow\Lambda_R]
\le CZ_{\Omega}[x,0],
\end{equation}
where $\Omega^{\bullet\bullet}$ is the graph obtained by merging the vertices of $\Lambda_r(x)$ together, and those of $\Lambda_R$ together (we identify the obtained vertices with the sets $\Lambda_r(x)$ and $\Lambda_R$ themselves), and $0^*$ and $x^*$ are dual vertices adjacent to $0$ and $x$ respectively. 
\end{corollary}

\begin{figure} 
		\begin{center}
			\includegraphics[scale=0.9]{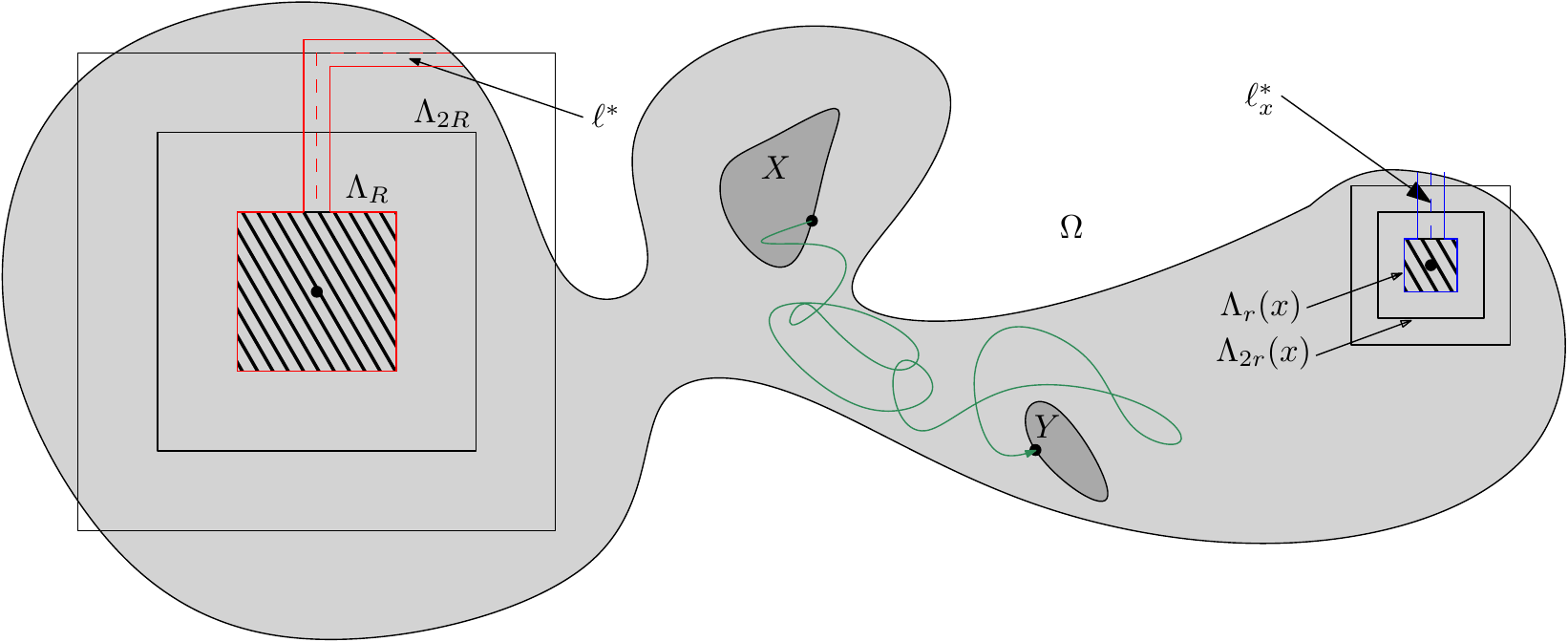}  
		\end{center}
		\caption{The graph $\Omega$ with the two boxes $\Lambda_R$ and $\Lambda_r(x)$. We also depicted an example of possible choice for $\ell_x^*$ and $\ell^*$. In red and blue the wired arcs of the domain $\Omega'$ obtained by removing all the edges enclosed by the red and blue parts. The quantity $Z_\Omega[X,Y]$ is obtained as the sum over every $x\in X$ and $y\in Y$ of the harmonic measure of $y$ seen from $x$ in $\Omega$. 
		}	\label{fig:harmonic}
\end{figure}

The proof will consist in using an estimate from \cite{CheDumHon16} that expresses random cluster crossing probabilities in quads with mixed boundary conditions in terms of the random-walk partition functions above. To use this result, we will create a quad by connecting $\Lambda_r(x)$ and $\Lambda_R$ to the boundary of $\Omega$ in a suitable way.

\begin{proof}It was proved in  \cite[Proposition~4.1]{CheDumHon16} that for every quad $(D,a,b,c,d)$,
\begin{equation}\label{eq:bound harmonic measure easy}
cZ_{D^*}[(ab)^*,(cd)^*]^{1/2}\le\phi_D^{(ab),(cd)}[(ab)\stackrel\omega\longleftrightarrow(cd)]\le CZ_{D}[(ab),(cd)]^{1/2},
\end{equation}
where $(ab)^*$ is the set of vertices on $\partial D^*$ neighboring vertices in $(ab)$, and similarly for $(cd)^*$. While only the upper bound was proved in \cite{CheDumHon16}, as mentioned in the paper, the lower bound can be obtained in a similar fashion.

Let us now explain how we use this estimate in our context. Note that, by the switching lemma~(Lemma~\ref{lem:switching lemma}) and the Edwards-Sokal coupling (Proposition~\ref{prop:coupling ES}), we have that 
\[
{\bf P}_{\Omega^{\bullet\bullet},\Omega^{\bullet\bullet}}^{\emptyset,\emptyset}[\Lambda_r(x)\stackrel{\n_1+\n_2}\longleftrightarrow\Lambda_R]=\langle\sigma_{\Lambda_r(x)}\sigma_{\Lambda_R}\rangle_{\Omega^{\bullet\bullet}}^2=\phi_{\Omega^{\bullet\bullet}}^0[\Lambda_r(x)\stackrel\omega\longleftrightarrow\Lambda_R]^2
\]
so that it suffices to estimate the latter.

We start with the upper bound of \eqref{eq:bound harmonic measure}. We wish to invoke \eqref{eq:bound harmonic measure easy}, but the problem is that $\Lambda_r(x)$ and $\Lambda_R(x)$ are not boundary arcs of $\Omega^{\bullet\bullet}$. Yet, one may create (we leave it as an exercise to the reader) two dual paths $\ell^*$ and $\ell_x^*$ from $\Lambda_R^*$ and $\Lambda_r(x)^*$ to $\partial\Omega^*$ such that
\[
Z_{\Omega'}[\Lambda_r(x)\cup\ell_x,\Lambda_R\cup\ell]\le C_2Z_{\Omega}[x,0],
\]
where 
\begin{itemize}[noitemsep]
\item $\Omega'$ is the domain obtained from $\Omega$ by removing all the edges with both endpoints in $\Lambda_r(x)\cup\Lambda_R$, and edges crossed by $\ell^*$ or $\ell^*_x$, 
\item  $\ell_x$ and $\ell$ are the sets of endpoints of edges crossed by $\ell^*$ and $\ell_x^*$ respectively.
\end{itemize}
If $\{\Lambda_r(x)\cup\ell_x,\Lambda_R\cup\ell\}$ denotes the wired boundary condition on these two sets, and free elsewhere, the comparison between boundary conditions for the random cluster model implies that 
\[
\phi_{\Omega^{\bullet\bullet}}^0[\Lambda_r(x)\stackrel\omega\longleftrightarrow\Lambda_R]\le \phi_{\Omega'}^{\Lambda_r(x)\cup\ell_x,\Lambda_R\cup\ell}[\Lambda_r(x)\cup\ell_x\stackrel\omega\longleftrightarrow\Lambda_R\cup\ell].\]
The upper bound then follows from \eqref{eq:bound harmonic measure easy} (now we are in the right context) and the bound on $Z_{\Omega}$.

For the lower bound, one may construct (again, we leave this as an exercise to the reader) two paths $\ell$ and $\ell_x$ from $\Lambda_R$ and $\Lambda_r(x)$ to $\partial\Omega$ such that
\[
Z_{\Omega^*}[x^*,0^*]\le C_2Z_{(\Omega'')^*}[\Lambda_r(x)^*,\Lambda_R^*],
\]
where $\Omega''$ is the graph obtained by removing the edges in $\ell\cup\ell_x$ and those strictly inside $\Lambda_R\cup\Lambda_r(x)$.
The proof then follows from the monotonicity properties of the random cluster model 
 and \eqref{eq:bound harmonic measure easy}, whose application is now justified since $\Lambda_R$ and $\Lambda_r(x)$ are intersecting the boundary of $\Omega''$.
\end{proof}

\section{Preliminaries for the random current model}\label{sec:preliminaries hugo}

In this section, we gather a number of new results of general interest. These results will be used extensively in the next sections.

\subsection{Mixing property of the random current model}

We start with a ratio weak mixing property that states that the probability of the intersection of events depending on sets of edges that are sufficiently well separated are comparable to the product of the probabilities of each event.  This will be a convenient substitute for the lack of independence in the model. With the same proof, we will also obtain a certain form of independence with respect to boundary conditions (under constraints that cannot be relaxed, such as the parity of the number of sources in a certain area of the quad), and with respect to the geometry of the graph far away. 

Consider a graph $\Omega$ partitioned into three subgraphs $G_0,G,G_1$ satisfying that 
either
\begin{itemize}[noitemsep]
\item[(i)] $G$ is a quad $(D,a,b,c,d)$ of $\mathbb Z^2$ of extremal distance in $(\kappa,\tfrac1\kappa)$ (with $\kappa>0$), and any path from $(bc)$ to $(da)$ disconnects $G_0$ from $G_1$ in $\Omega$;
\item[(ii)] $G$ is an annulus $\mathrm{Ann}(R,2R)$ and $\partial\Lambda_{3R/2}$ disconnects $G_0$ from $G_1$.
\end{itemize}
We insist on the fact that $G_0$ and $G_1$ can be any graph (not necessarily subsets of $\mathbb Z^2$), see Fig.~\ref{fig:mixing}.

\begin{figure} 
		\begin{center}
			\includegraphics[scale=0.9]{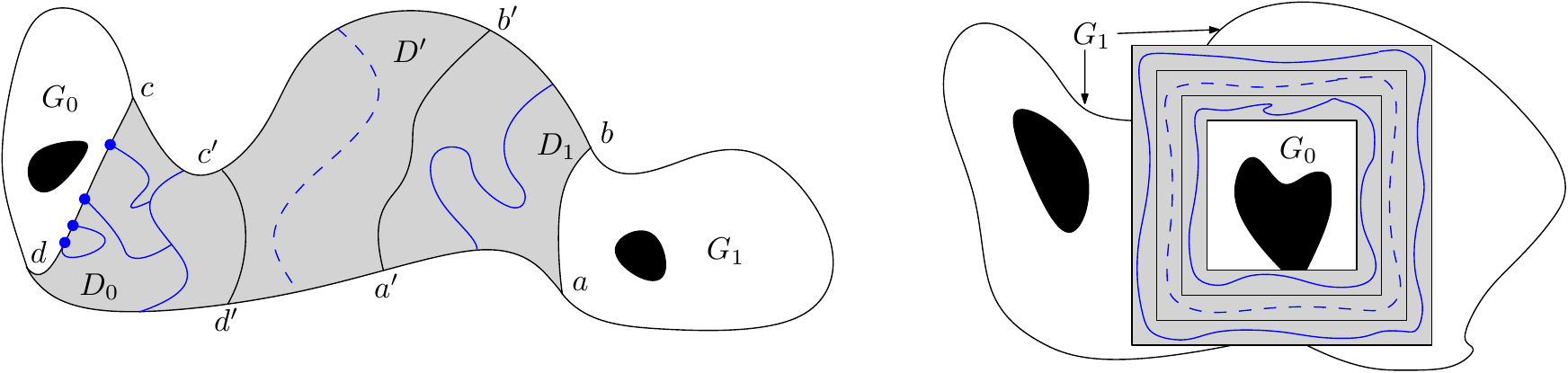}  
		\end{center}
		\caption{The settings (i) on the left and (ii) on the right. The black parts correspond to holes in $G_i$ (note that such holes are forbidden in the gray area). The graphs $G_i$ may not be subsets of the square lattice.  We will in particular use this in the case of graphs obtained from subgraphs of the square lattice by merging vertices together. We also depicted the splitting that is used in the proof of the mixing property, as well as the existence of paths in $\omega$ in blue lines. We depicted the absence of path in $\omega$ using dashed lines.
		}	\label{fig:mixing}
\end{figure}

\begin{proposition}[Mixing of the single random current]\label{lem:mixing}
For every $\kappa>0$, 
there exist $c_{\rm mix}>0$ and $C_{\rm mix}>0$ such that for every graph $\Omega$ satisfying either (i) or (ii), all events $E$ and $F$ depending on edges in $G_0$ and $G_1$ respectively, and any set of sources $A$ in $G_0\cup G_1$,
\begin{align}\label{eq:mixing}
c_{\rm mix}\,{\bf P}^A_\Omega[E]\,{\bf P}^A_\Omega[F]&\le{\bf P}^A_\Omega[E\cap F]\le C_{\rm mix}\,{\bf P}^A_\Omega[E]\,{\bf P}^A_\Omega[F].
\end{align}
Furthermore, if one considers a set $B\subset G_0\cup G_1$ such that $B\cap G_0=A\cap G_0$, then
\begin{align}\label{eq:independence in boundary conditions}
c_{\rm mix}\,{\bf P}^B_\Omega[E]&\le{\bf P}^A_\Omega[E]\le C_{\rm mix}\,{\bf P}^B_\Omega[E],
\end{align}
and if one considers another graph $\Omega'$ that differs from $\Omega$ only in $G_1$, 
\begin{align}\label{eq:independence in graph}
c_{\rm mix}\,{\bf P}^\emptyset_{\Omega'}[E]&\le{\bf P}^\emptyset_\Omega[E]\le C_{\rm mix}\,{\bf P}^\emptyset_{\Omega'}[E].
\end{align}
\end{proposition}

To prove \eqref{eq:mixing} (the other inequalities are obtained in a similar fashion), we condition on the currents in $G_0$ and $G_1$ and express each term in the displayed equations above in terms of spin-spin correlations of the Ising model on $G$. Then, we interpret these spin-spin correlations in terms of the random cluster model using the Edwards-Sokal coupling, and use crossing estimates for the random cluster model to compare the different spin-spin correlations. 

\begin{proof}We prove the first inequality, the others two follow from the same observations. For $i=0,1$, let $\partial_i G$ be the set of vertices in $\partial G$ that are neighbors (in $\Omega$) of a vertex in $G_i$. Define the subgraph $\overline G_i$ of $\Omega$ obtained from $G_i$ by adding the vertices in $\partial_i G$ and edges with one endpoint in $G_i$ and one in $\partial_iG$. Below, the sums over $\n_i$ mean the sum over $\n_i$ on $\overline G_i$ such that $\partial\n_i\cap G_i=A\cap G_i$ and similarly for $\n'_i$. Finally, let $A_i:=\partial\n_i\cap \partial _i G$ and $A_i':=\partial\n_i'\cap \partial _i G$.

We split each current in three parts (one in $\overline G_0$, one in $\overline G_1$, and one on $G$, which is averaged upon to give spin-spin correlations). We obtain 
\begin{align*}
{\bf P}^A_\Omega[E\cap F]&=\frac{\displaystyle\sum_{\n_0,\n_1}w_{\rm RC}(\n_0)w_{\rm RC}(\n_1)\mathbf 1_{\n_0\in E}\mathbf1_{\n_1\in F}\langle\sigma_{A_0}\sigma_{A_1}\rangle_G}{\displaystyle\sum_{\n'_0,\n'_1}w_{\rm RC}(\n'_0)w_{\rm RC}(\n'_1)\langle\sigma_{A'_0}\sigma_{A'_1}\rangle_G},\\
{\bf P}^A_\Omega[E]&=\frac{\displaystyle\sum_{\n_0,\n_1'}w_{\rm RC}(\n_0)w_{\rm RC}(\n'_1)\mathbf 1_{\n_0\in E}\langle\sigma_{A_0}\sigma_{A'_1}\rangle_G}{\displaystyle\sum_{\n'_0,\n'_1}w_{\rm RC}(\n'_0)w_{\rm RC}(\n'_1)\langle\sigma_{A'_0}\sigma_{A'_1}\rangle_G},\\
{\bf P}^A_\Omega[F]&=\frac{\displaystyle\sum_{\n_0',\n_1}w_{\rm RC}(\n'_0)w_{\rm RC}(\n_1)\mathbf 1_{\n_1\in F}\langle\sigma_{A_0'}\sigma_{A_1}\rangle_G}{\displaystyle\sum_{\n'_0,\n'_1}w_{\rm RC}(\n'_0)w_{\rm RC}(\n'_1)\langle\sigma_{A'_0}\sigma_{A'_1}\rangle_G}.
\end{align*}
To conclude the proof, it suffices to show that as soon as $|A_0|$, $|A'_0|$, $|A_1|$, and $|A'_1|$  have the same parity, we get that
\begin{equation}\label{eq:hug1}
\langle\sigma_{A_0}\sigma_{A_1}\rangle_G\langle\sigma_{A_0'}\sigma_{A_1'}\rangle_G\ge c_{\rm mix}\langle\sigma_{A_0}\sigma_{A_1'}\rangle_G\langle\sigma_{A_0'}\sigma_{A_1}\rangle_G
\end{equation}
(the upper bound follows by symmetry). 

We treat the case (ii) first, i.e.~ the case of $G={\rm Ann}(R,2R)$ (see Fig.~\ref{fig:mixing}) and then explain how to solve the case of quads. Use the Edwards-Sokal coupling (Proposition~\ref{prop:coupling ES}) to rephrase these quantities in terms of the random cluster model. For instance, the left-most term becomes
\[
\langle\sigma_{A_0}\sigma_{A_1}\rangle_{\mathrm{Ann}(R,2R)}=\phi_{\mathrm{Ann}(R,2R)}^0[\mathcal F_{A_0\cup A_1}].
\]
We start with the case of $|A_0|$ and $|A_1|$ of even cardinality. Let $E_0$ be the event that there exists a circuit in $\omega_{|\Lambda_{4R/3}}$ surrounding $\Lambda_R$, and $E_1$ be the event that there exists a circuit in $\omega_{|\Lambda_{2R}}$ surrounding $\Lambda_{5R/3}$. Also, let $E^*$ be the event that there does not exist any path in $\omega$ between $\partial\Lambda_{4R/3}$ and $\partial\Lambda_{5R/3}$. On the one hand, the FKG inequality \eqref{eq:FKG} and the inclusion of events imply that 
\begin{equation}\label{eq:hug2}
\phi_{\mathrm{Ann}(R,2R)}^0[\mathcal F_{A_0\cup A_1}]\ge \phi_{\mathrm{Ann}(R,2R)}^0[\mathcal F_{A_0}]\phi_{\mathrm{Ann}(R,2R)}^0[\mathcal F_{A_1}].
\end{equation}
On the other hand, the comparison between boundary conditions and \eqref{eq:RSW} imply that
\begin{align}
\phi_{\mathrm{Ann}(R,2R)}^0[\mathcal F_{A_0\cup A_1}]&\le\phi_{\mathrm{Ann}(R,2R)}^0[\mathcal F_{A_0\cup A_1}| E_0\cap E^*\cap E_1]\nonumber\\
&\le C_0\phi_{\mathrm{Ann}(R,2R)}^0[\mathcal F_{A_0\cup A_1}\cap E_0\cap E^*\cap E_1]\nonumber\\
&=C_0\phi_{\mathrm{Ann}(R,2R)}^0[\mathcal F_{A_0}\cap E_0\cap E^*\cap\mathcal F_{A_1}\cap E_1]\nonumber\\
&\le C_0\phi_{\mathrm{Ann}(R,2R)}^0[\mathcal F_{A_0}\cap E_0|E^*\cap\mathcal F_{A_1}\cap E_1]\phi_{\mathrm{Ann}(R,2R)}^0[\mathcal F_{A_1}\cap E_1|E^*].\label{eq:on1}
\end{align}
Now, let $\pmb\Omega$ be the set of vertices in ${\rm Ann}(R,2R)$ that are not connected to the complement of $\Lambda_{5R/3}$. We deduce from the spatial Markov property, the comparison between boundary conditions, and the fact that $\mathcal F_{A_0}\cap E_0$ is increasing, that for every $\Omega$ for which $\{\pmb\Omega=\Omega\}\cap E^*\cap \mathcal F_{A_1}\cap E_1$ is non-empty,
\[
\phi^0_{\mathrm{Ann}(R,2R)}[\mathcal F_{A_0}\cap E_0|\{\pmb\Omega=\Omega\}\cap E^*\cap \mathcal F_{A_1}\cap E_1]=\phi^0_{\Omega}[\mathcal F_{A_0}\cap E_0]\le \phi^0_{{\rm Ann}(R,2R)}[\mathcal F_{A_0}\cap E_0].
\]
Summing over all those $\Omega$  and using the inclusion of events gives
\begin{equation}\label{eq:on2}
\phi^0_{\mathrm{Ann}(R,2R)}[\mathcal F_{A_0}\cap E_0| E^*\cap \mathcal F_{A_1}\cap E_1]\le \phi^0_{{\rm Ann}(R,2R)}[\mathcal F_{A_0}\cap E_0]\le \phi^0_{{\rm Ann}(R,2R)}[\mathcal F_{A_0}].
\end{equation}
Similarly, one gets
\begin{equation}\label{eq:on3}
\phi^0_{\mathrm{Ann}(R,2R)}[\mathcal F_{A_1}\cap E_1| E^*]\le \phi^0_{{\rm Ann}(R,2R)}[\mathcal F_{A_1}].
\end{equation}
 Together, \eqref{eq:hug2}--\eqref{eq:on3} imply  that 
 \begin{align*}
\phi_{\mathrm{Ann}(R,2R)}^0[\mathcal F_{A_0}]\phi_{\mathrm{Ann}(R,2R)}^0[\mathcal F_{A_1}]\le \phi_{\mathrm{Ann}(R,2R)}^0[\mathcal F_{A_0\cup A_1}]
 &\le C_0\phi^0_{\mathrm{Ann}(R,2R)}[\mathcal F_{A_0}]\phi^0_{\mathrm{Ann}(R,2R)}[\mathcal F_{A_1}].
\end{align*}
Plugging this factorization property (for $A_0,A_1,A'_0,A'_1$) concludes the proof of \eqref{eq:hug1} when $|A_0|$ and $|A_1|$ are even.

When $A_0$ and $A_1$ are odd, note that $\mathcal F_{A_0\cup A_1}\cap E_0\cap E_1$ is equal to the intersection of 
\begin{itemize}
\item the event $\mathcal F'_{A_0}$ that $E_0$ occurs and every cluster except the cluster of the inner-most circuit in $\Lambda_{4R/3}$ surrounding $\Lambda_R$ intersects $A_0$ an even number of times, 
\item the event $\mathcal F'_{A_1}$ that $E_1$ occurs and every cluster except the cluster of the outer-most circuit in $\Lambda_{2R}$ surrounding $\Lambda_{5R/3}$ intersects $A_1$ an even number of times,
\item the event $F$ that $\Lambda_R$ is connected to $\partial\Lambda_{2R}$.
\end{itemize}
Using this observation, a sequence of inequalities similar to the previous one concludes the proof.

Finally, when we are in the case $(i)$, one may split the quad $(D,a,b,c,d)$ into three quads $(D_0,c',c,d,d')$, $(D',a',b',c',d')$, and $(D_1,a',a,b',b)$ (see Fig.~\ref{fig:mixing}) in such a way that the quads have extremal length in $(\kappa',1/\kappa')$ with $\kappa'=\kappa'(\kappa)>0$. Then, one sets $E_0$ to be the event that $D_0$ is crossed from $(cc')$ to $(dd')$, $E_1$ to be the event that $D_1$ is crossed from $(aa')$ to $(bb')$, and $E^*$ the event that $(a'b')$ is not connected to $(c'd')$. The rest of the proof is the same.
\end{proof}

\subsection{Monotonicity properties of the double random current}

One of the most important properties of the Ising and random cluster models is that they are positively associated. This property is at the core of most arguments dealing with these models, as it conveniently leads to monotonicity of the averages of certain ``increasing'' observables, and to the classical FKG inequality. Unfortunately, the double random current does not satisfy the positive association. Nevertheless, for certain connection probabilities, it still enjoys some sort of monotonicity. Below, we collect some examples of these specific monotonicity properties.

We start with the monotonicity of connectivity properties with respect to coupling constants (see Fig.~\ref{fig:monotonicity}).
For the next lemma, we, for once, speak of the Ising model on $G$ with non-negative coupling constants $J:=(J_{x,y}:\{x,y\}\subset E)\in \mathbb R_+^{E}$ defined like the nearest neighbor Ising model, but with Hamiltonian 
\[
H_{G,J}(\sigma)=-\sum_{\{x,y\}\in E}J_{x,y}\sigma_x\sigma_y.
\]
We omit the dependency on $J$ in the notation below but one should remember that we consider this more general framework for the next lemma.
 \begin{lemma}[Monotonicity in coupling constants]\label{lem:monotonicity}
Consider two graphs $G$ and $G'$ and two sets $U$ and $V$. The quantity
\[
\mathbf P^{\emptyset,\emptyset}_{G,G'}[U\stackrel{\n_1+\n_2}\longleftrightarrow V]
\]
is increasing in the coupling constants of edges on $G$ (resp.~on $G'$) that are disconnected in $G$ (resp.~in $G'$) from $V$ by $U$ (by which we mean that any path from an endpoint of the edge to $V$ must intersect $U$).\end{lemma}

By symmetry it is also increasing in coupling constants that are disconnected in $G$ from $U$ by $V$.

The proof consists in conditioning on the union $C$ of all the clusters intersecting $V$ in $\n_1+\n_2$ and then splitting the sum using the fact that all the edges that have exactly one endpoint in $C$ have zero current in $\n_1$ and $\n_2$. Then, the monotonicity follows from the monotonicity of certain ratios of partition functions. 
\begin{figure} 
		\begin{center}
			\includegraphics[scale=1]{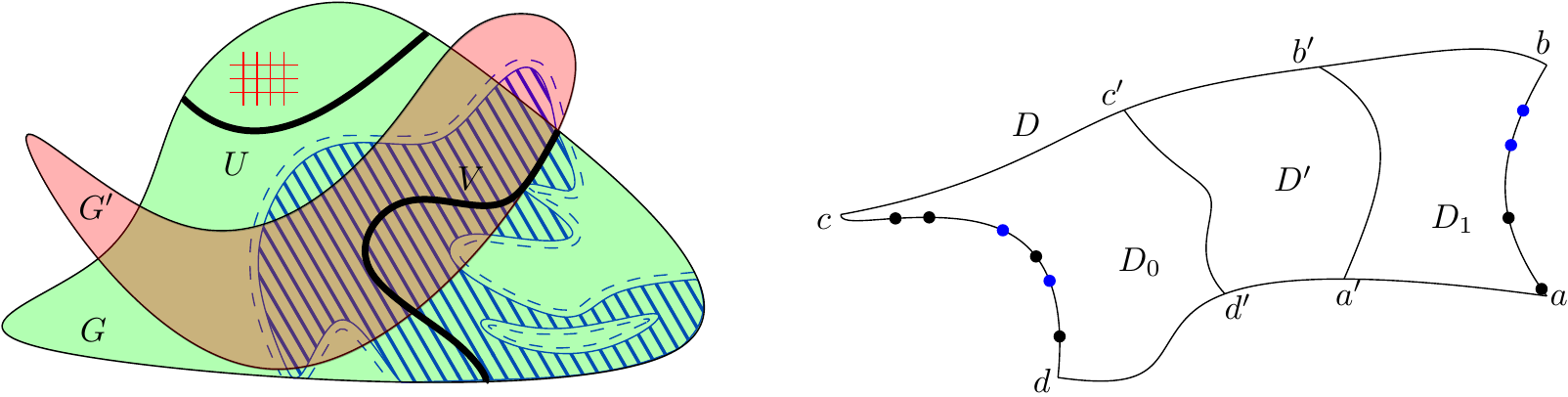}  
		\end{center}
		\caption{On the left, the two graphs $G$ and $G'$, as well as the sets $U$ and $V$. The red grid depicts the edges that are disconnected from $V$ by $U$. The blue part is the graph $C$ composed of the clusters in $\n_1+\n_2$ intersecting $V$. On the right, the quad $(D,a,b,c,d)$ with the quads $(D',a',b',c',d')$, $(D_0,c,d,d',c')$, and $(D_1,a,b,b',a')$. The two last quads are introduced to be able to use the mixing property and erase sources. The quad $(D',a',b',c',d')$ is used to bound from above the crossing probability. The black and blue dots denote the sources of the two currents. Note that there is an even number of such sources on both sides and for both currents (otherwise there must deterministically be a path from one side to the other).
		}	\label{fig:monotonicity}
\end{figure}
\begin{proof}
Let $\mathsf C$ be the clusters (in $\n_1+\n_2$) intersecting $V$. For a subset $C$ not intersecting $U$, let $\overline C$ and $\overline C'$ be the set of vertices of $G$ and $G'$ that are either in $C$ or neighboring a vertex in $C$. Since the two currents vanish on edges with one endpoint in $\mathsf C$ and one outside, we deduce the following factorization relation:
\begin{align*}
\mathbf P^{\emptyset,\emptyset}_{G,G'}[\mathsf C=C]=\frac{Z_{\overline C,\overline C'}^{\emptyset,\emptyset}[\mathsf C=C] Z_{G\setminus C}^\emptyset Z_{G'\setminus C}^\emptyset}{Z_G^\emptyset Z_{G'}^\emptyset},
\end{align*}
where $Z_{\overline C,\overline C'}^{\emptyset,\emptyset}[\mathsf C=C] $ denotes the sum of $w_{\rm RC}(\n_1)w_{\rm RC}(\n_2)\mathbb I[\mathsf C(\n_1+\n_2)=C]$ for $\n_1$ on $\overline C$ and $\n_2$ on $\overline C'$, and $G\setminus C$ denotes the subgraph of $G$ induced by edges with both endpoints outside of $C$.

Note that $Z_{\overline C,\overline C'}^{\emptyset,\emptyset}[\mathsf C=C] $ does not depend on the coupling constants for edges that are disconnected in $G$ from $V$ by $U$. As a consequence, the dependency in the coupling constants that we are interested in is encapsulated in the quantities $Z_{G\setminus C}^\emptyset/Z_G^\emptyset$ and $Z_{G\setminus C'
}^\emptyset/Z_G^\emptyset$, that we need to prove are decreasing in the corresponding coupling constants. 

We prove the result for the former quantity. First, $Z_{G\setminus C}^\emptyset$ can be interpreted as $2^{-|G\setminus C|}$ times the partition function of the Ising model on $G\setminus C$, or equivalently $2^{-|G|}$ times the partition function of the Ising model on $G$ with coupling constants equal to 0 on edges in $\overline C$. Yet, the latter is also the average, on the standard Ising model on $G\setminus C$, of the function defined for $\sigma\in\{\pm1\}^{G\setminus C}$ by
\[
S_C(\sigma):=\sum_{\tau\in\{\pm1\}^{\overline C}:\tau_{|\overline C\setminus C}=\sigma_{|\overline C\setminus C}}e^{-H_{\overline C}(\tau)}\qquad\text{where}\qquad 
H_{\overline C}(\tau):=-\sum_{x,y\in\overline C:\{x,y\}\in E}J_{xy}\tau_x\tau_y.
\]
 As a consequence, we end up with the equality
\[
\frac{Z_G^\emptyset}{Z_{G\setminus C}^\emptyset}=\langle S_C\rangle_{G\setminus C},
\]
where  $\langle\cdot\rangle_{G\setminus C}$ is a slight abuse of notation here and denotes the measure on $G$ with coupling constants equal to 0 on the edges of $\overline C$. We see that the average of $S_C$
 is increasing in coupling constants by the second Griffiths' inequality since 
 \[
S_C=\sum_{x_1,\dots,x_n\subset C}\lambda(x_1,\dots,x_n)\sigma_{x_1}\cdots\sigma_{x_n}
\] (with $\lambda(\cdot)\ge0$) is a positive combination of products of spins, thus concluding the proof. \end{proof}

The second monotonicity property we are interested in deals with the adjunction of sources.

\begin{lemma}[Monotonicity in sources]\label{lem:monotonicity in sources}
Consider two graphs $G$ and $G'$. For every $U,V\subset G$, every $A\subset V$, and every $B$ arbitrary,
\[
\mathbf P^{A,B}_{G,G'}[U\stackrel{\n_1+\n_2}\longleftrightarrow V]\ge\mathbf P^{\emptyset,B}_{G,G'}[U\stackrel{\n_1+\n_2}\longleftrightarrow V]. \]
\end{lemma}

\begin{proof}
We proceed as in the previous proof (we reuse the notation $S_C$) and end up with
\[
\frac{Z_G^A}{Z_{G\setminus C}^A}=2^{|C|}\frac{\langle \sigma_AS_C\rangle_{G\setminus C}}{\langle \sigma_A\rangle_{G\setminus C}}\ge2^{|C|}\langle S_C\rangle_{G\setminus C}=\frac{Z_G^\emptyset}{Z_{G\setminus C}^\emptyset},
\]
where the inequality is due to the second Griffiths' inequality. Then one retransforms the quantity using the same transformation as in the previous proof.
\end{proof}

We deduce from these two monotonicity properties two useful corollaries.

\begin{corollary}\label{cor:upper bound crossing}
There exists $C>0$ such that for every quad $(D,a,b,c,d)$, 
\[
\mathbf P^{\emptyset,\emptyset}_{D,D}[(ab)\stackrel{\n_1+\n_2}\longleftrightarrow (cd)]\le CZ_{D}[(ab),(cd)].
\]
\end{corollary}

The idea behind the proof will be used repeatedly in the next sections so let us discuss it here in details. We consider a graph $D^{\bullet\bullet}$ that is obtained from $D$ by merging the vertices in $(ab)$ and $(cd)$ together into two ``master'' vertices, that we still call $(ab)$ and $(cd)$. Note that the Ising model on this graph can also be understood as an Ising model on $D$ with coupling constants equal to infinity for edges with both endpoints in $(ab)$, or both endpoints in $(cd)$. In particular, by the monotonicity in coupling constants the connectivity probability between $(ab)$ and $(cd)$ in the double random-current is larger on $D^{\bullet\bullet}$ than on $D$. Now, the advantage of $D^{\bullet\bullet}$ over $D$ is that $(ab)$ and $(cd)$ are single vertices, and that we can therefore use the switching lemma to rewrite the connectivity probabilities in terms of spin-spin correlations (for the Ising model on $D^{\bullet\bullet}$), which can then be estimated using  the Edwards-Sokal coupling and the random cluster model. This trick, that we will refer to as the {\em merging-vertices trick}, of passing to a graph with merged vertices will be used quite often when trying to estimate the probability that two sets are connected to each other. 

We now present the proof (which is shorter than the explanation).
\begin{proof}
By monotonicity in coupling constants (applied twice, once to $\n_1$ and once to $\n_2$), we have that  
\begin{align*}
\mathbf P^{\emptyset,\emptyset}_{D,D}[(ab)\stackrel{\n_1+\n_2}\longleftrightarrow (cd)]&\le \mathbf P^{\emptyset,\emptyset}_{D^{\bullet\bullet},D^{\bullet\bullet}}[(ab)\stackrel{\n_1+\n_2}\longleftrightarrow (cd)]=\langle \sigma_{(ab)}\sigma_{(cd)}\rangle_{D^{\bullet\bullet}}^2,
\end{align*}
where $D^{\bullet\bullet}$ is the graph obtained from $D$  by merging the vertices of $(ab)$ and $(cd)$ into two vertices that we keep denoting $(ab)$ and $(cd)$, and the equality is due to the switching lemma. Then, the Edwards-Sokal coupling and Corollary~\ref{cor:bound harmonic measure} imply the claim. 
\end{proof}

\begin{corollary}[crossing probabilities with arbitrary sources]\label{cor:upper bound crossing close 1}
For every $\kappa>0$, there exists $c=c(\kappa)>0$ such that for every quad $(D,a,b,c,d)$ with extremal distance bounded by $\kappa$ from above, and every set of sources $A$ and $B$ on $(ab)\cup(cd)$ that have an even intersection with $(ab)$ and $(cd)$,
\[
\mathbf P^{A,B}_{D,D}[(ab)\stackrel{\n_1+\n_2}\longleftrightarrow (cd)]\le 1-c.
\]
\end{corollary}

The proof of this corollary consists in splitting the quad $D$ into three quads and in combining the merging-vertices trick to estimate the crossing probability of the mid-section together with the mixing property.
\begin{proof}
We recommend to look at Fig.~\ref{fig:monotonicity}. Divide $D$ into three quads $D_0,D',D_1$ of extremal distance bounded by $\kappa'=\kappa'(\kappa)$ as in the proof of the mixing property. Then, use the inclusion of events in the first inequality, the mixing property \eqref{eq:independence in boundary conditions} (for the complementary event) in the second to get that 
\begin{align*}
\mathbf P^{A,B}_{D,D}[(ab)\stackrel{\n_1+\n_2}\longleftrightarrow (cd)]&\le \mathbf P^{A,B}_{D,D}[(a'b')\stackrel{\n_1+\n_2}\longleftrightarrow(c'd')]\le 1-c_{\rm mix}^2(1-\mathbf P^{\emptyset,\emptyset}_{D,D}[(a'b')\stackrel{\n_1+\n_2}\longleftrightarrow (c'd')]).
\end{align*}
Yet, the monotonicity in coupling constants and the Edwards-Sokal coupling imply that
\begin{align*}
\mathbf P^{\emptyset,\emptyset}_{D,D}[(a'b')\stackrel{\n_1+\n_2}\longleftrightarrow (c'd')]&\le \mathbf P^{\emptyset,\emptyset}_{D^{\bullet\bullet},D^{\bullet\bullet}}[(a'b')\stackrel{\n_1+\n_2}\longleftrightarrow (c'd')]=\phi_{D'}^{(a'b'),(c'd')}[(a'b')\stackrel{\omega}\longleftrightarrow (c'd')]^2,
\end{align*}
where $D^{\bullet\bullet}$ is the graph obtained from $D$ by merging all the vertices in $D_0$ into one vertex, and all those in $D_1$ into another one.

It remains to observe that since $D'$ has extremal distance smaller than $\kappa'$, the previous corollary implies that
\[
\phi_{D'}^{(a'b'),(c'd')}[(a'b')\stackrel{\omega}\longleftrightarrow (c'd')]\le 1-c(\kappa').
\]
The last three displayed equations imply the claim.\end{proof}

\section{Bounds on crossing probabilities for the double random-current: proof of Theorem~\ref{thm:crossing free}}\label{sec:crossing estimate}

In this section, we investigate crossing probabilities for the double random current in quads and prove Theorem~\ref{thm:crossing free}.  We split the section in two.  In the next subsection, we first study the expected number of boxes near the boundary that are connected to $\Lambda_R$. Then, we prove the theorem in the following subsection. Finally, the last subsection is a proof of a similar result which will be useful in next sections.
Below, for a domain  $\Omega$ containing $\Lambda_{R}$, let $\Omega^\bullet$ be the graph obtained by merging all the vertices of $\Lambda_R$ together. We identify the new vertex with $\Lambda_R$.

\subsection{On the expected number of boxes near the boundary that are connected to $\Lambda_R$}\label{sec:4.1}

For a box $B$, call $\overline B$ and $\underline B$ the twice bigger and twice smaller boxes centred on the same vertex. For a domain $\Omega$, let $\Omega^\square_r$ be the connected component containing $\Lambda_R$ in the union of all the boxes $B=\Lambda_r(x)$ with $x\in r\mathbb Z^2$ such that $\Lambda_{2r}(x)
\subset \Omega$. Also, let
\[
\partial_r^\square\Omega:=\{B=\Lambda_r(x)\text{ with $x\in r\mathbb Z^2$ such that $\Lambda_{r}(x)
\subset \Omega^\square_r$ and $\Lambda_{3r}(x)\not\subset \Omega$}\}.
\] 
The reader should be aware that $\Omega^\square_r$ is a subset of $\Omega$ while $\partial_r^\square\Omega$ is a set of boxes of size $r$.
At this stage, one may wonder why we consider only boxes with $\Lambda_{r}(x)\subset\Omega^\square_r$ and not simply $\Lambda_{r}(x)\subset\Omega$. The reason comes from Lemma~\ref{lem:1a}, see Remark~\ref{rmk:1a} below it.

For $\ep>0$, introduce the random variables (see Fig.~\ref{fig:definitionobservables}):
\begin{align*}
N&:=\sum_{B\in\partial_r^\square\Omega}\mathbb I[B\stackrel{\n_1+\n_2}\longleftrightarrow \Lambda_R],\\
\overline N&:=\sum_{B\in\partial_r^\square\Omega}\mathbb I[\overline B\stackrel{\n_1+\n_2}\longleftrightarrow \Lambda_R],\\
\underline N^\circ&:=\sum_{B\in\partial_r^\square\Omega}\mathbb I[\underline B\stackrel{\n_1+\n_2}\longleftrightarrow \Lambda_R,\mathcal A(B)],\\
\overline N_\ep&:=\sum_{B\in\partial_r^\square\Omega}\mathbb I[\mathcal E_\ep(\overline B)],\\
N_\ep^\circ&:=\sum_{B\in\partial_r^\square\Omega}\mathbb I[\mathcal E_\ep(B)\cap\mathcal A(B)],
\end{align*}
where
\begin{align}
\label{eq:def A}\mathcal A(B)&:=\{\text{there exists a circuit of $\n_1+\n_2>0$ in $\overline B$  surrounding $B$}\},\\
\label{eq:def E}\mathcal E_\ep(B)&:= 
\{{\bf P}^{\emptyset,\emptyset}_{\Omega^\bullet,\Omega^\bullet}[\underline B\stackrel{\n_1+\n_2}\longleftrightarrow\Lambda_R|(\n_1+\n_2)_{|B^c}]\ge \ep\}.
\end{align}
\begin{figure} 
		\begin{center}
			\includegraphics[scale=0.8]{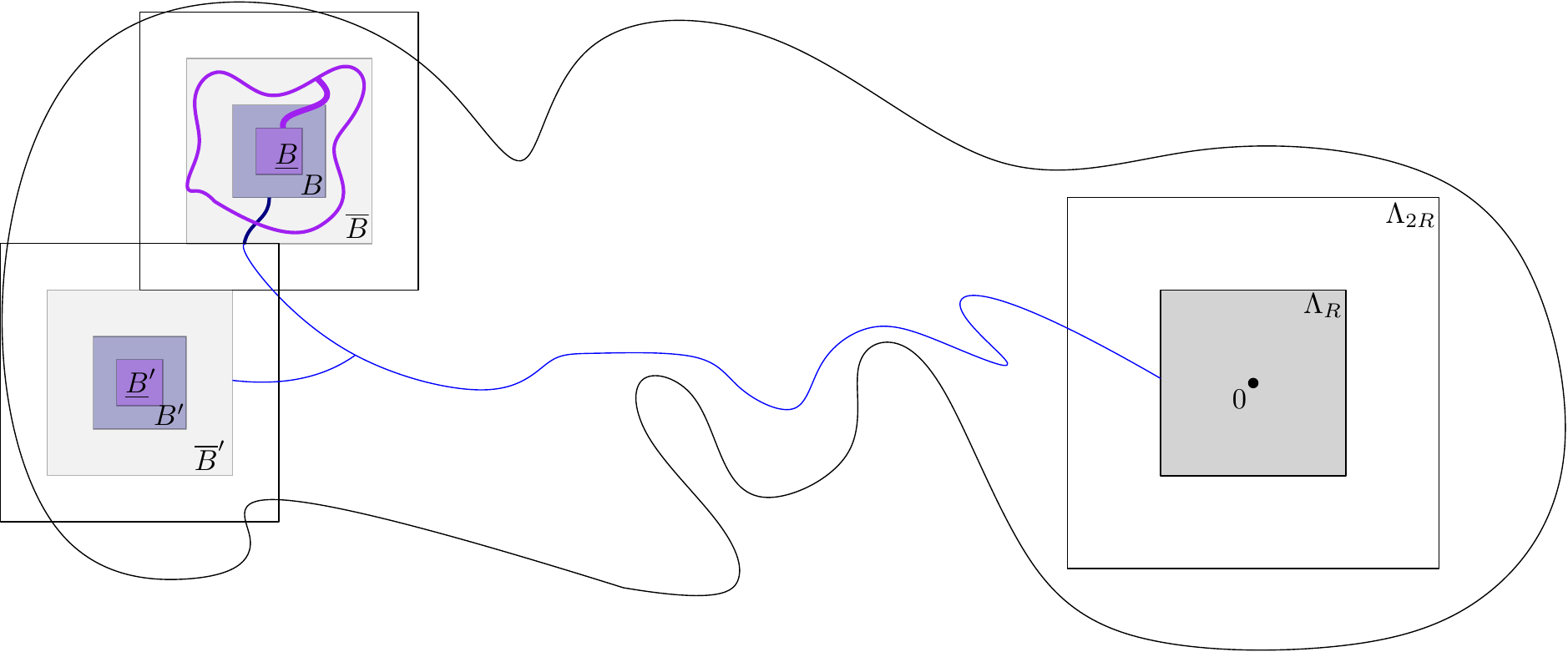}  
		\end{center}
		\caption{A picture of the domain $\Omega$ with two examples of boxes in $\partial_r^\square\Omega$. In blue, the path in $\n_1+\n_2$ needed to be counted in $\overline N$. In darker blue, the additional event needed to be counted in $N$, in violet the event needed to be counted in $\underline N^\circ$. Finally, a box is counted in $N_\ep^\circ$ if the blue and dark blue paths are occurring, and if conditioned on everything outside $B$, the bold violet occurs with probability larger than $\ep$. 
		}	\label{fig:definitionobservables}
\end{figure}

\begin{remark}\label{rmk:p}
The introduction of the previous random variables is technical but important. They satisfy the following features:
\begin{itemize}[noitemsep]
\item[(i)]
$N_\ep^\circ\le N\le \overline N$ and $\overline N_\ep\le \overline N$, 
\item[(ii)] $\mathcal E_\ep(B)$ does not depend on what happens inside $B$ and is included in $B\stackrel{\n_1+\n_2}\longleftrightarrow \Lambda_R$,
\item[(iii)] if $\mathcal A(B)$ occurs, then connections between vertices outside of $\overline B$ are not impacted by what happens in $B$,
\item[(iv)] conditioned on  $\mathcal E_\ep(B)$, the probability of $\underline B$ being connected to $\Lambda_R$ in $\n_1+\n_2$ is larger than $\ep$. Note for future reference that this is also true in the graph $\Omega^{\bullet\bullet}$ where vertices of $\underline B$ and $\Lambda_R$ are merged into two vertices\footnote{Indeed, fix the realizations of $\n_1$ and $\n_2$ outside of $B$ and consider the graphs $\underline\Omega^\bullet$ and $\underline\Omega^{\bullet\bullet}$ obtained from $\Omega^\bullet$ and $\Omega^{\bullet\bullet}$ by merging each cluster of $(\n_1+\n_2)_{|B^c}$ into a single vertex. The previous statement follows from the inequality 
\[
{\bf P}^{\emptyset,\emptyset}_{\underline\Omega^{\bullet\bullet},\underline\Omega^{\bullet\bullet}
}[\underline B\stackrel{\n_1+\n_2}\longleftrightarrow\Lambda_R]\ge {\bf P}^{\emptyset,\emptyset}_{\underline\Omega^\bullet,\underline\Omega^\bullet}[\underline B\stackrel{\n_1+\n_2}\longleftrightarrow\Lambda_R],
\]
which is a consequence of the monotonicity in coupling constants (Lemma~\ref{lem:monotonicity}).
}.
\end{itemize} 
\end{remark}

We start the proof of Theorem~\ref{thm:crossing free} with a lemma dealing with the expectations of the previously defined random variables. 

\begin{lemma}[Expectations of order 1]\label{lem:expected number}
There exist $\ep,C,c>0$ such that for all $r,R$ with $1\le r\le R$ and every domain $\Omega\supset\Lambda_{2R}$, 
\begin{align}\label{eq:expected number}
c\le \tfrac12\,{\bf E}^{\emptyset,\emptyset}_{\Omega^\bullet,\Omega^\bullet}[\underline N^\circ]\le {\bf E}^{\emptyset,\emptyset}_{\Omega^\bullet,\Omega^\bullet}[N_\ep^\circ]\le {\bf E}^{\emptyset,\emptyset}_{\Omega^\bullet,\Omega^\bullet}[N]\le {\bf E}^{\emptyset,\emptyset}_{\Omega^\bullet,\Omega^\bullet}[\overline N]&\le C.
\end{align}
\end{lemma}

For future reference, this lemma implies the same bounds for $B$ instead of $ \overline B$ in the definitions of the observables.
Note that the third and fourth inequalities are trivial by (i) of Remark~\ref{rmk:p}. The fifth one (i.e.~the right-most) is a direct consequence of the merging-vertices trick and bounds in terms of random-walk partition functions, see below. The first one (i.e.~the left-most) is also a consequence of the merging-vertices trick and bounds in terms of random-walk partition functions, but this time combined with the mixing property (in order to use this mixing property one merges vertices of the box $\underline{\underline B}$ which is twice smaller than $\underline B$). The second inequality is a combination of the previous ones, using the definition of $N_\ep^\circ$ in terms of the conditional probability of creating a connection.

\begin{proof}
As mentioned above, the third and fourth inequalities are trivial by (i) of Remark~\ref{rmk:p}. 
We now focus on the right-most inequality. 
If $\Omega^{\bullet\blacksquare}$ denotes the graph obtained from $\Omega^\bullet$ by merging the vertices of $\overline B$ together, the monotonicity in coupling constants (Lemma~\ref{lem:monotonicity}) and Corollary~\ref{cor:bound harmonic measure} give that 
\begin{align*}
{\bf E}^{\emptyset,\emptyset}_{\Omega^\bullet,\Omega^\bullet}[\overline N]&=\sum_{B\in \partial_r^\square\Omega}{\bf P}^{\emptyset,\emptyset}_{\Omega^\bullet,\Omega^\bullet}[\overline B\stackrel{\n_1+\n_2}\longleftrightarrow \Lambda_R]\\
&\le \sum_{B\in \partial_r^\square\Omega}{\bf P}^{\emptyset,\emptyset}_{\Omega^{\bullet\blacksquare},\Omega^{\bullet\blacksquare}}[\overline B\stackrel{\n_1+\n_2}\longleftrightarrow \Lambda_R]\\
&\le C_0\sum_{x:\Lambda_r(x)\in \partial_r^\square\Omega}Z_{\Omega}[0,x]\le C_1,
\end{align*}
where the last bound follows from standard random walk estimates.

\begin{figure} 
		\begin{center}
			\includegraphics[scale=0.93]{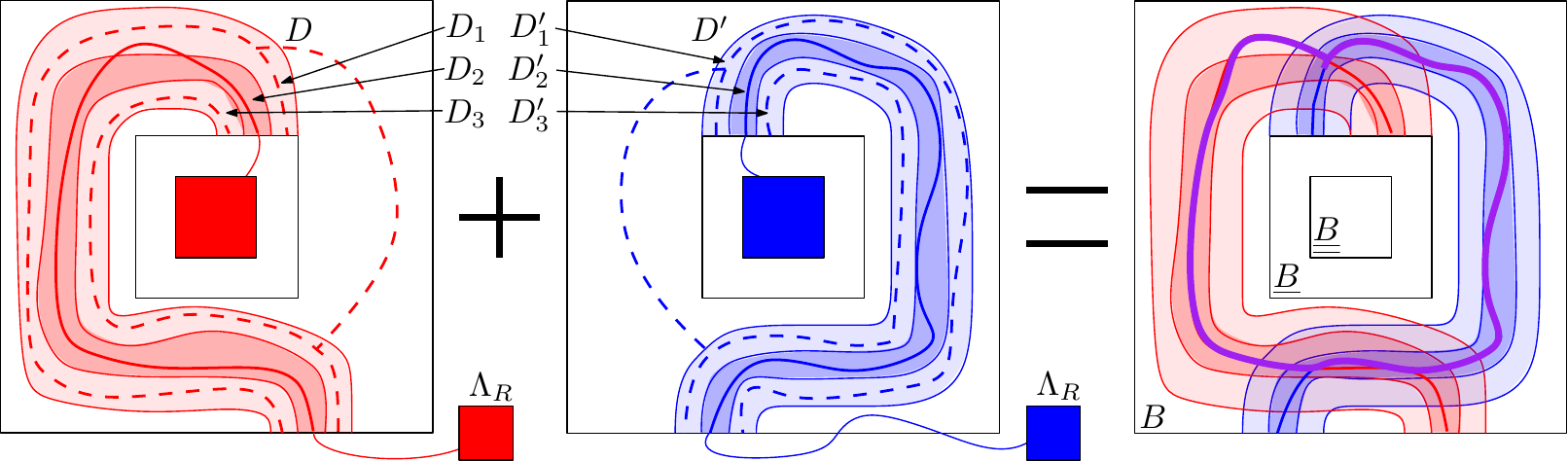}  
		\end{center}
		\caption{On the left, the domain $D=D_1\cup D_2\cup D_3$. We require that $D_1$ and $D_3$ are not crossed between their sides strictly inside $B\setminus \underline B$ by a path of $\n_1>0$ (which is depicted by a dashed path from inside to outside corresponding to a dual path crossing only edges with $\n_1$-current equal to zero). Then, we ask that all clusters of $\n_1>0$ going from inside to outside in $B\setminus \underline B$ must intersect $D_2$ (it is depicted in dashed again). Finally, the source constraint forces the existence of a primal path from $\underline{\underline B}$ to $\Lambda_R$ going through $D$ (note that it is not necessarily contained in $D_2$).	 In the middle, the corresponding picture for $\n_2$. Finally, the combination of the two currents necessarily includes a loop (in purple) in $B\setminus\underline B$ surrounding $\underline B$.	}	\label{fig:A(B)}
\end{figure}

We now turn to the left-most inequality. Set $\underline{\underline B}:=\Lambda_{r/4}(x)$ where $x$ is the center of $B$ (it is the box twice smaller than $\underline B$). Let $\Omega^{\bullet{\bigcdot}}$ denote the graph obtained from $\Omega^\bullet$ by merging the vertices  of $\underline{\underline B}$ together. Also, recall the definition of $\mathcal A(B)$. The mixing property \eqref{eq:independence in graph} gives 
\begin{align}\label{eq:kjkk}
{\bf P}^{\emptyset,\emptyset}_{\Omega^{\bullet},\Omega^{\bullet}}[\underline B\stackrel{\n_1+\n_2}\longleftrightarrow \Lambda_R,\mathcal A(B)]&\ge c_{\rm mix} {\bf P}^{\emptyset,\emptyset}_{\Omega^{\bullet{\bigcdot}},\Omega^{\bullet{\bigcdot}}}[\underline B\stackrel{\n_1+\n_2}\longleftrightarrow \Lambda_R,\mathcal A(B)]
\end{align}
(the fact that we need the place where we merged vertices to be well apart from the edges involved in the events under consideration is the main reason for introducing $\underline{\underline B}$). 

Now, the inclusion of events and the switching lemma lead to
\begin{align}
{\bf P}^{\emptyset,\emptyset}_{\Omega^{\bullet{\bigcdot}},\Omega^{\bullet{\bigcdot}}}[\underline B\stackrel{\n_1+\n_2}\longleftrightarrow \Lambda_R,\mathcal A(B)]&\ge {\bf P}^{\emptyset,\emptyset}_{\Omega^{\bullet{\bigcdot}},\Omega^{\bullet{\bigcdot}}}[\underline{\underline B}\stackrel{\n_1+\n_2}\longleftrightarrow \Lambda_R,\mathcal A(B)]\nonumber\\
&=\langle\sigma_{\underline{\underline B}}\sigma_{\Lambda_R}\rangle_{\Omega^{\bullet{\bigcdot}}}^2{\bf P}^{\{\underline{\underline B},\Lambda_R\},\{\underline{\underline B},\Lambda_R\}}_{\Omega^{\bullet{\bigcdot}},\Omega^{\bullet{\bigcdot}}}[\mathcal A(B)].\label{eq:ahj0}
\end{align}
Then, Corollary~\ref{cor:bound harmonic measure} implies that 
\begin{equation}\label{eq:ahj1}
 \langle\sigma_{\underline{\underline B}}\sigma_{\Lambda_R}\rangle_{\Omega^{\bullet{\bigcdot}}}^2\ge c_0 Z_{\Omega^*}[x^*,0^*].
\end{equation}
We also claim that
\begin{equation}\label{eq:ahj2}
{\bf P}^{\{\underline{\underline B},\Lambda_R\},\{\underline{\underline B},\Lambda_R\}}_{\Omega^{\bullet{\bigcdot}},\Omega^{\bullet{\bigcdot}}}[\mathcal A(B)]\ge c_1.
\end{equation}
Indeed (see Fig.~\ref{fig:A(B)} for an illustration), consider two quads $D$ and $D'$ from $\partial B$ to $\partial\overline B$ that have extremal length in $[\kappa,1/\kappa]$ and such that any crossing in $D$ combined with a crossing in $D'$ contains a circuit surrounding $B$ in $\overline B$. Now, let $D_1,D_2,D_3$ (resp.~$D'_1,D'_2,D'_3$) be a partition of $D$ (resp.~$D'$) into three quads connecting $\partial B$ to $\partial\overline B$ with extremal lengths in $[\kappa',1/\kappa']$. Then, for the first current, we can force
\begin{itemize}[noitemsep]
\item the absence of a path in $\n_1$ disconnecting $\partial B$ from $\partial\overline B$ in $D_1$, 
\item the absence of a path in $\n_1$ disconnecting $\partial B$ from $\partial\overline B$ in $D_3$,
\item the absence of a $(B\setminus \underline B)$-cluster in $\n_1$ from $\partial B$ to $\partial\overline B$ that is not intersecting $D_2$
\end{itemize}
 with positive probability (simply use Corollary~\ref{cor:upper bound crossing close 1} in $D_1$ and $D_3$, and then condition on the cluster of $D_2$ in $B\setminus(D_2\cup  \underline B)$ and apply an argument similar to the one leading to Corollary~\ref{cor:upper bound crossing close 1} in the complement  -- we leave this simple adaptation to the reader). Note that in particular there must be, because of source constraints prescribed by ${\bf P}^{\{\underline{\underline B},\Lambda_R\},\{\underline{\underline B},\Lambda_R\}}_{\Omega^{\bullet{\bigcdot}},\Omega^{\bullet{\bigcdot}}}$, a crossing of $D$ from $\partial B$ to $\partial\overline B$. One can do the same with $\n_2$ and prove that with positive probability there is a crossing of $D'$ from $\partial B$ to $\partial\overline B$. Since $\n_1$ and $\n_2$ are independent, we deduce that with positive probability $c_1$ the event $\mathcal A(B)$ occurs.

 Overall, plugging \eqref{eq:ahj1} and \eqref{eq:ahj2} into \eqref{eq:ahj0}, and summing over $B$ gives 
\begin{align}{\bf E}^{\emptyset,\emptyset}_{\Omega^\bullet,\Omega^\bullet}[\underline N^\circ]&
\ge c_{\rm mix}c_0c_1\sum_{x:\Lambda_r(x)\in\partial_r^\square\Omega}Z_{\Omega^*}[x^*,0^*]\ge c_2Z_{\Omega^*}[0^*,\partial\Omega^*]\ge c_3,
\end{align}
where again the last inequality follows from a random walk estimate.

Finally, the second inequality of \eqref{eq:expected number} follows easily from the other two bounds. Indeed, note that $\underline N^\circ$ is smaller than $N_\ep^\circ$ plus the sum over the boxes $B$ that are connected in $\n_1+\n_2$ to $\Lambda_R$ but such that $\mathcal E_\ep(B)$ does not occur. Yet, by definition the latter is smaller in expectation than $\ep \mathbf E_{\Omega^\bullet,\Omega^\bullet}^{\emptyset,\emptyset}[N]$ by using the spatial Markov property. We get that
\begin{equation}\label{eq:kj}
{\bf E}^{\emptyset,\emptyset}_{\Omega^\bullet,\Omega^\bullet}[\underline N^\circ]\le {\bf E}^{\emptyset,\emptyset}_{\Omega^\bullet,\Omega^\bullet}[N_\ep^\circ]+\ep{\bf E}^{\emptyset,\emptyset}_{\Omega^\bullet,\Omega^\bullet}[N].
\end{equation}
We deduce from the lower and upper bounds on the expectations of $\underline N^\circ$ and $ N$ that for $\ep$ small enough (but independent of everything else) we have that 
\begin{equation}\label{eq:kjk}
{\bf E}^{\emptyset,\emptyset}_{\Omega^\bullet,\Omega^\bullet}[N_\ep^\circ]\ge \tfrac12{\bf E}^{\emptyset,\emptyset}_{\Omega^\bullet,\Omega^\bullet}[\underline N^\circ].
\end{equation}
This concludes the proof.
\end{proof}

\subsection{Proof of Theorem~\ref{thm:crossing free}}\label{sec:theorem crossing free}

We divide the proof between  the lower and upper bounds. We start with the former, which is a second moment estimate on $N_\ep^\circ$. As usual, we use the merging-vertices trick together with random walk estimates. 

\begin{proof}[Proof of the lower bound in Theorem~\ref{thm:crossing free}]
Note that $N_\ep^\circ>0$ implies the existence of $B$ in $\partial_r^\square\Omega$ that is connected to $\Lambda_R$ in $\n_1+\n_2$, which gives a connection between $\Lambda_R$ and $\partial_{4r}\Omega$. It therefore suffices to prove that with probability $c/\log(R/r)$, $N_\ep^\circ>0$ (this proves the result for $4r$ instead of $r$ but this is irrelevant here). To get this, we use a second moment method on $N_\ep^\circ$. 
The previous lemma gives that
${\bf E}^{\emptyset,\emptyset}_{\Omega^\bullet,\Omega^\bullet}[N_\ep^\circ]\ge c_0$ and we therefore
 focus on the second moment. 
 
 Set $\mathcal E_\ep^\circ(B):=\mathcal A(B)\cap\mathcal E_\ep(B)$.
 For $B=\Lambda_r(x),B'=\Lambda_r(x')\in\partial_r^\square\Omega$, let $\Omega^{\bullet\bullet\bullet}$ be the graph obtained from $\Omega^\bullet$ by merging all the vertices of $\underline B$ together, and all those of $\underline B'$ together. As before, we start by using the mixing property \eqref{eq:independence in graph}  together with Remark~\ref{rmk:p}(ii)-(iii) to get
  \begin{align*}
{\bf P}^{\emptyset,\emptyset}_{\Omega^{\bullet},\Omega^{\bullet}}[\mathcal E^\circ_\ep(B)\cap \mathcal E^\circ_\ep(B')]&\le C_{\rm mix}^2{\bf P}^{\emptyset,\emptyset}_{\Omega^{\bullet\bullet\bullet},\Omega^{\bullet\bullet\bullet}}[\mathcal E^\circ_\ep(B)\cap \mathcal E^\circ_\ep(B')].
\end{align*}
 While we used that the event $\mathcal E^\circ_\ep(B)\cap \mathcal E^\circ_\ep(B')$ does not depend on edges in $B$ and $B'$ to invoke mixing, we now wish to use the switching lemma, and would like to have $\underline B$ and $\underline B'$ connected to $\Lambda_R$. This is where we use the parameter $\ep>0$ and the last part of the definition of $\mathcal E^\circ_\ep(B)$. More precisely, conditioning on everything outside $B$ and $B'$, we get from Remark~\ref{rmk:p}(iii) and (iv) that 
   \begin{align*}
{\bf P}^{\emptyset,\emptyset}_{\Omega^{\bullet},\Omega^{\bullet}}[\mathcal E^\circ_\ep(B)\cap \mathcal E^\circ_\ep(B')]&\le  \ep^{-2}{\bf P}^{\emptyset,\emptyset}_{\Omega^{\bullet\bullet\bullet},\Omega^{\bullet\bullet\bullet}}[\underline B,\underline B'\stackrel{\n_1+\n_2}\longleftrightarrow\Lambda_R]
\end{align*}
(notice the appearance of the parameter $\ep>0$).
Now, \cite[Proposition~A.3]{AizDum19} gives that for any graph $G$ and any four vertices $0,x,u,v$ in this graph, 
 \[
 \mathbf P_{G,G}^{\{0\}\Delta\{x\},\emptyset}[u,v\stackrel{\n_1+\n_2}\longleftrightarrow 0]\le \frac{\langle\sigma_0\sigma_u\rangle_G\langle\sigma_u\sigma_v\rangle_G\langle\sigma_v\sigma_x\rangle_G}{\langle\sigma_0\sigma_x\rangle_G}+\frac{\langle\sigma_0\sigma_v\rangle_G\langle\sigma_v\sigma_u\rangle_G\langle\sigma_u\sigma_x\rangle_G}{\langle\sigma_0\sigma_x\rangle_G}.
 \]
 Applied to $G=\Omega^{\bullet\bullet\bullet}$, $0=x=\Lambda_R$, $u=\underline B$, and $v=\underline B'$, we get that
   \begin{align*}
{\bf P}^{\emptyset,\emptyset}_{\Omega^{\bullet},\Omega^{\bullet}}[\mathcal E^\circ_\ep(B)\cap \mathcal E^\circ_\ep(B')]&\le  2\ep^{-2}\langle \sigma_{\Lambda_R}\sigma_{\underline B}\rangle_{\Omega
^{\bullet\bullet\bullet}}
\langle \sigma_{\underline B}\sigma_{\underline B'}\rangle_{\Omega
^{\bullet\bullet\bullet}}\langle \sigma_{\underline B'}\sigma_{\Lambda_R}\rangle_{\Omega
^{\bullet\bullet\bullet}}.
\end{align*}
An application of the Edwards-Sokal coupling with the random cluster model  (simply use \eqref{eq:RSW} combined with FKG to create an open circuit around the third vertex) implies that 
 \begin{align*}
{\bf P}^{\emptyset,\emptyset}_{\Omega^{\bullet},\Omega^{\bullet}}[\mathcal E^\circ_\ep(B)\cap \mathcal E^\circ_\ep(B')]&\le  C_0(\ep) \langle \sigma_{\Lambda_R}\sigma_{\underline B}\rangle_{\Omega_{\underline B'}
^{\bullet\bullet}}
\langle \sigma_{\underline B}\sigma_{\underline B'}\rangle_{\Omega
^{\bullet\bullet}_{\Lambda_R}}\langle \sigma_{\underline B'}\sigma_{\Lambda_R}\rangle_{\Omega_{\underline B}
^{\bullet\bullet}},
\end{align*}
where $\Omega_\#^{\bullet\bullet}$ is the graph with two of the three sets $\underline B$, $\underline B'$, and $\Lambda_R$ collapsed into single vertices, the set which is not collapsed to a single vertex being the one indicated in $\#$.

Finally, Corollary~\ref{cor:bound harmonic measure} gives that
 \begin{align*}
{\bf P}^{\emptyset,\emptyset}_{\Omega^{\bullet},\Omega^{\bullet}}[\mathcal E^\circ_\ep(B)\cap \mathcal E^\circ_\ep(B')]&\le C_1(\ep)[Z_\Omega(0,x)Z_\Omega(x,x')Z_\Omega(x',0)]^{1/2}.
\end{align*}
Now, a fairly simple random walk estimate using that $\Omega$ contains $\Lambda_{2R}$ but not $\Lambda_{3R}$ implies that the sum over $B$ and $B'$ satisfies
\[
{\bf E}^{\emptyset,\emptyset}_{\Omega^{\bullet},\Omega^{\bullet}}[(N_\ep^\circ)^2]=\sum_{B,B'\in \partial_r^\square\Omega}{\bf P}^{\emptyset,\emptyset}_{\Omega^{\bullet},\Omega^{\bullet}}[\mathcal E^\circ_\ep(B)\cap \mathcal E^\circ_\ep(B')]\le C_2(\ep)\log(R/r).
\]
Overall, we deduce from the Cauchy-Schwarz inequality that 
\[
{\bf P}^{\emptyset,\emptyset}_{\Omega^{\bullet},\Omega^{\bullet}}[N_\ep^\circ>0]\ge \frac{{\bf E}^{\emptyset,\emptyset}_{\Omega^{\bullet},\Omega^{\bullet}}[N_\ep^\circ]^2}{{\bf E}^{\emptyset,\emptyset}_{\Omega^{\bullet},\Omega^{\bullet}}[(N_\ep^\circ)^2]}\ge \frac{c_3(\ep)}{\log(R/r)}.
\]
This concludes the proof.\end{proof}

We now turn to the proof of the upper bound in Theorem~\ref{thm:crossing free}. The proof is based on the following idea: we already know that the expected number of boxes in $\partial_r^\square\Omega$ that are connected to $\Lambda_R$ in $\n_1+\n_2$ is uniformly bounded. Therefore, it suffices to show that the probability that a box in $\partial_r^\square\Omega$ is connected to $\Lambda_R$, but that there are only few other boxes in $\partial_r^\square\Omega$ that are connected to $\Lambda_R$, is much smaller than the probability that the box is connected to $\Lambda_R$ in $\n_1+\n_2$. In order to do that, we will use the switching lemma and a second-moment method to prove that conditioned on $B$ being connected to $\Lambda_R$ in $\n_1+\n_2$, in each well-defined annulus around $B$ (the aspect-ratio of the annulus will not be constant but on contrary will grow quickly with the distance to $B$), there is a positive probability of finding $B'\in\partial_r^\square\Omega$ connected to $\Lambda_R$ in $\n_1+\n_2$. Some technicalities will force us to juggle with the different random variables $N,\underline N^\circ,\overline N,\overline N_\ep$ defined above, and this is the reason for introducing so many objects in the first place.

\begin{proof}[Proof of the upper bound in Theorem~\ref{thm:crossing free}]
We focus on the case of $R/r$ large. 
Fix an integer $K>0$ and set $\ep:=1/K$. 

We have that 
\begin{align*}
{\bf P}^{\emptyset,\emptyset}_{\Omega^\bullet,\Omega^\bullet}[N>0]&\le \underbrace{{\bf P}^{\emptyset,\emptyset}_{\Omega^\bullet,\Omega^\bullet}[N>K]}_{(R)}+\underbrace{{\bf P}^{\emptyset,\emptyset}_{\Omega^\bullet,\Omega^\bullet}[N>0,\overline N_\ep=0]}_{(S)}+\underbrace{{\bf E}^{\emptyset,\emptyset}_{\Omega^\bullet,\Omega^\bullet}[\overline N_\ep\mathbb I( N\le K)]]}_{(T)}. 
\end{align*}
Bounding $(R)$ is straightforward using the Markov inequality
\begin{align*}
(R)&\le \frac1K\, {\bf E}^{\emptyset,\emptyset}_{\Omega^\bullet,\Omega^\bullet}[N]. 
\end{align*}
For $(S)$,
conditioning for each box $B$ on $\n_1+\n_2$ outside $\overline B$, and then applying the definition of $\mathcal E_\ep(\overline B)$ (recall that $\ep=1/K$), gives
\begin{align*}
(S)&\le {\bf E}^{\emptyset,\emptyset}_{\Omega^\bullet,\Omega^\bullet}[N\mathbb I(\overline N_\ep=0)]\le \frac1K {\bf E}^{\emptyset,\emptyset}_{\Omega^\bullet,\Omega^\bullet}[ N]. 
\end{align*}
We deduce from  Lemma~\ref{lem:expected number} that
\begin{align*}
(R)+(S)&\le  \frac{2C}K.
\end{align*}
We turn to the bound on $(T)$ which represents most of the work. Fix a ball $B=\Lambda_r(x)\in\partial_r^\square\Omega$ and let $\Omega^{\bullet\bullet}$ be the graph obtained from $\Omega^\bullet$ by merging all the vertices in $ B$. As mentioned above, the idea of the proof is to show that conditioned on $B$ being connected to $\Lambda_R$ in $\n_1+\n_2$, there are many other boxes in $\partial_r^\square\Omega$ that are connected to $\Lambda_R$ in $\n_1+\n_2$, so that the probability that $N\le K$ is small. 

Let $N_B$ be the number of boxes in $\partial_r^\square\Omega$ that are connected in $\n_1+\n_2$ to $\Lambda_R$ {\rm without using any edge of} $\overline B$. 
Using that $N_B$ does not depend on $\n_1$ and $\n_2$ inside $B$, we find that 
\begin{align*}
{\bf P}^{\emptyset,\emptyset}_{\Omega^\bullet,\Omega^\bullet}[\mathcal E_\ep(\overline B),N\le K]&\le {\bf P}^{\emptyset,\emptyset}_{\Omega^\bullet,\Omega^\bullet}[\mathcal E_\ep(\overline B),N_B\le K]\\
&\le \ep^{-1}{\bf P}^{\emptyset,\emptyset}_{\Omega^{\bullet\bullet},\Omega^{\bullet\bullet}}[B\stackrel{\n_1+\n_2}\longleftrightarrow \Lambda_R,N_B\le K]\\
&= \ep^{-1}{\bf P}^{\emptyset,\emptyset}_{\Omega^{\bullet\bullet},\Omega^{\bullet\bullet}}[B\stackrel{\n_1+\n_2}\longleftrightarrow \Lambda_R]{\bf P}^{\{B,\Lambda_R\},\{B,\Lambda_R\}}_{\Omega^{\bullet\bullet},\Omega^{\bullet\bullet}}[N_B\le K],
\end{align*}
where the inequality is obtained using Remark~\ref{rmk:p}(iii)--(iv) like in the previous proof, and the equality is a consequence of the switching lemma.

We would therefore deduce $(T)\le C/K$ from the definition of $\ep$ in terms of $K$ and Lemma~\ref{lem:expected number}, if we can show that for $R/r$ large enough,
\begin{equation}\label{eq:im} {\bf P}^{\{B,\Lambda_R\},\{B,\Lambda_R\}}_{\Omega^{\bullet\bullet},\Omega^{\bullet\bullet}}[N_B\le K]\le C/K^2.\end{equation}
In order to prove \eqref{eq:im}, consider the sequence of integers $(L_i)_i$ defined by 
\[
L_i:=r8^{8^i}
\]  
and introduce $D_i$ to be the part of $\Omega$ between $\ell_{\rm in}(L_i)$ and $\ell_{\rm out}(L_i)$, where $\ell_{\rm in}(L_i)$ and $\ell_{\rm out}(L_i)$ are the arcs of $\partial\Lambda_{L_i}(x)$ and $\partial\Lambda_{L_{i+1}}(x)$ separating (in $\Omega$) $B$ from $\Lambda_R$ which are the closest to $B$.

The idea of the proof of \eqref{eq:im} is to show that for some uniform constant $c_0>0$, in each $i$ with $L_{i+1}\le R$, 
\begin{equation}\label{eq:ahu1}
{\bf P}^{\{B,\Lambda_R\},\{B,\Lambda_R\}}_{\Omega^{\bullet\bullet},\Omega^{\bullet\bullet}}[E_i]\ge \tfrac12c_0.\end{equation}
where
\begin{equation}
E_i:=\Big\{\begin{array}{c}\exists\text{ a unique $D_i$-cluster in $\n_1+\n_2$ crossing $D_i$ from $\ell_{\rm in}(L_i)$ to $\ell_{\rm in}(L_{i+1})$}\\ \text{and this cluster intersects some $B\in \partial_r^\square\Omega$ that is included in $D_i$}\end{array}\Big\}.
\end{equation}

To prove \eqref{eq:ahu1}, define $D_i^{\rm in}$ to be the connected component of $D_i\cap \mathrm{Ann}(x,L_i,L_i^2)$ closest to $B$ (in the graph $\Omega$), and $D_i^{\rm out}$ the connected component of $D_i\cap\mathrm{Ann}(x,L_i^4,L_i^8)$ closest to $\Lambda_R$. Apply Lemma~\ref{lem:1b} below in $D_i^{\rm in}$ and $D_i^{\rm out}$ and Lemma~\ref{lem:1a} to $D_i\cap \mathrm{Ann}(x,L_i^2,L_i^4)$.

The event $E_i$ being defined in such a way that it depends on $\mathrm{Ann}(x,L_i,L_{i+1})$ only, the mixing property \eqref{eq:mixing} easily implies that the probability that fewer than $c_1\log\log(R/r)$ integers $i$ are such that $E_i$ occurs is smaller than $1/\log(R/r)^{c_2}$ once $R/r$ is large enough. 

Setting $K:=c_1\log\log(R/r)$, we deduce that 
\[
 {\bf P}^{\{B,\Lambda_R\},\{B,\Lambda_R\}}_{\Omega^{\bullet\bullet},\Omega^{\bullet\bullet}}[N_B\le K]\le1/\log(R/r)^{c_2}.
\]
 This implies \eqref{eq:im} for $R/r$ large enough. This concludes the proof of the theorem with $\epsilon(x)=3C/(c_1\log\log(1/x))$, subject to the two lemmata that we used to prove \eqref{eq:ahu1}.\end{proof}

The rest of this section is devoted to the proofs of Lemmata~\ref{lem:1a} and \ref{lem:1b}.

\begin{lemma}[Existence of intersections in each annulus] \label{lem:1a}
 There exist $c_0,C_0>0$  such that for all $R,r,L$ such that $C_0\le L\le \sqrt{R/r}$, every $R$-centred domain $\Omega$, every $B=\Lambda_r(x)\in \partial^\square_r\Omega$, and a connected component $D$ of $\Omega\cap  \mathrm{Ann}(x,rL,rL^2)$ disconnecting $B$ from $\Lambda_R$, \begin{equation}
{\bf P}^{\{B,\Lambda_R\},\{B,\Lambda_R\}}_{\Omega^{\bullet\bullet},\Omega^{\bullet\bullet}}[\Lambda_R\text{ is connected in $\n_1+\n_2$ to a box of $\partial_r^\square\Omega$ included in  $D$}]\ge c_0,\end{equation}
where $\Omega^{\bullet\bullet}$ is the graph obtained from $\Omega$ by merging the vertices in $\Lambda_R$ and $B$.
\end{lemma}

\begin{remark}\label{rmk:1a}
Notice that we are using here the fact that $\Lambda_{r}(x)\subset\Omega^\square_r$ and not only $\Lambda_{r}(x)\subset\Omega$. Indeed, this guarantees that there is a ``corridor'' of width $2r$ going from $x$ to $\Lambda_R$. In particular, there is at least one ball in every $ \mathrm{Ann}(x,8^jr,8^{j+1}r)$ with $L\le 8^j\le L^2/8$ that belongs to $\partial_r^\square\Omega$.
\end{remark}

The proof of the lemma is very similar to the proof of the lower bound of Theorem~\ref{thm:crossing free}. It is based on a second moment method for a slightly modified version of $N_\ep^\circ$ (here we sum over boxes of $\partial_r^\square\Omega$ that intersect a certain box). The twist is that the second moment will be of order $(\log (R/r))^2$ instead of $\log(R/r)$, and the first moment of order $\log(R/r)$ instead of 1, which will imply a uniform lower bound on the probability.

For the next proofs (and also later in the paper), we need the following definitions. For a $R$-centred domain $\Omega$, $x\in\Omega$ such that $\Lambda_r(x)\in\partial^\square_r\Omega$, and $j\ge0$ such that $r8^j\le \tfrac18R$,  introduce the vertex $y_j=y_j(\Omega,x,r,R)\in \Omega$ and the non-negative number $\rho_j=\rho_j(\Omega,x,r,R)$ defined as follows:
\begin{itemize}[noitemsep]
\item[(i)] if there exists $\rho<\tfrac1{10}r8^j$ and $y\in \mathrm{Ann}(x,2r 8^j+2\rho,4r 8^j-2\rho)$  such that $\Lambda_{2\rho}(y)\subset \Omega^\square_r$ and $\Lambda_{6\rho}(y)$ is disconnecting $\Lambda_r(x)$ from $\Lambda_R$ in $\Omega$ but $\Lambda_{5\rho}(y)$ is not,
then write $\rho_j$ for the smallest such radius, and $y_j$ for the associated vertex $y$ (if there is more than one, pick one according to an arbitrary rule).
\item[(ii)] otherwise, fix $\rho_j:=\tfrac1{10}r8^j$ and $y_j$ any vertex on the shortest portion of $\partial\Lambda_{3r8^j}(x)$ disconnecting $B$ from $\Lambda_R$ in $\Omega$ such that $\Lambda_{5\rho_j}(y_j)\subset\Omega^\square_r\not\supset\Lambda_{6\rho_j}(y_j)$. 
\end{itemize}

A key output of the previous definitions is that the minimality of $\rho_ j$ enables to find a quad $D_j\subset\Lambda_{6\rho_j}(y_j)$ (see also Fig.~\ref{fig:definitionyrho3}) satisfying that 
\begin{itemize}[noitemsep,nolistsep]
\item $D_j\setminus\Lambda_{\rho_j/2}(y_j)$ is disconnected in two;
\item $D_j$ is disconnecting $B$ from $\Lambda_R$;
\item the extremal distance of $D_j$ is between $c_1$ and $1/c_1$ for some universal constant $c_1>0$;
\end{itemize}
Now, divide $D_ j$ into three disjoint quads $D_j^{(-1)},D_ j^{(0)},D_j^{(1)}$ with extremal distance between $c_1/3$ and $3/c_1$ (see also Fig.~\ref{fig:definitionyrho3}), where the constant $c_1$ is small yet independent of everything.

\begin{figure} 
		\begin{center}
			\includegraphics[scale=0.6]{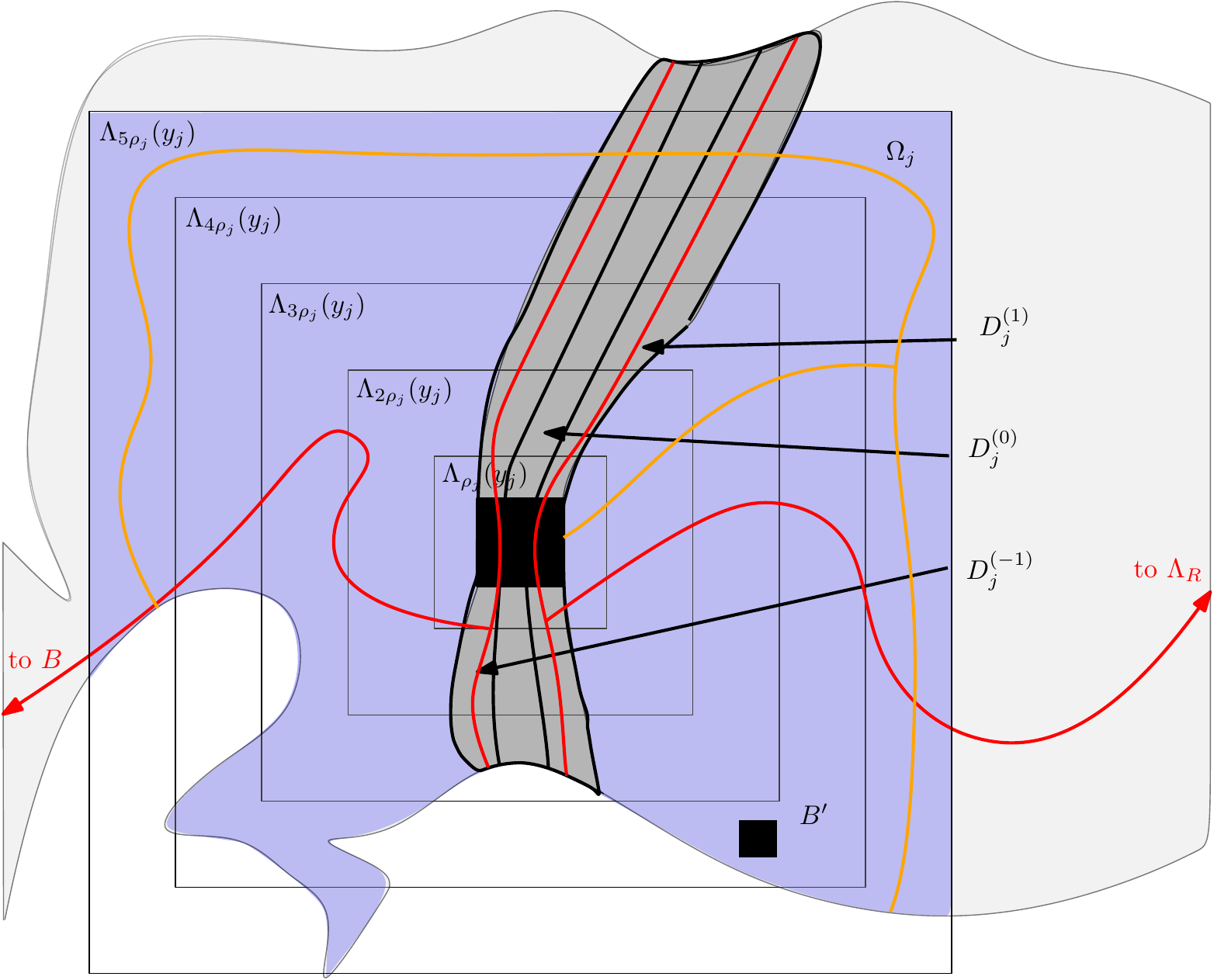}  
		\end{center}
		\caption{The domain $D_j$ in grey and black. Its existence is guaranteed by the fact that $\rho_j$ was taken to be minimal. In red, some primal and dual paths used in the proof of the claim. In blue, the domain $\Omega_j^{\bullet\bullet}$ (we depicted the merging of the vertices in black). In orange, the paths used to get \eqref{eq:hug3}.
		}	\label{fig:definitionyrho3}
\end{figure}

\begin{proof} Let $B:=\Lambda_r(x)\in \partial_r^\square\Omega$.  For $L\le \sqrt{R/r}$, let $J=J(L)$ be the set of integers $j$ such that $L\le r8^j\le \tfrac18L^2$. 
The monotonicity in sources from  Lemma~\ref{lem:monotonicity in sources} implies that 
\begin{equation}
{\bf P}^{\{B,\Lambda_R\},\{B,\Lambda_R\}}_{\Omega^{\bullet\bullet},\Omega^{\bullet\bullet}}[\Lambda_R\text{ conn.~to a box of $\partial_r^\square\Omega$ in  $D$}]\ge {\bf P}^{\{B,\Lambda_R\},\emptyset}_{\Omega^{\bullet\bullet},\Omega^{\bullet\bullet}}[\Lambda_R\text{ conn.~to a box of $\partial_r^\square\Omega$ in  $D$}].
\end{equation} 
We now focus on bounding the right-hand side.
For $j\in J$, let $N_\ep^\circ(j)$ be defined as $N_\ep^\circ$ but restricting the sum to boxes of $\partial_r^\square\Omega$ that are included in $\overline B_j:=\Lambda_{4\rho_j}(y_j)$ (note that there is at least one such box).
We have that  for every $j\ne j'$ in $J$,
 \begin{align}
{\bf E}^{\{B,\Lambda_R\},\emptyset}_{\Omega^{\bullet\bullet},\Omega^{\bullet\bullet}}[N_\ep^\circ(j)]&\ge  c_1,\label{eq:lower bound first moment}\\
 {\bf E}^{\{B,\Lambda_R\},\emptyset}_{\Omega^{\bullet\bullet},\Omega^{\bullet\bullet}}[N_\ep^\circ(j)^2] &\le C_1\log(\rho_j/r),\label{eq:upper bound second moment}\\
{ \bf E}^{\{B,\Lambda_R\},\emptyset}_{\Omega^{\bullet\bullet},\Omega^{\bullet\bullet}}[N_\ep^\circ(j)N_\ep^\circ(j')]&\le C_2\label{eq:upper bound second moment bis}.
 \end{align}
We explain how to prove these inequalities by taking the example of \eqref{eq:lower bound first moment} (the other ones can be obtained similarly). 
We start with two claims.
\paragraph{Claim 1} {\em There exists a constant $c_0>0$ independent of everything such that for every $B'\in \partial_r^\square\Omega$ contained in $\Lambda_{4\rho_j}(y_j)$, we have that 
\begin{equation}\label{eq:pos1}
{\bf P}^{\{B,\Lambda_R\},\emptyset}_{\Omega^{\bullet\bullet},\Omega^{\bullet\bullet}}[\mathcal A(B')\cap \mathcal E_\ep(B')]\ge c_0 \frac{\langle\sigma_{B}\sigma_{\underline{\underline B}'}\rangle_{\Omega^{\bullet{\bigcdot}\bullet}}\langle\sigma_{\underline{\underline B}'}\sigma_{\Lambda_R}\rangle_{\Omega^{\bullet{\bigcdot}\bullet}}}{\langle\sigma_{B}\sigma_{\Lambda_R}\rangle_{\Omega^{\bullet{\bigcdot}\bullet}}},
\end{equation}
where $\Omega^{\bullet{\bigcdot}\bullet}$ is the graph obtained from $\Omega^{\bullet\bullet}$ by merging all the vertices of the box $\underline{\underline B}'$ that is four times smaller than $B'$ together. }
\bigbreak
While we already used similar arguments in the proof of Lemma~\ref{lem:expected number}, let us provide additional details.
\begin{proof}
Following the reasoning leading to \eqref{eq:kj} and \eqref{eq:kjk} gives that for $\ep$ small enough,
\[
{\bf P}^{\{B,\Lambda_R\},\emptyset}_{\Omega^{\bullet\bullet},\Omega^{\bullet\bullet}}[\mathcal A(B')\cap \mathcal E_\ep(B')]\ge c_2 {\bf P}^{\{B,\Lambda_R\},\emptyset}_{\Omega^{\bullet\bullet},\Omega^{\bullet\bullet}}[\mathcal A(B')\cap \{\underline B'\stackrel{\n_1+\n_2}\longleftrightarrow \Lambda_R\}].
\]
Then, the mixing property, inclusion of events and the switching lemma give, like in \eqref{eq:kjkk} and \eqref{eq:ahj0}, that 
\begin{align*}
{\bf P}^{\{B,\Lambda_R\},\emptyset}_{\Omega^{\bullet\bullet},\Omega^{\bullet\bullet}}[\mathcal A(B')\cap \{\underline B'\stackrel{\n_1+\n_2}\longleftrightarrow \Lambda_R\}]&\ge c_{\rm mix} {\bf P}^{\{B,\Lambda_R\},\emptyset}_{\Omega^{\bullet{\bigcdot}\bullet},\Omega^{\bullet{\bigcdot}\bullet}}[\mathcal A(B')\cap \{{\underline B}'\stackrel{\n_1+\n_2}\longleftrightarrow \Lambda_R\}]\\
&\ge c_{\rm mix} {\bf P}^{\{B,\Lambda_R\},\emptyset}_{\Omega^{\bullet{\bigcdot}\bullet},\Omega^{\bullet{\bigcdot}\bullet}}[\mathcal A(B')\cap \{\underline{\underline B}'\stackrel{\n_1+\n_2}\longleftrightarrow \Lambda_R\}]\\
&=c_{\rm mix}\frac{\langle\sigma_{B}\sigma_{\underline{\underline B}'}\rangle_{\Omega^{\bullet{\bigcdot}\bullet}}\langle\sigma_{\underline{\underline B}'}\sigma_{\Lambda_R}\rangle_{\Omega^{\bullet{\bigcdot}\bullet}}}{\langle\sigma_{B}\sigma_{\Lambda_R}\rangle_{\Omega^{\bullet{\bigcdot}\bullet}}}{\bf P}^{\{B,\underline{\underline B}'\},\{\underline{\underline B}',\Lambda_R\}}_{\Omega^{\bullet\bullet},\Omega^{\bullet\bullet}}[\mathcal A(B')]\end{align*}
(the fact that we need the place where we merged vertices to be well apart from the edges involved in the events under consideration is the main reason for introducing $\underline{\underline B}'$). 

Then, like in \eqref{eq:ahj2} one may prove that 
\begin{equation*}
{\bf P}^{\{B,\underline{\underline B}'\},\{\underline{\underline B}',\Lambda_R\}}_{\Omega^{\bullet\bullet},\Omega^{\bullet\bullet}}[\mathcal A(B')]\ge c_3.
\end{equation*}
This concludes the proof of the claim.\end{proof}

We now use the definition of $(y_j,\rho_j)$ in a crucial fashion to get the next claim.

\paragraph{Claim 2} {\em If $\underline B_j:=\Lambda_{\rho_j}(y_j)$ and $B'\in\partial_r^\square\Omega$ included in $\Lambda_{4\rho_j}(y_j)$, we have that}
\begin{equation}\label{eq:bh0}
c_1\langle\sigma_{\underline B_j}\sigma_{\underline{\underline B}'}\rangle_{\Omega_j^{\bullet\bullet}}^2\le \frac{\langle\sigma_{B}\sigma_{\underline{\underline B}'}\rangle_{\Omega^{\bullet{\bigcdot}\bullet}}\langle\sigma_{\underline{\underline B}'}\sigma_{\Lambda_R}\rangle_{\Omega^{\bullet{\bigcdot}\bullet}}}{\langle\sigma_{B}\sigma_{\Lambda_R}\rangle_{\Omega^{\bullet{\bigcdot}\bullet}}}\le C_1\langle\sigma_{\underline B_j}\sigma_{\underline{\underline B}'}\rangle_{\Omega_j^{\bullet\bullet}}^2,
\end{equation}
{\em 
where $\Omega_j^{\bullet\bullet}$ is the graph obtained from $\Omega\cap \Lambda_{5\rho_j}(y_j)$ (note that this is included in the annulus $\mathrm{Ann}(x,r8^j,r8^{j+1})$) by identifying all the vertices that are in $\underline B_j$ together, and all those in $\underline B'$ together. }
\begin{proof} We first prove that
 \begin{equation}\label{eq:bh1}
c_1\langle\sigma_{B}\sigma_{\Lambda_R}\rangle_{\Omega^{\bullet{\bigcdot}\bullet}}\le \langle\sigma_{B}\sigma_{\underline B_j}\rangle_{\Omega^{\bullet{\bigcdot}\bullet}}\langle\sigma_{\underline B_j}\sigma_{\Lambda_R}\rangle_{\Omega^{\bullet{\bigcdot}\bullet}}\le \langle\sigma_{B}\sigma_{\Lambda_R}\rangle_{\Omega^{\bullet{\bigcdot}\bullet}}.
\end{equation}
The upper bound follows directly from the FKG inequality so we focus on the lower bound. For the lower bound, observe that for $B$ to be connected in the random cluster model to $\Lambda_R$ in $\Omega$, there must be a path from $B$ to $D_j^{(0)}$, and similarly a path from $D_j^{(0)}$ to $\Lambda_R$. As a consequence, the Edwards-Sokal coupling, the mixing property of the random cluster model and the fact that $D_j^{(0)}$ disconnects $B$ from $\Lambda_R$ give
\[
\langle\sigma_{B}\sigma_{\Lambda_R}\rangle_{\Omega^{\bullet{\bigcdot}\bullet}}=\phi^0_{\Omega^{\bullet{\bigcdot}\bullet}}[B\stackrel{\omega}\longleftrightarrow \Lambda_R]\le C_{\rm mix}\phi^0_{\Omega^{\bullet{\bigcdot}\bullet}}[B\stackrel{\omega}\longleftrightarrow D_j^{(0)}]\phi^0_{\Omega^{\bullet{\bigcdot}\bullet}}[D_j^{(0)}\stackrel{\omega}\longleftrightarrow \Lambda_R].
\] 
It only remains to observe that the RSW theorem (to create open crossings in $D_j^{(-1)}$ and $D_j^{(1)}$ disconnecting $B$ from $\Lambda_R$ in $\Omega$ as in Fig.~\ref{fig:definitionyrho3}) implies that
\[
\phi_{\Omega^{\bullet{\bigcdot}\bullet}}[B\stackrel{\omega}\longleftrightarrow D_j^{(0)}]\le C_1\phi^0_{\Omega^{\bullet{\bigcdot}\bullet}}[B\stackrel{\omega}\longleftrightarrow \underline B_j]\quad\text{and}\quad \phi^0_{\Omega^{\bullet{\bigcdot}\bullet}}[D_j^{(0)}\stackrel{\omega}\longleftrightarrow \Lambda_R]\le C_1\phi^0_{\Omega^{\bullet{\bigcdot}\bullet}}[\underline B_j\stackrel{\omega}\longleftrightarrow \Lambda_R].
\]
Together with the Edwards-Sokal coupling, this gives \eqref{eq:bh1}.

We can now prove the lower bound of \eqref{eq:bh0}. First, 
\begin{equation}\label{eq:hug4}
\langle\sigma_{\underline B_j}\sigma_{\underline{\underline B}'}\rangle_{\Omega^{\bullet{\bigcdot}\bullet}}\ge c_2\langle\sigma_{\underline B_j}\sigma_{\underline{\underline B}'}\rangle_{\Omega_j^{\bullet\bullet}}.
\end{equation}
Indeed, use the Edwards-Sokal coupling to rephrase the problem in terms of the random cluster model. Then, the FKG inequality and RSW show that conditionally on $\underline B_j$ being connected to $\underline{\underline B}'$, there is an open path in $\Lambda_{9\rho_j/2}(y_j)$ disconnecting $\Lambda_{4\rho_j}(y_j)$ from $\partial\Lambda_{9\rho_j/2}(y_j)$ with probability bounded by a uniform constant. Now, conditioned on everything inside $\Lambda_{9\rho_j/2}(y_j)$, we can use RSW to prove that there does not exist any path from $\Lambda_{9\rho_j/2}(y_j)$ to $\partial\Lambda_{5\rho_j}(y_j)$ with uniformly bounded probability. We deduce that the probability that $\underline B_j$ is connected to $\underline{\underline B}'$ but not to $\partial\Lambda_{5\rho_j}(y_j)$ is larger than constant times the probability that $\underline B_j$ and $\underline{\underline B}'$ are connected. Conditioning on the absence of connection, then using the spatial-Markov property and the comparison between boundary conditions concludes the proof of \eqref{eq:hug4}. 

Use the FKG inequality, then \eqref{eq:bh1} and \eqref{eq:hug4} to get that 
\begin{align*}
\langle\sigma_{B}\sigma_{\underline{\underline B}'}\rangle_{\Omega^{\bullet{\bigcdot}\bullet}}\langle\sigma_{\underline{\underline B}'}\sigma_{\Lambda_R}\rangle_{\Omega^{\bullet{\bigcdot}\bullet}}&\ge \langle\sigma_{B}\sigma_{\underline B_j}\rangle_{\Omega^{\bullet{\bigcdot}\bullet}}\langle\sigma_{\underline B_j}\sigma_{\underline{\underline B}'}\rangle_{\Omega^{\bullet{\bigcdot}\bullet}}^2\langle\sigma_{\Lambda_R}\sigma_{\underline B_j}\rangle_{\Omega^{\bullet{\bigcdot}\bullet}}\\
&\ge c_1c_2\langle\sigma_{B}\sigma_{\Lambda_R}\rangle_{\Omega^{\bullet{\bigcdot}\bullet}}\langle\sigma_{\underline B_j}\sigma_{\underline{\underline B}'}\rangle_{\Omega_j^{\bullet\bullet}}^2.
\end{align*}
We now turn to the upper bound. We prove that 
\begin{equation}\label{eq:bh2}
\langle\sigma_{B}\sigma_{\underline{\underline B}'}\rangle_{\Omega^{\bullet{\bigcdot}\bullet}}\le C\langle\sigma_{B}\sigma_{\underline B_j}\rangle_{\Omega^{\bullet{\bigcdot}\bullet}}\langle\sigma_{\underline B_j}\sigma_{\underline{\underline B}'}\rangle_{\Omega_j^{\bullet\bullet}}\quad\text{and}\quad\langle\sigma_{\Lambda_R}\sigma_{\underline{\underline B}'}\rangle_{\Omega^{\bullet{\bigcdot}\bullet}}\le C\langle\sigma_{\Lambda_R}\sigma_{\underline B_j}\rangle_{\Omega^{\bullet{\bigcdot}\bullet}}\langle\sigma_{\underline B_j}\sigma_{\underline{\underline B}'}\rangle_{\Omega_j^{\bullet\bullet}}.
\end{equation}
In order to see it, use the RSW estimates for the random cluster model to get that
\begin{align}
\langle\sigma_{B}\sigma_{\underline{\underline B}'}\rangle_{\Omega^{\bullet{\bigcdot}\bullet}}=\phi_{\Omega^{\bullet{\bigcdot}\bullet}}[B\longleftrightarrow \underline{\underline B}']
&\le C\phi_{\Omega^{\bullet{\bigcdot}\bullet}}[B\stackrel{\omega}\longleftrightarrow \underline B_j]\phi_{\Omega_j^{\bullet\bullet}}[\underline B_j\stackrel{\omega}\longleftrightarrow \underline{\underline B}'].\label{eq:hug3}
\end{align}
More precisely, we perform a construction which is related to the one leading to \eqref{eq:hug4}. Namely, we construct a crossing in $\Omega_j^{\bullet\bullet}$ disconnecting $\overline B_j$ from $B$, and then a path from $\underline B_j$ to $\partial\Omega_j^{\bullet\bullet}$ with positive probability, and use the FKG inequality to combine it with the path from $B$ to $\underline B'$ (see the orange paths in Fig.~\ref{fig:definitionyrho3}). 

Similarly, we find that 
\begin{align}
\langle\sigma_{\Lambda_R}\sigma_{\underline{\underline B}'}\rangle_{\Omega^{\bullet{\bigcdot}\bullet}}
&\le C\phi_{\Omega^{\bullet{\bigcdot}\bullet}}[\Lambda_R\stackrel{\omega}\longleftrightarrow \underline B_j]\phi_{\Omega_j^{\bullet\bullet}}[\underline B_j\stackrel{\omega}\longleftrightarrow \underline{\underline B}'].
\end{align}
Combining the two  inequalities in \eqref{eq:bh2} together with the right inequality of \eqref{eq:bh1} implies the upper bound.
\end{proof}
We are now in a position to prove \eqref{eq:lower bound first moment}. Indeed, one uses the previous claims and an estimate on spin-spin correlations that is similar to the lower bound in Lemma~\ref{lem:expected number} to directly get \eqref{eq:lower bound first moment}. 

The other inequalities are obtained in a similar fashion as in the lower bound of Theorem~\ref{thm:crossing free} using  \cite[Proposition~A.3]{AizDum19} to express the probability that two boxes are connected to $B$, this time using variations around the upper bound in the previous claim.

Overall, we deduce that
  \begin{align*}
 {\bf E}^{\{ B,\Lambda_R\},\emptyset}_{\Omega^{\bullet\bullet},\Omega^{\bullet\bullet}}\Big[\sum_{j\in J}N_\ep^\circ(j)\Big]&\ge 2c_1\log L,\\
  {\bf E}^{\{B,\Lambda_R\},\emptyset}_{\Omega^{\bullet\bullet},\Omega^{\bullet\bullet}}\Big[\Big(\sum_{j\in J}N_\ep^\circ(j)\Big)^2\Big]&\le C_1\sum_{j\in J}\log(\rho_j/r)+C_3(\log L)^2\le C_4(\log L)^2
 \end{align*}
 since $\rho_j\le r8^j$.
 
 The Cauchy-Schwarz inequality implies that 
 \[
 {\bf E}^{\{B,\Lambda_R\},\emptyset}_{\Omega^{\bullet\bullet},\Omega^{\bullet\bullet}}[\exists j\in J:N_\ep^\circ(j)>0]\ge c_3(\ep).
 \]
Since the event on the left implies the existence of a path from $\Lambda_R$ to a box of $\partial_r^\square\Omega$ in the annulus, we deduce the result. \end{proof}

We introduce a few more notation, see Fig.~\ref{fig:Omega0}. For $L$, let
\begin{align*}
\ell_{\rm in}=\ell_{\rm in}(L)&:=\text{the arc of $\partial\Lambda_{rL}(x)$ disconnecting $B$ from $\Lambda_R$ in $\Omega$ which is closest to $B$ in $\Omega$},\\
\ell'_{\rm in}=\ell_{\rm in}'(L)&:=\text{the arc of $\partial\Lambda_{rL^{4/3}}(x)$ disconnecting $B$ from $\Lambda_R$ in $\Omega$ which is closest to $B$ in $\Omega$,}\\
\ell'_{\rm out}=\ell'_{\rm out}(L)&:=\text{the arc of $\partial\Lambda_{rL^{5/3}}(x)$ disconnecting $B$ from $\Lambda_R$ in $\Omega$ which is closest to $B$ in $\Omega$,}\\
 \ell_{\rm out}=\ell_{\rm out}(L)&:=\text{the arc of $\partial\Lambda_{rL^2}(x)$ disconnecting $B$ from $\Lambda_R$ in $\Omega$ which is closest to $B$ in $\Omega$}.
\end{align*}
Let $D=D(L)$ be the part of $\Omega$ between $\ell_{\rm in}$ and $\ell_{\rm out}$.

\begin{lemma}[Uniqueness of the cluster crossing an annulus] \label{lem:1b} For every $\eta>0$, there exists $C=C(\eta)>0$ such that for all $R,r,L$ such that $C\le L\le \sqrt{R/r}$, every $R$-centred domain $\Omega$, every $B\in \partial^\square_r\Omega$,
\begin{align*}
{\bf P}^{\{B,\Lambda_R\},\{B,\Lambda_R\}}_{\Omega^{\bullet\bullet},\Omega^{\bullet\bullet}}[\exists \text{ 2 $D$-clusters in $\n_1+\n_2$ intersecting both }\ell_{\rm in}\text{ and }\ell_{\rm out}]&\le \eta,
\end{align*}
where $\Omega^{\bullet\bullet}$ is the graph obtained from $\Omega$ by merging the vertices in $\Lambda_R$, and those in $B$.
\end{lemma}
Let $D_{\rm in}$ be the part of $D$ between $\ell_{\rm in}$ and $\ell_{\rm in}'$ (with these two arcs included) and  $D_{\rm out}$  be the part of $D$ between $\ell_{\rm out}'$ and $\ell_{\rm out}$ (with these two arcs included).

Below, for a current $\n$ we write $\eta(\n)$ for the set of edges with odd current in $\n$. For the proof, we proceed in three steps:
\begin{itemize}
\item In the first step, we show that with very good probability, only one $D_{\rm in}$-cluster of $\eta(\n_1)$ is crossing $D_{\rm in}$ and similarly for $D_{\rm out}$.  Due to the source constraint, this implies that on this event there exists exactly one $D$-cluster in $\eta(\n_1)$, denoted $\mathbf C(\n_1)$, containing a crossing of $D_{\rm in}$ and $D_{\rm out}$. The same is true for $\n_2$ (we introduce the corresponding random variable $\mathbf C(\n_2)$).
\item In the second step, we prove that with very good probability, on the previous event, the two clusters $\mathbf C(\n_1)$ and $\mathbf C(\n_2)$ intersect.
This is the most technical part of the proof. The idea is to first look at $\mathbf C(\n_1)$, and see that it must typically ``use a substantial amount of the room between $\ell_{\rm in}'$ and $\ell_{\rm out}'$'', and then to see that $\mathbf C(\n_2)$ has small probability to ``cross from $\ell_{\rm in}'$ and $\ell_{\rm out}'$ without intersecting $\mathbf C(\n_1)$''.

\item In the last step, we prove that there is no $D$-cluster in $\n_1+\n_2$ intersecting $\ell_{\rm in}$ and $\ell_{\rm out}$ but not $\mathbf C(\n_1)\cup\mathbf C(\n_2)$.
\end{itemize}

\begin{figure} 
		\begin{center}
			\includegraphics[scale=1.0]{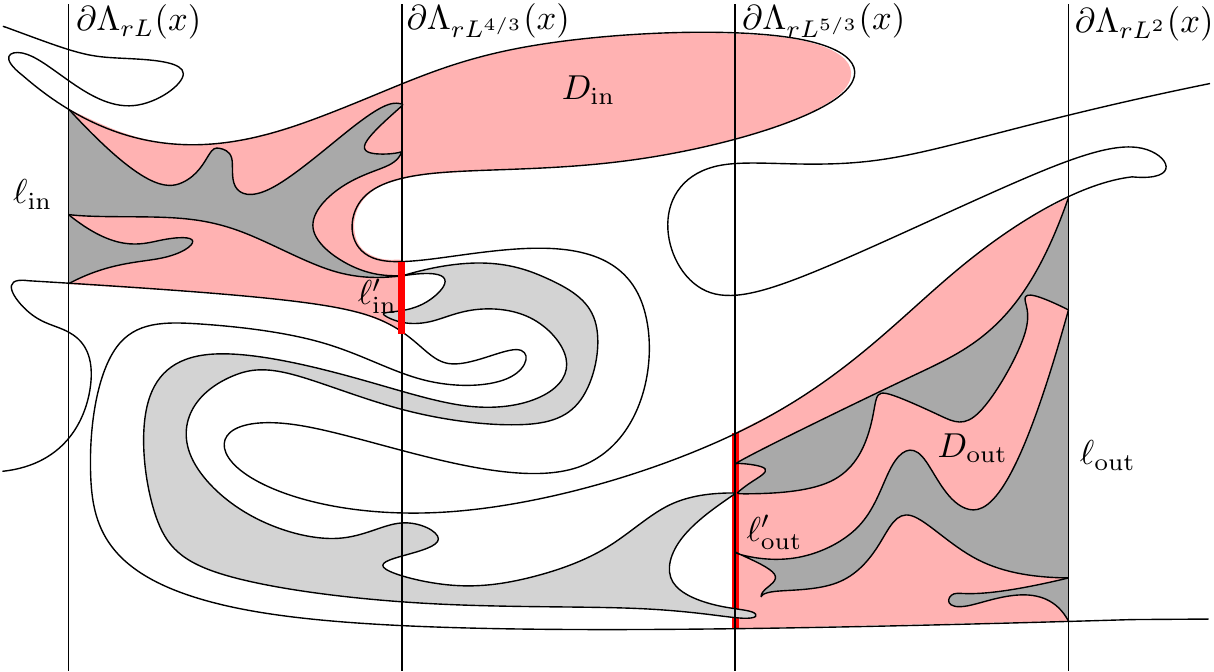}  
		\end{center}
		\caption{The arcs $\ell_{\rm in}$ and $\ell_{\rm out}$ as well as the domains $D_{\rm in}$ and $D_{\rm out}$ (in red). We also depicted in dark gray the $D_{\rm in}-$ and $D_{\rm out}$-clusters of $\eta(\n_1)$ that intersect $\ell'_{\rm in}$ and $\ell_{\rm in}$, and $\ell'_{\rm out}$ and $\ell_{\rm out}$ respectively. In light gray, the $D$-clusters of $\n_1$ intersecting $\ell'_{\rm in}\cup\ell'_{\rm out}$.
		}	\label{fig:Omega0}
\end{figure}

\begin{proof}
We refer to Fig.~\ref{fig:Omega0} for the following definitions. Introduce the events \begin{align*}
F_{\rm in}&:=\{\n: \text{$\exists$ two $D_{\rm in}$-clusters of $\eta(\n)$ intersecting $\ell_{\rm in}$ and $\ell_{\rm in}'$}\},\\
F_{\rm out}&:=\{\n: \text{$\exists$ two $D_{\rm out}$-clusters of $\eta(\n)$ intersecting $\ell_{\rm out}'$ and $\ell_{\rm out}$}\}\end{align*}
and 
\[
F:=\{(\n_1,\n_2):\n_1\text{ or }\n_2\text{ belongs to } F_{\rm in}\cup F_{\rm out}\}.
\]
On $F^c$, define $\mathbf C(\n_1)$ and $\mathbf C(\n_2)$ to be the {\em unique} $D$-clusters in $\eta(\n_1)$ and $\eta(\n_2)$ intersecting $\ell_{\rm in}$ and $\ell_{\rm out}$ and introduce
\begin{align*}
G&:=F^c\cap \{(\n_1,\n_2):\mathbf C(\n_1)\cap\mathbf C(\n_2)= \emptyset\},\\
H&:=F^c\cap \{(\n_1,\n_2):\exists\text{ a $D$-cluster in $\n_1+\n_2$ intersecting $\ell_{\rm in}$ and $\ell_{\rm out}$ but not $\mathbf C(\n_1)\cup\mathbf C(\n_2)$}\}.
\end{align*}
We start with the observation that 
\[
\{\exists \text{ 2 $D$-clusters in $\n_1+\n_2$ intersecting both }\ell_{\rm in}\text{ and }\ell_{\rm out}\}\subset F\cup G\cup H
\]
so that it suffices to bound the probability of the three events on the right separately.  We do it for each event separately in the following three claims. 

\paragraph{Claim 1.} {\em There exist $c,C\in(0,\infty)$ independent of everything such that}
\[
\mathbf P_{\Omega^{\bullet\bullet},\Omega^{\bullet\bullet}}^{\{B,\Lambda_R\},\{B,\Lambda_R\}}[F]\le CL^{-c}.
\]

\begin{proof}[Proof of Claim 1]
We bound the probability of $\n_1\in F_{\rm in}$. Call $(ab)$ and $(cd)$ the parts of $\partial\Omega$ corresponding to the boundaries of $D$ (or equivalently the two arcs obtained from $\partial D_{\rm in}$ by removing $\ell_{\rm in}$ and $\ell'_{\rm in}$). 
Let $\Gamma$ be the crossing in $\eta(\n_1)\cap D_{\rm in}$  between $\ell'_{\rm in}$ and $\ell_{\rm in}$ that is the closest to $(ab)$ -- note that $\Gamma$ must exist because of the source constraints. Set $\mathbf D^*$ to be the set of faces of $D_{\rm in}^*$ that are reachable from $(cd)$ in $D_{\rm in}^*$ without crossing $\Gamma$. 

In the low-temperature expansion interpretation of Proposition~\ref{prop:coupling Kramers Wannier}, the source constraints on $\eta(\n)$ implies that the faces bordering $(ab)$ and $(cd)$ receive spin say plus for those bordering $(ab)$, and minus for those bordering $(cd)$. Now, the definition of $\Gamma$ implies that the faces of $\mathbf D^*$ bordering $\Gamma$ on the side of $(cd)$ also receive the spin minus. 
Also, the existence of an additional crossing in $\eta(\n_1)$ crossing $\mathbf D^*$ from $\ell_{\rm in}$ to $\ell'_{\rm in}$ would imply the existence of a $*$-connected path of faces with spin plus going from $\ell_{\rm in}$ to $\ell'_{\rm in}$ in $\mathbf D^*$. 

Using the FKG inequality and the fact that conditioned on $\Gamma$, the Ising model on $\mathbf D^*$ has minus boundary conditions on the part of the boundary  strictly inside ${\rm Ann}(x,rL,rL^{4/3})$, this probability is smaller than the probability that for an Ising model on ${\rm Ann}(rL,rL^{4/3})$ with plus boundary conditions, there is no path of minuses surrounding the origin. We conclude\footnote{The probability for the critical Ising model on ${\rm Ann}(k,K)$ with plus boundary conditions of not finding a circuit of minus surrounding the origin is bounded by $C(k/K)^c$. Indeed, consider the Edwards-Sokal coupling from Proposition~\ref{prop:coupling ES} with the random cluster measure with wired boundary conditions. Using \eqref{eq:RSW}, we may show that with probability $1-C_0(\tfrac kK)^{c_0}$ for some $c_0$ small enough, there exist $c_0\log(K/k)$ distinct clusters in ${\rm Ann}(k/K)$ surrounding $0$, and not connected to $\partial\Lambda_k$ or $\partial\Lambda_K$. Since each of this cluster receives a spin minus with probability 1/2 thanks to the Edwards-Sokal coupling, we deduce that the probability of having a crossing of pluses from inside to outside in ${\rm Ann}(k,K)$ with plus boundary conditions is bounded by $C(\tfrac kK)^{3c}$ for some uniform constants $c,C>0$.
} that for some $c,C\in(0,\infty)$ independent of everything,
\begin{equation}\label{eq:uhug}
\mathbf P_{\Omega^{\bullet\bullet}}^{\{B,\Lambda_R\}}[F_{\rm in}]\le CL^{-c}.
\end{equation}
We may obtain the same bound for the event $\n_1\in F_{\rm out}$ and for $\n_2$. The result follows by applying the union bound and by changing the constants $c$ and $C$. 
\end{proof}
 
We turn to the bound on the probability of $G$.
 \paragraph{Claim 2.} {\em There exist $c,C\in(0,\infty)$ independent of everything such that }
\[
\mathbf P_{\Omega^{\bullet\bullet},\Omega^{\bullet\bullet}}^{\{B,\Lambda_R\},\{B,\Lambda_R\}}[G]\le CL^{-c}.
\]

\begin{proof}[Proof of Claim 2]
Introduce the two sets  
 \begin{align*}
\pmb\Omega_1&:={\bf A}_1\cup (D\setminus\{\text{union of $(D_{\rm in}\cup D_{\rm out})$-clusters of $\eta(\n_1)$ intersecting $\ell_{\rm in}\cup\ell_{\rm out}$}\}),\\
\pmb\Omega_2&:={\bf A}_2\cup (D\setminus\{\text{union of $(D_{\rm in}\cup D_{\rm out})$-clusters of $\eta(\n_2)$ intersecting $\ell_{\rm in}\cup\ell_{\rm out}$}\}),
\end{align*}
where, for $i=1,2$, ${\bf A}_i$ is the set of sources in $\ell_{\rm in}'\cup\ell_{\rm out}'$ of the restrictions of $\eta(\n_i)$ to the union of $(D_{\rm in}\cup D_{\rm out})$-clusters of $\eta(\n_i)$ intersecting $\ell_{\rm in}\cup\ell_{\rm out}$.

We reuse the definitions of $(y_j,\rho_j)$ from the proof of the previous lemma.
For every realization $(\Omega_1,\Omega_2,A_1,A_2)$ of $(\pmb\Omega_1,\pmb\Omega_2,{\bf A}_1,{\bf A}_2)$ compatible with the event $F^c$, the spatial Markov property implies that 
\begin{equation}
\mathbf P_{\Omega^{\bullet\bullet},\Omega^{\bullet\bullet}}^{\{B,\Lambda_R\},\{B,\Lambda_R\}}[G|(\pmb\Omega_1,\pmb\Omega_2,{\bf A}_1,{\bf A}_2)=(\Omega_1,\Omega_2,A_1,A_2)]= \mathbf P_{\Omega_1,\Omega_2}^{A_1,A_2}[\mathbf C'(\n_1)\cap \mathbf C'(\n_2)=\emptyset],
\end{equation}
where
 $\mathbf C'(\n_i)$ is the union of  the connected components of $\eta(\n_i)\cap \Omega_i$ intersecting $A_i$. 
 
 Consider a subset $J'$ of the set $J_0\subset J$ of indices $j$ with $L^{4/3}\le r8^j< \tfrac18 L^{5/3}$ and let 
\[
G'(J'):=\{\text{for every $j\in J'$, }\mathbf C'(\n_1)\text{ does  not intersect }B_j\},
\]
where $B_j:=\Lambda_{\rho_j}(y_j)$ (recall the definition of $\rho_j$ and $y_j$ above Lemma~\ref{lem:1b}). By conditioning on the union $\underline{\mathbf C}(\n_1)$ of the clusters in $\n_1$ intersecting $B_j$ for some $j\in J'$, we may write that for every possible realization $\underline C$ of $\underline{\mathbf C}(\n_1)$,
\begin{align}
{\bf P}^{A_1}_{\Omega_1}[\underline{\mathbf C}(\n_1)=\underline C]&=\sum_{\n\sim \underline C}w(\n)\frac{Z^{A_1}(\Omega_1\setminus\underline C)}{Z^{A_1}(\Omega_1)},
\end{align}
where for each $\underline C$, $\n\sim\underline C$ denotes a current on the set of edges with endpoints in $\underline C$ satisfying that every vertex in $\underline C$ is connected in $\n$ to $B_j$ for some $j\in J'$. To get the previous equality, we used that $\n_1$ is zero on edges with one endpoint in $\underline C$ and one outside of $\underline C$.
Now, one can rewrite the right-hand side in such a way that 
\begin{align}\label{eq:ahah1}
{\bf P}^{A_1}_{\Omega_1}[\underline{\mathbf C}(\n_1)=\underline C]&={\bf P}^{\emptyset}_{\Omega_1}[\underline{\mathbf C}(\n_1)=\underline C]\frac{\langle\sigma_{A_1}\rangle_{\Omega_1\setminus\underline C}}{\langle\sigma_{A_1}\rangle_{\Omega_1}}\le {\bf P}^{\emptyset}_{\Omega_1}[\underline{\mathbf C}(\n_1)=\underline C]\frac{\langle\sigma_{A_1}\rangle_{\Omega_1\setminus(\cup_{j\in J'} B_j)}}{\langle\sigma_{A_1}\rangle_{\Omega_1}},
\end{align}
where the second inequality is due to the monotonicity of spin-spin correlations in coupling constants. Summing over possible realizations of $\underline C$ compatible with the occurrence of $G'(J')$ gives  that 
\begin{align}
{\bf P}^{A_1}_{\Omega_1}[G'(J')]&\le \frac{\langle\sigma_{A_1}\rangle_{ \Omega_1\setminus(\cup_{j\in J'} B_j)}}{\langle\sigma_{A_1}\rangle_{\Omega_1}}.\label{eq:aho0}\end{align}
Now, we claim that there exists $c_0>0$ independent of everything such that  for every $\mathbf J\subset J_0$ and $\mathbf j\in J_0\setminus\mathbf J$, 
\begin{equation}\label{eq:aho}
\langle\sigma_{A_1}\rangle_{ \Omega_1\setminus(\cup_{j\in \mathbf J\cup\{\mathbf j\}} B_j)}\le (1-c_0)\langle\sigma_{A_1}\rangle_{ \Omega_1\setminus(\cup_{j\in \mathbf J}B_j)}.
\end{equation}
Indeed, let $\{\mathcal F_{A_1}\text{ outside }B_{\mathbf j}\}$ be the event that every cluster of $\omega_{|B_{\mathbf j}^c}$ contains an even number of vertices in $A_1$. The fact that we work on $\Omega_1\setminus(\cup_{j\in \mathbf J\cup\{\mathbf j\}} B_j)$ (for the equality) and the comparison between boundary conditions (for the inequality) give that 
\begin{align*}
\phi^0_{\Omega_1\setminus(\cup_{j\in \mathbf J\cup\{\mathbf j\}} B_j)}[\mathcal F_{A_1}]=\phi^0_{\Omega_1\setminus(\cup_{j\in \mathbf J\cup\{\mathbf j\}} B_j)}[\mathcal F_{A_1}\text{ outside }B_{\mathbf j}]&\le\phi^0_{\Omega_1\setminus(\cup_{j\in \mathbf J} B_j)}[\mathcal F_{A_1}\text{ outside }B_{\mathbf j}].
\end{align*}
Recall the definitions of $D_j^{(-1)},D_j^{(0)},D_j^{(1)}$ from above. Introduce the events (see Fig.~\ref{fig:definitionyrho2}):
\begin{itemize}[noitemsep,nolistsep]
\item[P1] There exist open paths in $D_\mathbf j^{(-1)}$ and $D_\mathbf j^{(1)}$ disconnecting $B$ from $\Lambda_R$ in $\Omega$, 
\item[P2] There exist two dual paths from $\partial\Omega$ to $B_\mathbf j$ disconnecting $D_\mathbf j^{(-1)}$ from $D_\mathbf j^{(1)}$ in $\Omega\setminus B_\mathbf j$, 
\item[P3] There exists an open path connecting the crossings in $D_\mathbf j^{(-1)}$ and $D_\mathbf j^{(1)}$.
\end{itemize} 
By \eqref{eq:RSW}, we get  that
\[
\phi^0_{\Omega_1\setminus(\cup_{j\in \mathbf J}B_j)}[\mathcal F_{A_1}\cap P_1\cap P_2\cap P_3]\ge c_0\phi^0_{\Omega_1\setminus(\cup_{j\in \mathbf J}B_j)}[\mathcal F_{A_1}].
\]
Since on the event on the left, $\{\mathcal F_{A_1}\text{ outside }B_{\mathbf j}\}$ does not occur (any path from $A_1\cap\ell_{\rm in}'$ to $A_1\cap\ell_{\rm out}'$ is forced to go through $B_\mathbf j$), we deduce that
\begin{align*}
\phi^0_{\Omega_1\setminus(\cup_{j\in \mathbf J}B_j)}[\mathcal F_{A_1}\text{ outside }B_{\mathbf j}]&\le(1-c_0)\phi^0_{\Omega_1\setminus(\cup_{j\in \mathbf J}B_j)}[\mathcal F_{A_1}].
\end{align*}
We deduce \eqref{eq:aho} from the two previous displayed equations using the Edwards-Sokal coupling (Proposition~\ref{prop:coupling ES}). 

Applying \eqref{eq:aho} repeatedly gives
\begin{equation}\label{eq:aho1}
\langle\sigma_{A_1}\rangle_{ \Omega_1\setminus(\cup_{j\in J'}  B_j)}\le (1-c_0)^{|J'|}\langle\sigma_{A_1}\rangle_{ \Omega_1},
\end{equation}
which implies, when plugged into \eqref{eq:aho0}, that \begin{align*}
{\bf P}^{A_1}_{\Omega_1}[G'(J')]&\le (1-c_0)^{|J'|}.\end{align*}

If $G'$ is the union of the $G'(J')$ for $|J_0\setminus J'|\le c_1\log L$, the union bound gives that provided that $c_1$ is small enough,
\begin{equation}\label{eq:ahoo}
{\bf P}^{A_1}_{\Omega_1}[G']\le e^{c_2 \log L} \cdot (1-c_0)^{|J'|}\le C_3L^{-c_3}.
\end{equation}
On the other hand, conditioning on $\mathbf C'(\n_1)$, then on the clusters of $\eta(\n_2)$ intersecting $\mathbf C'(\n_1)$, and then using the same proof as for \eqref{eq:ahah1}, we obtain that
\begin{align}\mathbf P_{\Omega_1,\Omega_2}[G\setminus G']
&\le \mathbf E_{\Omega_1,\Omega_2}^{A_1,\emptyset}\Big[\frac{\langle\sigma_{A_2}\rangle_{\Omega_2\setminus\mathbf C(\n_1)}}{\langle\sigma_{A_2}\rangle_{\Omega_2}}\mathbf 1_{\n_1\notin G'}\Big].\label{eq:hug6}\end{align}
Conditioning on $\n_1\notin G'$, one can follow an argument similar to the one leading to \eqref{eq:ahoo} to get that 
\begin{equation}\label{eq:hug7}
\frac{\langle\sigma_{A_2}\rangle_{\Omega_2\setminus\mathbf C(\n_1)}}{\langle\sigma_{A_2}\rangle_{\Omega_2}}\le CL^{-c}.
\end{equation}
Claim 2 follows from the combination of \eqref{eq:ahoo}--\eqref{eq:hug7} above.
\end{proof}
We conclude the proof with the bound on the probability of the event $H$. 

\paragraph{Claim 3.} {\em There exist $c,C\in(0,\infty)$ independent of everything such that }
\[
\mathbf P_{\Omega^{\bullet\bullet},\Omega^{\bullet\bullet}}^{\{B,\Lambda_R\},\{B,\Lambda_R\}}[H]\le CL^{-c}.
\]

\begin{proof}[Proof of Claim 3]
Recall the definition of $\pmb\Omega_i$ from the previous proof and introduce
 \begin{align*}
\pmb\Omega'_i:=\pmb\Omega_i\setminus\{\text{the $D$-cluster of $\n_1+\n_2$ intersecting $\mathbf C(\n_1)\cup\mathbf C(\n_2)$}\}.
\end{align*}
For every possible realization $(\Omega_1',\Omega'_2)$ of $(\pmb\Omega_1',\pmb\Omega'_2)$,  the spatial Markov property implies that 
\[
\mathbf P_{\Omega^{\bullet\bullet},\Omega^{\bullet\bullet}}^{\{B,\Lambda_R\},\{B,\Lambda_R\}}[H|(\pmb\Omega_1',\pmb\Omega_2')=(\Omega_1',\Omega_2')]\le \mathbf P_{\Omega_1',\Omega_2'}^{\emptyset,\emptyset}[\ell'_{\rm in}\stackrel{\n_1+\n_2}\longleftrightarrow \ell'_{\rm out}].
\]
Using that $\Omega'_1$ and $\Omega'_2$ coincide between $\ell_{\rm in}'$ and $\ell_{\rm out}'$, the  monotonicity in coupling constants (Lemma~\ref{lem:monotonicity}) gives that
\[
\mathbf P_{\Omega_1',\Omega_2'}^{\emptyset,\emptyset}[\ell_{\rm in}'\stackrel{\n_1+\n_2}\longleftrightarrow \ell_{\rm out}']\le \mathbf P_{\Omega_0^{\bullet\bullet},\Omega_0^{\bullet\bullet}}^{\emptyset,\emptyset}[\ell_{\rm in}'\stackrel{\n_1+\n_2}\longleftrightarrow \ell_{\rm out}'],
\]
where $\Omega_0^{\bullet\bullet}$ denotes the graph obtained from $\Omega'_1$ (or equivalently $\Omega'_2$) by merging the vertices enclosed by $\ell_{\rm in}'$ and exterior to $\ell_{\rm out}'$ into two vertices denoted $\ell_{\rm in}'$ and $\ell_{\rm out}'$. Now,
\[
\mathbf P_{\Omega_0^{\bullet\bullet},\Omega_0^{\bullet\bullet}}^{\emptyset,\emptyset}[\ell_{\rm in}'\stackrel{\n_1+\n_2}\longleftrightarrow \ell_{\rm out}']= \phi_{\Omega_0^{\bullet\bullet}}^0[\ell_{\rm in}'\stackrel\omega\longleftrightarrow \ell_{\rm out}']^2\le CL^{-c}.
\]
Summing over every possible $\Omega_0$ gives the result.
\end{proof}

The previous three claims together conclude the proof.
\end{proof}
\begin{figure} 
		\begin{center}
			\includegraphics[scale=0.6]{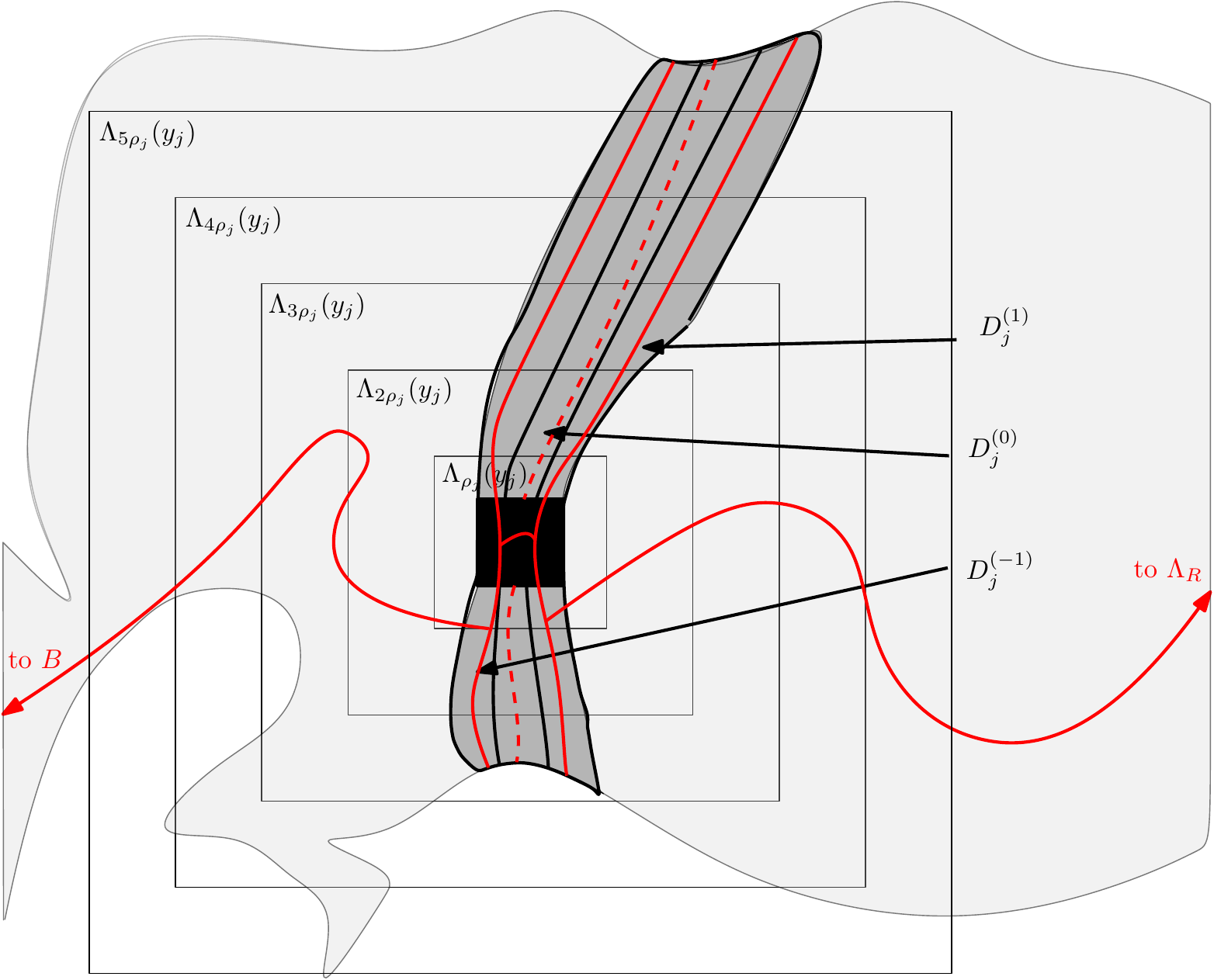}  
		\end{center}
		\caption{In red, a combination of primal and dual paths (in dashed) in the random cluster model guaranteeing that  conditioned on $\mathcal F_{A_1}$, with positive probability $\mathcal F_{A_1}$ does not occur using edges of $\Omega_1\setminus B_j$ only.
		}	\label{fig:definitionyrho2}
\end{figure}

\subsection{A related corollary}

We call a quad $(\Omega,a,b,c,d)$ $c_0$-{\em regular} at scale $R$ if $\Omega$ is $R$-centred and if the probability for a simple random-walk starting from $0$ to end on $(ab)$, $(bc)$, $(cd)$ and $(da)$ is larger than $c_0$. Let us mention that by construction  the distance between $(ab)$ and $(cd)$ is larger than or equal to $c_1R$ for some $c_1=c_1(c_0)>0$.
Let $\partial_r(ab)$ and $\partial_r(cd)$ be the set of vertices within a distance $r$ of $(ab)$ and $(cd)$ respectively.

\begin{corollary}[Boundary to boundary crossing probability in double random current]\label{cor:crossing double random current in quads}
For every $c_0>0$, there exists $c=c(c_0)>0$ such that for all $r,R$ with $r\le cR$, every $c_0$-regular quad $(\Omega,a,b,c,d)$ at scale $R$,
\begin{equation}
\mathbf P_{\Omega,\Omega}^{\emptyset,\emptyset}[\partial_r(ab)\stackrel{\n_1+\n_2}\longleftrightarrow\partial_r(cd)]\ge \frac{c}{\log(R/r)^2}.
\end{equation}
\end{corollary}

With the help of the FKG inequality, this result would be an easy application of the lower bound in Theorem~\ref{thm:crossing free} together with the RSW theorem. In the context of the double random current, one is forced to redo the whole proof as FKG is not available.

\begin{proof}Let $\partial_r^\square(ab)$ and $\partial_r^\square(cd)$ be the sets of boxes $B=\Lambda_r(x)$ with $x\in r\mathbb Z^2$ such that $\Lambda_{r}(x)\subset \Omega^\square_r$ and $\Lambda_{3r}(x)$ intersects $(ab)$ and $(cd)$ respectively. The idea is to replace the variables $N$, $\overline N$, $\underline N^\circ$, and $N_\ep^\circ$ by random variables $M$, $\overline M$, $\underline M^\circ$, and $M_\ep^\circ$ defined as 
\begin{align}
M&:=\sum_{B\in \partial_r^\square(ab)}\sum_{B'\in\partial_r^\square(cd)}\mathbb I [B\stackrel{\n_1+\n_2}\longleftrightarrow  B'],\\
\overline M&:=\sum_{B\in \partial_r^\square(ab)}\sum_{B'\in\partial_r^\square(cd)}\mathbb I [\overline B\stackrel{\n_1+\n_2}\longleftrightarrow  \overline B'],\\
\underline M^\circ&:=\sum_{B\in \partial_r^\square(ab)}\sum_{B'\in\partial_r^\square(cd)}\mathbb I[\underline B\stackrel{\n_1+\n_2}\longleftrightarrow  \underline B',\mathcal A(B),\mathcal A(B')],\\
M_\ep^\circ&:=\sum_{B\in \partial_r^\square(ab)}\sum_{B'\in\partial_r^\square(cd)}\mathbb I[\mathcal E_\ep(B,B')\cap\mathcal A(B)\cap\mathcal A(B')],
\end{align}
where $\mathcal A(B)$ is defined in \eqref{eq:def A} and 
\begin{equation}
\mathcal E_\ep^\circ(B,B'):=\{B\stackrel{\n_1+\n_2}\longleftrightarrow B'\}\cap \{\mathbf P^{\emptyset,\emptyset}_{\Omega,\Omega}[\underline B\stackrel{\n_1+\n_2}\longleftrightarrow \underline B'|(\n_1+\n_2)_{|(B\cup B')^c}]\ge \ep\}.
\end{equation}
Then, one may follow the lines of the proofs of Lemma~\ref{lem:expected number} and use the $c_0$-regularity assumption to show that 
\[
\mathbf E_{\Omega,\Omega}^{\emptyset,\emptyset}[M_\ep^\circ]\ge c
\]
 and use the upper bound in Theorem~\ref{thm:crossing free} to show that 
 \[
 \mathbf E_{\Omega,\Omega}^{\emptyset,\emptyset}[(M_\ep^\circ)^2]\le C\log(R/r)^2
 \] 
 (we use that the distance between $(ab)$ and $(cd)$ is larger than or equal to $c_1R$). The result follows from the Cauchy-Schwarz inequality.
\end{proof}

\section{Absence of thick pivotal points: Proofs of Theorems~\ref{thm:absence closed pivotal expectation white} and \ref{thm:absence closed pivotal expectation black}}\label{sec:pivotal points}

In this section, we focus on the proofs of Theorems~\ref{thm:absence closed pivotal expectation white} and \ref{thm:absence closed pivotal expectation black}.
We prove the first theorem in Section~\ref{sec:pivotal points expectation}, and the second in Section~\ref{sec:pivotal points theorem}.

\subsection{Proof of  Theorem~\ref{thm:absence closed pivotal expectation white}}\label{sec:pivotal points expectation}

We split the proof of the theorem into two lemmata. 

\begin{lemma}\label{prop:closed pivotal}
There exists $C>0$ such that for all $r,R$ with $R\ge r\ge1$,
 \begin{align}
{\bf P}_{\mathbb Z^2,\mathbb Z^2}^{\emptyset,\emptyset}[A_4^\square(r,R)]&\le C(\tfrac rR)^2.\label{eq:4-arm closed}
\end{align}
\end{lemma}

\begin{figure}
\begin{center}
			\includegraphics[scale=0.8]{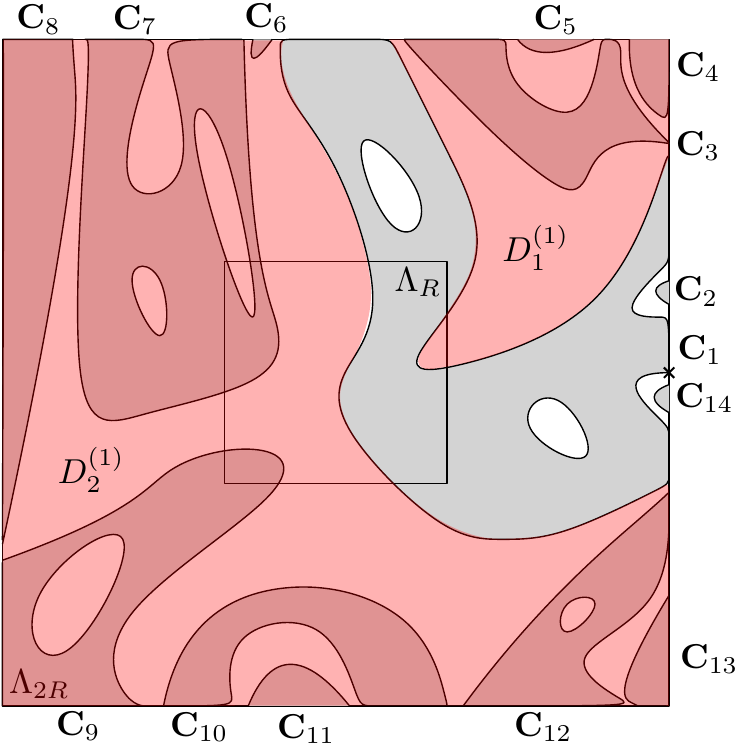}
		\end{center}
\caption{A picture of the clusters $\mathbf C_1,\dots,\mathbf C_{14}$ for which $n_1=1$. There are two connected components $D_2^{(1)}$ and $D_1^{(1)}$.
}\label{fig:A4}
\end{figure}

\begin{proof}
To lighten the notation, we prove the result for $3R$ instead of $R$. Below, denote by $\mathbf C_1,\dots,\mathbf C_n,\dots$ the $\Lambda_{2R}$-clusters in $\n_1+\n_2$ that intersect $\partial\Lambda_{2R}$ according to the smallest vertex in $\partial\Lambda_{2R}$ it contains, where the vertices on the boundary are indexed counterclockwise starting from $(2R,0)$. Also, let $\mathbf C_{\le n}$ be the union of the $\mathbf C_m$ for $m\le n$. We introduce $n_k$ to be the index of the $k$-th $\Lambda_{2R}$-cluster that contains a crossing of $\mathrm{Ann}(R,2R)$.

For a fixed $k$, let $\mathbf D_1^{(k)},\dots,\mathbf D^{(k)}_L$ be the connected components of $\Lambda_{2R}\setminus \mathbf C_{\le n_k}$ crossing $\mathrm{Ann}(R,2R)$ (see Fig.~\ref{fig:A4}). For each $\mathbf D^{(k)}_\ell$ with $\ell\le L$, let 
\[
\mathbf M^{(k)}_{\ell}:=|\{x\in r\mathbb Z^2\cap\Lambda_R:\Lambda_r(x)\cap \mathbf C_{\le n_k}\ne \emptyset\text{ and }\Lambda_r(x)\stackrel{\n_1+\n_2}\longleftrightarrow\partial\Lambda_{2R}\text{ in }\mathbf D^{(k)}_\ell\}|.
\] Note that $\sum_\ell\mathbf M^{(k)}_\ell$ is larger than or equal to the number of $x\in r\mathbb Z^2\cap\Lambda_R$ such that $\Lambda_r(x)$ intersects both some $\mathbf C_i$ with $i\le n_k$ and some $\mathbf C_j$ with $j>n_k$.

Now, Corollary~\ref{cor:upper bound crossing} implies easily that a.s.
\begin{equation}\label{eq:uhu0}
{\bf P}_{\mathbb Z^2,\mathbb Z^2}^{\emptyset,\emptyset}[\Lambda_r(x)\stackrel{\n_1+\n_2}\longleftrightarrow \partial\Lambda_{2R}\text{ in }\mathbf D_\ell^{(k)}|\mathbf C_{\le n_k}]\le C_0Z_{\mathbf D^{(k)}_\ell\setminus\Lambda_r(x)}[\Lambda_{r}(x),\mathbb Z^2\setminus\Lambda_{2R}].
\end{equation}
Summing over $x\in r\mathbb Z^2\cap \Lambda_R$ and interpreting the result in terms of the expected number of boxes $\Lambda_r(x)$ that are visited by a random walk starting from $\partial\Lambda_{2R}$ before exiting $\mathbf D_\ell^{(k)}$, we find that (we leave the details to the reader) a.s.
\begin{equation}\label{eq:uhu1}
{\bf E}_{\mathbb Z^2,\mathbb Z^2}^{\emptyset,\emptyset}[\mathbf M^{(k)}_{\ell}|\mathbf C_{\le n_k}]\le C_1Z_{\mathbf D^{(k)}_\ell}[\partial\mathbf D^{(k)}_\ell\cap\Lambda_R,\mathbb Z^2\setminus\Lambda_{2R}].
\end{equation}
Summing over every $\ell$ and averaging over the possible realizations of $\mathbf C_{\le n_k}$ gives that 
\[
{\bf E}_{\mathbb Z^2,\mathbb Z^2}^{\emptyset,\emptyset}\big[\sum_\ell\mathbf M^{(k)}_\ell\mathbf 1_{\{\mathbf C_{n_k}\text{ exists}\}}\big]\le C_2.
\] Summing over every $k$, we get that 
\[
{\bf E}_{\mathbb Z^2,\mathbb Z^2}^{\emptyset,\emptyset}[\mathbf N]\le C_2{\bf E}_{\mathbb Z^2,\mathbb Z^2}^{\emptyset,\emptyset}[\mathbf X],
\]
where $\mathbf N$ is the number of $x\in r\mathbb Z^2\cap\Lambda_R$ such that $A_4^\square(x,r,3R)$ occurs,  and $\mathbf X$ is the number of $\Lambda_{2R}$-clusters that contain a crossing of ${\rm Ann}(R,2R)$. We conclude the proof by showing that for $R/r\ge C$ with $C$ large enough,
\[
{\bf E}_{\mathbb Z^2,\mathbb Z^2}^{\emptyset,\emptyset}[\mathbf X]\le C_3,
\] 
which directly follows from the inequalities, for every $k\ge0$,
\begin{equation}\label{eq:uhu3}
{\bf P}_{\mathbb Z^2,\mathbb Z^2}^{\emptyset,\emptyset}[\mathbf X\ge k+1|\mathbf X\ge k]\le \tfrac12.
\end{equation}
To see the latter, observe that when $R/r$ is large enough, replacing $\Lambda_r(x)$ by $\Lambda_R$ in \eqref{eq:uhu1} and summing implies that a.s.
\begin{equation}\label{eq:uhu2}
{\bf P}_{\mathbb Z^2,\mathbb Z^2}^{\emptyset,\emptyset}[\Lambda_R\stackrel{\n_1+\n_2}\longleftrightarrow \partial\Lambda_{2R}|\mathbf C_{\le n_k}]\le\tfrac12.
\end{equation}
Averaging over the $\mathbf C_{\le n_k}$ concludes the proof of \eqref{eq:uhu3}.

Overall, we deduce that 
\[
(R/r)^2{\bf P}_{\mathbb Z^2,\mathbb Z^2}^{\emptyset,\emptyset}[A_4^\square(r,3R)]\le {\bf E}_{\mathbb Z^2,\mathbb Z^2}^{\emptyset,\emptyset}[\mathbf N]\le C_2C_3,
\]
which concludes the proof.\end{proof}

We now state an important corollary. For $\delta>0$, let ${\rm Sep}_\delta(r)$ be the event that there does not exist any $x\in \partial\Lambda_r$ such that $A_4^\square(x,\delta r,r/4)$ occurs. Note that on this event, clusters in $\n_1+\n_2$ of radius $r/4$ intersecting $\partial\Lambda_r$ are necessarily ``separated'' by a distance at least $\delta r$ on $\partial\Lambda_r$. 
\begin{corollary}[Separability of arms]\label{cor:sep}
For every $\ep>0$, there exists $\delta_0=\delta_0(\ep)>0$ such that for every $0<\delta\le \delta_0$ and $r\ge 1/\delta$,
\begin{align}\label{eq:bound G_j}
{\bf P}_{\mathbb Z^2,\mathbb Z^2}^{\emptyset,\emptyset}[{\rm Sep}_\delta(r)]\ge 1-\ep.
\end{align}
\end{corollary}
In words, this implies that typically, given a distance $r$, long clusters remain at a reasonable macroscopic distance from each other near $\partial\Lambda_r$. This will be a very convenient tool in the next proofs.

\begin{proof}
The proof is obvious using the union bound and \eqref{eq:4-arm closed}. 
\end{proof}

\begin{lemma}
For every $\ep>0$, there exists $\eta=\eta(\ep)>0$ such that for all $r,R$ such that $1\le r\le \eta R$ and every domain $\Omega\supset\Lambda_{2R}$,
\begin{align}
\label{eq:closed pivotal}{\bf P}_{\Omega,\Omega}^{\emptyset,\emptyset}[\exists x\in \Lambda_R:A_4^\square(x,r,R)]&\le \ep.
\end{align}
\end{lemma}

 Our goal here is to apply Theorem~\ref{thm:crossing free}. Roughly speaking, the idea is that if $A^\square(x,r,R)$ occurs, then conditioned on the first cluster, the second cluster should be connecting the $\eta R$-neighborhood of the first cluster to a box $\Lambda_{\kappa R}(y)$ with $1\gg\kappa\gg\eta$ which is far from the first cluster. This has small probability by Theorem~\ref{thm:crossing free}. Implementing this idea is not especially long, but slightly cumbersome, due to two small technicalities: first, one needs to be able to ``explore the first cluster'', in the sense that one should condition on it leaving sufficiently vast uncharted territories outside of it to apply Theorem~\ref{thm:crossing free}; and second, once this is done, one should be able to find $y$ such that the translate of the domain by $y$ satisfies the assumptions of Theorem~\ref{thm:crossing free} for $\kappa R$. To guarantee all these conditions, we introduce two families of events $E(y,z)$ and $F(y,z)$ and go through a few trivial manipulations to place ourselves in the right framework.

\begin{proof} We prove the result for $4R$ instead of $R$.
Consider $\kappa$ and $\eta$ to be fixed later (think of $1\gg\kappa\gg\eta>0$). Let $\rho:=\lfloor \kappa R\rfloor$ and recall that $r\le \eta R$. Below, we assume that $\eta\ll1$ so that in particular $r\ll R$. 
 Let 
 \[
 \mathbf N:=|\{u\in \rho\mathbb Z^2\cap\Lambda_{2R}:A^\square_4(u,2\rho,R/4)\}|,
 \] and introduce (see Fig.~\ref{fig:Annulus} on the left), for $y,z\in \rho\mathbb Z^2\cap\Lambda_{2R}$,
\begin{align*}
F(y,z)&:=\{\Lambda_\rho(y)\stackrel{\n_1+\n_2}\longleftrightarrow \text{$2r$-neighborhood of $\mathcal C(z,\rho)$}\text{ in $\Lambda_{2R}$}\},\\
E(y,z)&:=\{\Lambda_\rho(z)\stackrel{\n_1+\n_2}\longleftrightarrow\partial\Lambda_{3R}\}\cap\{\Lambda_{2\rho}(y)\stackrel{\n_1+\n_2}\nxlra\Lambda_\rho(z)\}\cap F(y,z),
\end{align*}
where $\mathcal C(z,\rho)$ is the union of the $\Lambda_{3R}$-clusters intersecting $\Lambda_\rho(z)$. 

At this stage, the introduction of the event $E(y,z)$ may seem like an unnecessary complication. The advantage of this event is that when conditioning on it, we will be able to first condition on the $\Lambda_{3R}$-clusters intersecting $\Lambda_\rho(z)$, then use the mixing property to remove all the potential sources on $\partial\Lambda_{3R}$ induced by this conditioning, and finally use crossing estimates in domains without sources to bound the probability that $\Lambda_\rho(y)$ is connected to the $2r$-neighborhood of the clusters intersecting $\Lambda_\rho(z)$.

First, we claim that 
\begin{equation}\label{eq:yg}
\{\exists x\in\Lambda_R:A_4^\square(x,r,4R)\}\subset\{\mathbf N>1/(8\kappa)\}\cup\Big(\bigcup_{y,z\in \rho\mathbb Z^2\cap\Lambda_{2R}}E(y,z)\Big).
\end{equation}
Indeed, assume that $A_4^\square(x,r,4R)$ occurs for some $x\in \Lambda_R$ and  $\mathbf N\le 1/(8\kappa)$. Consider two distinct $\Lambda_{3R}$-clusters ${\mathcal C}$ and ${\mathcal C}'$ that intersect both $\partial\Lambda_{3R}$ and $\Lambda_r(x)$:
\begin{itemize}
\item
First, there must exist $z\in \rho\mathbb Z^2\cap\Lambda_{2R}$ with $\Lambda_\rho(z)\cap{\mathcal C}\ne \emptyset$ and $\Lambda_\rho(z)\cap{\mathcal C}'=\emptyset$ since otherwise all the boxes $\Lambda_\rho(z)$ with $z\in \rho\mathbb Z^2$ that are intersected by $\mathcal C$ in ${\rm Ann}(R,2R)$ must also be intersected by $\mathcal C'$, but there are at least $R/\rho\ge 1/(8\kappa)$ such boxes since $\mathcal C$ contains a crossing from $\Lambda_R$ to $\partial\Lambda_{2R}$, which is not compatible with the bound on $\mathbf N$. 

\item Second, let $\underline{\mathcal C}'$ be {\em any} $\Lambda_{2R}$-cluster contained in ${\mathcal C}'$ and intersecting $\Lambda_r(x)$ and $\partial\Lambda_{2R}$. Since $\underline{\mathcal C}'\subset{\mathcal C}'$, there must exist $y\in  \rho\mathbb Z^2\cap\Lambda_{2R}$ with $\Lambda_\rho(y)\cap\underline{\mathcal C}'\ne \emptyset$ and $\Lambda_{2\rho}(y)\cap\mathcal C(z,\rho)=\emptyset$ otherwise there would be too many boxes for which $A_4^\square(y,2\rho,R/4)$ occurs. Indeed, note that the vertex $z$ is either in ${\rm Ann}(R,3R/2)$ or in ${\rm Ann}(3R/2,2R)$. Assume it is in the first case (the second one can be treated similarly), then any box $\Lambda_{2\rho}(y)$ with $y\in \rho\mathbb Z^2\cap {\rm Ann}(7R/4,2R)$ that is intersected by $\underline{\mathcal C}'$ and $\mathcal C(z,\rho)$ is such that $A_4^\square(y,2\rho,R/4)$ occurs. Yet, there are at least $\tfrac14 R/\rho\ge 1/(8\kappa)$ boxes such that $\Lambda_\rho(y)$ intersects $\underline{\mathcal C'}$. This would again be contradictory with the bound on $\mathbf N$.
\end{itemize}
The two paragraphs together imply that the event $E(y,z)$ occurs. 
 
We deduce from \eqref{eq:yg} that 
\begin{equation}\label{eq:ahy1}
{\bf P}_{\Omega,\Omega}^{\emptyset,\emptyset}[\exists x\in\Lambda_R:A_4^\square(x,r,3R)]\le {\bf P}_{\Omega,\Omega}^{\emptyset,\emptyset}[\mathbf N>1/(8\kappa)]+\sum_{y,z\in\rho\mathbb Z^2\cap\Lambda_{2R}}{\bf P}_{\Omega,\Omega}^{\emptyset,\emptyset}[E(y,z)].
\end{equation}
The first term on the right-hand side is bounded, using the mixing property \eqref{eq:independence in graph}, Markov's inequality and Lemma~\ref{prop:closed pivotal}, by
\[
{\bf P}_{\Omega,\Omega}^{\emptyset,\emptyset}[\mathbf N>1/(8\kappa)]\le8\kappa\, {\bf E}_{\Omega,\Omega}^{\emptyset,\emptyset}[\mathbf N]\le C_{\rm mix} 8\kappa\, {\bf E}_{\mathbb Z^2,\mathbb Z^2}^{\emptyset,\emptyset}[\mathbf N]\le C_0\kappa.
\]
For the second term, we bound ${\bf P}_{\Omega,\Omega}^{\emptyset,\emptyset}[E(y,z)]$ by first conditioning on  $\mathcal C(z,\rho)$. Then, for $E(y,z)$ to occur it must be that $F(y,z)$ does. Using the mixing property \eqref{eq:independence in graph} again to relate the probability in $\Omega\setminus\mathcal C(z,\rho)$ with the sources induced by the currents $\n_1$ and $\n_2$ on $\mathcal C(z,\rho)$ with the probability in $\Lambda_{3R}\setminus \mathcal C(z,\rho)$ without any source (we work in $\Lambda_{3R}\setminus\mathcal C(z,\rho)$ since the event that $\Lambda_\rho(y)$ is connected to the $2r$-neighborhood of $\mathcal C(z,\rho)$ in $\Lambda_{2R}$ depends on what happens inside $\Lambda_{2R}$ only).
In particular the connected component $\Omega'$ of $\Lambda_{3R}\setminus \mathcal C(z,\rho)$ containing $y$ has a connected boundary (remember that $\mathcal C(z,\rho)$ contains $\Lambda_\rho(z)$ and intersects $\partial\Lambda_{3R}$), and we are in a position to use Theorem~\ref{thm:crossing free} to bound the probability that $\Lambda_\rho(y)$ is connected in $\n_1+\n_2$ to $\partial_{2r}\Omega'\cap\Lambda_{2R}$. We deduce that 
\[
{\bf P}_{\Omega,\Omega}^{\emptyset,\emptyset}[E(y,z)]\le \epsilon(8\eta/\kappa),
\]
where $\epsilon(\cdot)$ is given by Theorem~\ref{thm:crossing free}.

Taking the union bound on $y,z$ and plugging the two previous displayed inequalities into \eqref{eq:ahy1}, we get that 
\[
{\bf P}_{\Omega,\Omega}^{\emptyset,\emptyset}[\exists x\in\Lambda_R:A_4^\square(x,r,R)]\le C_0\kappa+\frac{C_1}{\kappa^{4}}\epsilon(8\eta/\kappa)\]
and the claim follows by first setting $\kappa=\kappa(\ep)$ small enough and then setting $\eta=\eta(\kappa,\ep)$ small enough.
\end{proof}

\begin{figure}
\begin{center}
			\includegraphics[scale=0.95]{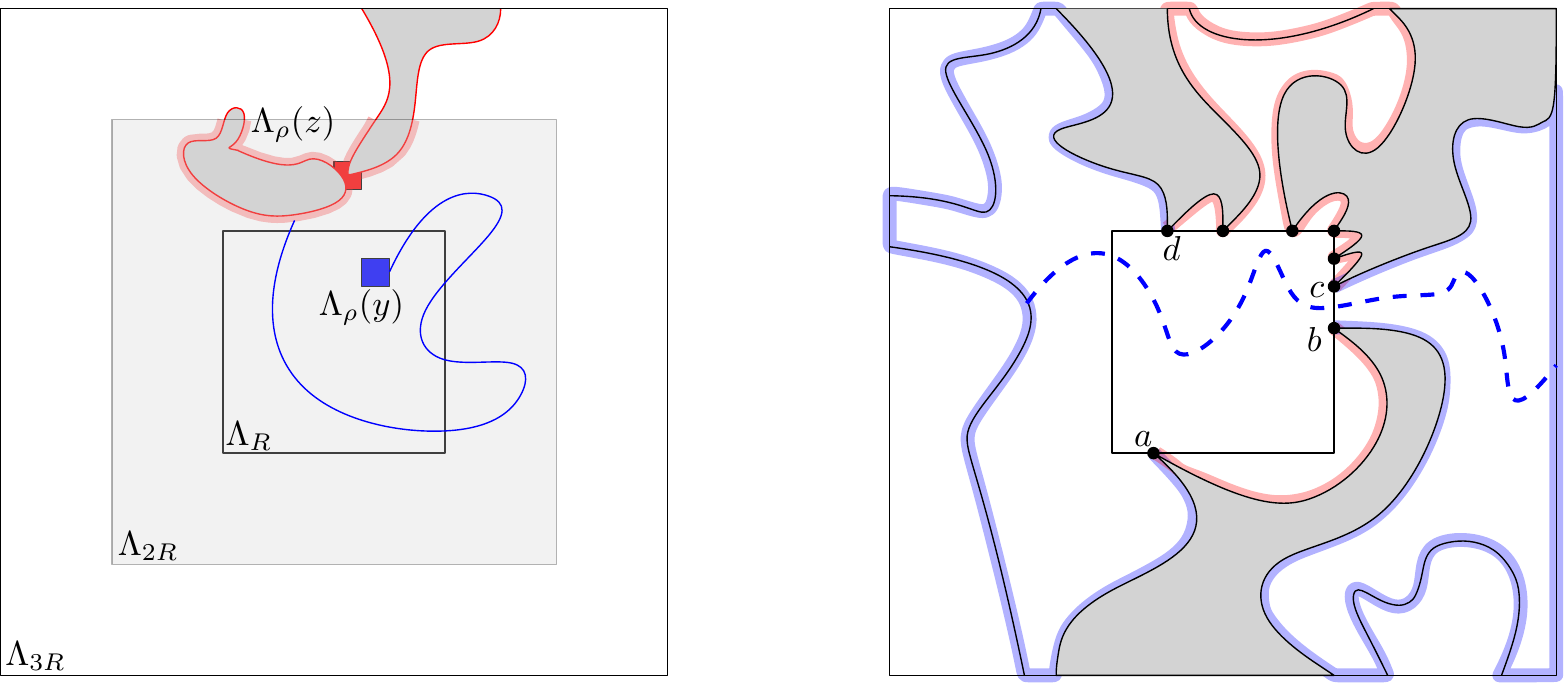}
		\end{center}
\caption{On the left, a depiction of the event $E(y,z)$. In grey, the union $\mathsf C$ of clusters in $\Lambda_{3R}$ intersecting $\Lambda_{\rho}(z)$ with its $r$-neighborhood in red, and the blue part denotes the path from $\Lambda_\rho(y)$ to the $r$-neighborhood of $\mathsf C$. On the right, the quad $D:=\Lambda_{R}\setminus \mathsf C$. In dashed we depicted a dual path crossing edges of zero $\n_1+\n_2$ current that prevents the existence of a path in $\n_1+\n_2$ between the two red arcs. On the event $H_j$, the extremal distance $\ell_D[(ab),(cd)]$ is bounded from above by $1/\kappa(\delta)$. }\label{fig:Annulus}
\end{figure}

\subsection{Proof of  Theorem~\ref{thm:absence closed pivotal expectation black}}\label{sec:pivotal points theorem}

In this section, define $r_j:=2^jr$ and $J:=\lfloor \log_2(R/r)\rfloor-1$. We will use the following sub-events of $A_r^\blacksquare(r,R)$ for $1\le r\le R$:
\begin{align*}
A_4^{\blacksquare\, \mathrm{odd}}(r,R)&:=\Big\{\begin{array}{c}\text{there exist two ${\rm Ann}(r,R)$-holes crossing ${\rm Ann}(r,R)$ and the shortest}\\
\text{dual path between them has even $(\n_1+\n_2)$-flux and odd $\n_1$-flux}\end{array}\Big\},\\
A_4^{\blacksquare\, \mathrm{even}}(r,R)&:=\Big\{\begin{array}{c}\text{there exist two ${\rm Ann}(r,R)$-holes crossing ${\rm Ann}(r,R)$ and the shortest}\\
\text{dual path between them has even $(\n_1+\n_2)$-flux and even $\n_1$-flux}\end{array}\Big\}.
\end{align*}
We have that 
\begin{equation}
A_4^\blacksquare(r,R)=A_4^\square(r,R)\cup A_4^{\blacksquare\, \mathrm{odd}}(r,R)\cup A_4^{\blacksquare\, \mathrm{even}}(r,R).\label{eq:decomposition A_4}
\end{equation}
Again, we split the proof into three lemmata and the proof of the theorem. We start by a lemma estimating the probability of $A_4^\blacksquare(r,R)$.
\begin{lemma}\label{prop:open pivotal}
There exists $C>0$ such that for all $r,R$ with $R\ge r\ge1$,
 \begin{align}
{\bf P}_{\mathbb Z^2,\mathbb Z^2}^{\emptyset,\emptyset}[A_4^\blacksquare(r,R)]&\le C(\tfrac rR)^2.\label{eq:4-arm open}
\end{align}
\end{lemma}

The proof is reminiscent of ``separation of arms'' arguments in percolation theory (see e.g.~\cite{CheDumHon16} for an example with the random cluster model): we consider the smallest scale $\rho$ at which large clusters are separated by a reasonable distance, and then we use our crossing estimates to express the probability of the event in terms of the one of $A_4^\square(\rho,R)$.

\begin{proof} 
 The switching principle (Lemma~\ref{lem:switch}) and the paragraph following it give that 
 \[
 {\bf P}_{\mathbb Z^2,\mathbb Z^2}^{\emptyset,\emptyset}[A_4^{\blacksquare\,  \rm even}(r,R)]={\bf P}_{\mathbb Z^2,\mathbb Z^2}^{\emptyset,\emptyset}[A_4^{\blacksquare\,  \rm odd}(r,R)].
 \]
By \eqref{eq:decomposition A_4}, we get that
\begin{align}
{\bf P}_{\mathbb Z^2,\mathbb Z^2}^{\emptyset,\emptyset}[A_4^\blacksquare\, (r,R)]&\le {\bf P}_{\mathbb Z^2,\mathbb Z^2}^{\emptyset,\emptyset}[A_4^\square(r,R)]+2{\bf P}_{\mathbb Z^2,\mathbb Z^2}^{\emptyset,\emptyset}[A_4^{\blacksquare\,  \rm even}(r,R)].\label{eq:h1}
\end{align}
By \eqref{eq:4-arm closed}, it suffices to focus on the bound of ${\bf P}_{\mathbb Z^2,\mathbb Z^2}^{\emptyset,\emptyset}[A_4^{\blacksquare\,  \rm even}(r,R)]$.
Consider $\ep>0$ such that $4C_{\rm mix}\,\ep<1$ ($C_{\rm mix}$ is given by \eqref{eq:mixing}) and let $\delta\in(0,\delta_0(\ep))$ with $\delta_0(\ep)$ given by Corollary~\ref{cor:sep}. 
By decomposing on the smallest $j$ at which ${\rm Sep}_\delta(r_j)$ occurs (there is a possibility that no such $j$ exists but this is taken into account by the second term in the next formula), we find that 
\begin{align}
{\bf P}_{\mathbb Z^2,\mathbb Z^2}^{\emptyset,\emptyset}[A_4^{\blacksquare\,  \rm even}(r,R)]&\le\sum_{j=1}^{J} {\bf P}_{\mathbb Z^2,\mathbb Z^2}^{\emptyset,\emptyset}\Big[{\rm Sep}_\delta(r_j),\bigcap_{\ell<j}{\rm Sep}_\delta(r_\ell)^c,A_4^{\blacksquare\,  \rm even}(r,R)\Big]\nonumber\\
&\qquad\qquad\qquad\qquad\qquad\qquad+{\bf P}_{\mathbb Z^2,\mathbb Z^2}^{\emptyset,\emptyset}\Big[\bigcap_{\ell\le J}{\rm Sep}_\delta(r_\ell)^c,A_4^{\blacksquare\,  \rm even}(r,R)\Big].\label{eq:h11}
\end{align}
First of all, since the events ${\rm Sep}_\delta(r_\ell)$ depend on well-separated regions, the mixing property implies that the second term on the right-hand side is bounded by
\begin{align}\label{eq:sep1}
{\bf P}_{\mathbb Z^2,\mathbb Z^2}^{\emptyset,\emptyset}\Big[\bigcap_{\ell\le J}{\rm Sep}_\delta(r_\ell)^c\Big]\le \Big(\prod_{\ell< J}{\bf P}_{\mathbb Z^2,\mathbb Z^2}^{\emptyset,\emptyset}[{\rm Sep}_\delta(r_\ell)^c]\Big)\Big(\prod_{\ell<J}C_{\rm mix}\Big)\le C_{\rm mix}^{J-1}\ep^J.
\end{align}
We now refer to Fig.~\ref{fig:Annulus} on the right. On $A_4^{\blacksquare\,  \rm even}(r,R)\cap {\rm Sep}_\delta(r_j)$, let $\mathsf C$ be the union of the $\mathrm{Ann}(r_j,R)$-clusters intersecting $\partial\Lambda_R$, and note that there must exist two arcs $(ab)$ and $(cd)$ of $\partial\Lambda_{r_j}$ separated by a distance at least $\delta r_j$ with the property that the $\mathrm{Ann}(r_j,R)$-clusters crossing $\mathrm{Ann}(r_j,R)$ end entirely either on $(ab)$ or $(cd)$, and that the $\n_1$- and $\n_2$-fluxes of the {\em union}\footnote{We do not claim this for individual $\mathrm{Ann}(r_j,R)$-clusters.} of the $\mathrm{Ann}(r_j,R)$-clusters ending in $(ab)$ is even (and therefore also for $(cd)$).
Denote the event that this happens by $H_j$ (note that it is not equal to $A_4^{\blacksquare\,  \rm even}(r,R)\cap {\rm Sep}_\delta(r_j)$ as some configurations could satisfy $H_j$ but not $A_4^{\blacksquare\,  \rm even}(r,R)\cap {\rm Sep}_\delta(r_j)$). 

 Since $H_j$ depends on the outside of $\Lambda_{r_j}$ only, and $\bigcap_{\ell<j}{\rm Sep}_\delta(r_\ell)^c$ on the inside of $\Lambda_{3r_j/4}$, we deduce from the mixing property \eqref{eq:mixing} and a bound similar to \eqref{eq:sep1} that 
\begin{align*}
{\bf P}_{\mathbb Z^2,\mathbb Z^2}^{\emptyset,\emptyset}\Big[{\rm Sep}_\delta(r_j),\bigcap_{\ell<j}{\rm Sep}_\delta(r_\ell)^c,A_4^{\blacksquare\,  \rm even}(r,R)\Big]&\le {\bf P}_{\mathbb Z^2,\mathbb Z^2}^{\emptyset,\emptyset}\Big[\bigcap_{\ell<j}{\rm Sep}_\delta(r_\ell)^c,H_j\Big]\\
&\le C_{\rm mix}^j\ep^{j-1}{\bf P}_{\mathbb Z^2,\mathbb Z^2}^{\emptyset,\emptyset}[H_j].
\end{align*}
On $H_j$, one may condition on $\mathsf C$ and everything outside of $\Lambda_R$ (recall that $\mathsf C$ is the union of the $\mathrm{Ann}(r_j,R)$-clusters intersecting $\partial\Lambda_R$). Then, in the complement $D:=\Lambda_{R}\setminus\mathsf C$ of the explored edges, we can find four points $a,b,c,d\in\partial\Lambda_{r_j}$ such that the two currents have an even set of sources on $(ab)$ and $(cd)$, and no source elsewhere. Since these two sets are at a distance at least $\delta r_j$ of each other, Corollary~\ref{cor:upper bound crossing close 1} applied to $(D,a,b,c,d)$ shows that there exists $c_1(\delta)>0$ independent of everything except $\delta$ such that a.s.
\[
{\bf P}_{\mathbb Z^2,\mathbb Z^2}^{\emptyset,\emptyset}[(ab)\stackrel{\n_1+\n_2}\nxlra(cd)\text{ in }D|H_j,\mathsf C]\ge c_1(\delta).
\]
In particular, on this event $A_4^\square(r_j,R)$ occurs, see Fig.~\ref{fig:Annulus} on the right again. We deduce from \eqref{eq:4-arm closed} that 
$${\bf P}_{\mathbb Z^2,\mathbb Z^2}^{\emptyset,\emptyset}[H_j]\le C_1(\delta){\bf P}_{\mathbb Z^2,\mathbb Z^2}^{\emptyset,\emptyset}[A_4^\square(r_j,R)]\le C_2(\delta)(\tfrac{r_j}R)^2.$$
Overall, we find that 
$${\bf P}_{\mathbb Z^2,\mathbb Z^2}^{\emptyset,\emptyset}\Big[{\rm Sep}_\delta(r_j),\bigcap_{\ell<j}{\rm Sep}_\delta(r_\ell)^c,A_4^{\blacksquare\,  \rm even}(r,R)\Big]\le C_2(\delta)C_{\rm mix}^j\ep^{j-1} (\tfrac {r_j}R)^2.$$
Plugging this bound into \eqref{eq:h11}, summing the estimate and then plugging it into \eqref{eq:h1} implies the result (we use that $r_j=2^jr$ and our assumption that $4C_{\rm mix}\,\ep<1$).
\end{proof}

 We now turn to the second lemma.
For $\delta>0$ and $\rho\ge r$, define
\[
F_\delta(\rho):=\{ \exists x\in  \Lambda_{r_j}:A_4^{\blacksquare\, \rm odd } (x,r,3\rho)\}\cap {\rm Sep}_\delta(\rho).
\]

\begin{lemma}\label{lem:claim 2} For every $\ep>0$, there exist $\delta_1=\delta_1(\ep)>0$ and $c_{\rm exists}(\ep)>0$ such that for all $r,\rho,R$ with $r\le \rho\le R/2$ and every $0<\delta\le\delta_1$,
\begin{equation}\label{eq:h6} {\bf P}_{\mathbb Z^2,\mathbb Z^2}^{\emptyset,\emptyset}[\exists x\in \Lambda_{\rho}:A_4^{\blacksquare\, \rm odd}(x,r,R)]\ge \ep(\rho/R)^2\qquad\Longrightarrow\qquad {\bf P}_{\mathbb Z^2,\mathbb Z^2}^{\emptyset,\emptyset}[F_\delta(\rho)]\ge c_{\rm exists}(\ep).\end{equation}
\end{lemma}

\begin{proof}
Fix $\ep>0$ and let 
\[
\delta_1(\ep):=\delta_0\big(\tfrac{\ep}{50C_{\rm mix}C}\big)\qquad\text{and}\qquad c_{\rm exists}(\ep):=\tfrac{\ep}{50C_{\rm mix}},
\]
where $C$ is given by Lemma~\ref{prop:open pivotal}, $C_{\rm mix}$ by \eqref{eq:mixing} and $\delta_0(\cdot)$ is given by Corollary~\ref{cor:sep}. 
For $\delta\in(0,\delta_1)$, note that 
\[
A_4^{\blacksquare\, \rm odd}(x,r,R)\subset A_4^{\blacksquare\, \rm odd}(x,r,3\rho)\cap A_4^{\blacksquare\, \rm odd}(x,5\rho,R)\] 
and that the events $A_4^{\blacksquare\, \rm odd}(x,r,3\rho)$ and $A_4^{\blacksquare\, \rm odd}(x,5\rho,R)$ depend respectively on edges in ${\rm Ann}(x,r,3\rho)$ and ${\rm Ann}(x,5\rho,R)$ only (the existence of the two holes and the fluxes can be determined looking only at the annuli). 

This implies that
\begin{align*}\ep (\tfrac {\rho}{R})^2&\le {\bf P}_{\mathbb Z^2,\mathbb Z^2}^{\emptyset,\emptyset}[F_\delta(\rho)\cap A_4^{\blacksquare\,  \rm odd}(5\rho,R)]+{\bf P}_{\mathbb Z^2,\mathbb Z^2}^{\emptyset,\emptyset}[{\rm Sep}_\delta(r_j)^c\cap A_4^{\blacksquare\, \rm odd}(5\rho,R)]\\
&\le C_{\rm mix}\big({\bf P}_{\mathbb Z^2,\mathbb Z^2}^{\emptyset,\emptyset}[F_\delta(\rho)]+ {\bf P}_{\mathbb Z^2,\mathbb Z^2}^{\emptyset,\emptyset}[{\rm Sep}_\delta(\rho)^c]\big) \times{\bf P}_{\mathbb Z^2,\mathbb Z^2}^{\emptyset,\emptyset}[A_4^{\blacksquare\, \rm odd}(5\rho,R)]\\
&\le C_{\rm mix}\Big({\bf P}_{\mathbb Z^2,\mathbb Z^2}^{\emptyset,\emptyset}[F_\delta(\rho)]+\frac{\ep}{50C_{\rm mix}C}\Big) \times C(\tfrac{5\rho}R)^2,
\end{align*}
where we used the union bound in the first, the mixing property \eqref{eq:mixing} in the second, and Corollary~\ref{cor:sep} and Lemma~\ref{prop:open pivotal} in the third.
 This implies \eqref{eq:h6} readily.
\end{proof}

The third lemma complements the previous one.

\begin{lemma}\label{lem:claim 3} For every $\ep>0$ and all $r,\rho,R$ with $r\le \rho\le R/2$,
\[
{\bf P}_{\mathbb Z^2,\mathbb Z^2}^{\emptyset,\emptyset}[\exists x\in \Lambda_{\rho}:A_4^{\blacksquare\, \rm odd}(x,r,R)]\le \ep(\tfrac\rho R)^2\quad\Longrightarrow\quad{\bf P}_{\mathbb Z^2,\mathbb Z^2}^{\emptyset,\emptyset}[\exists x\in \Lambda_R:A_4^{\blacksquare\, \rm odd}(x,r,R)]\le \ep.
\]
\end{lemma}

\begin{proof}
Cover $\Lambda_R$ by boxes $\Lambda_{\rho}(y)$ with $y\in T:=(2\rho+1)\mathbb Z^2\cap \Lambda_R$, the union bound gives
\begin{align*}
{\bf P}_{\mathbb Z^2,\mathbb Z^2}^{\emptyset,\emptyset}[\exists x\in \Lambda_R:A_4^{\blacksquare\, \rm odd}(x,r,R)]&\le\sum_{y\in T}{\bf P}_{\mathbb Z^2,\mathbb Z^2}^{\emptyset,\emptyset}[\exists x\in \Lambda_{\rho}(y):A_4^{\blacksquare\, \rm odd}(x,r,R)].\end{align*}
The invariance under translation implies the result.
\end{proof}
We are now in a position to prove Theorem~\ref{thm:absence closed pivotal expectation black}. 
\begin{proof}[Proof of Theorem~\ref{thm:absence closed pivotal expectation black}] 
Lemma~\ref{prop:open pivotal} gives \eqref{eq:open pivotal} so we focus on \eqref{eq:open pivotal existence}.  Using the mixing property \eqref{eq:independence in graph}, we replace $\Omega$ by $\mathbb Z^2$. Also, the switching principle (Lemma~\ref{lem:switch}) and the paragraph following it imply that
\begin{align*}{\bf P}_{\mathbb Z^2,\mathbb Z^2}^{\emptyset,\emptyset}[\exists x\in \Lambda_R:A_4^{\blacksquare\, \rm even}(x,r,R)]=  {\bf P}_{\mathbb Z^2,\mathbb Z^2}^{\emptyset,\emptyset}[\exists x\in \Lambda_R:A_4^{\blacksquare\, \rm odd} (x,r,R)].\end{align*}
By \eqref{eq:decomposition A_4} and \eqref{eq:closed pivotal}, 
\begin{align*}
{\bf P}_{\mathbb Z^2,\mathbb Z^2}^{\emptyset,\emptyset}[\exists x\in \Lambda_R:A_4^{\blacksquare}(x,r,R)]\le {\bf P}_{\mathbb Z^2,\mathbb Z^2}^{\emptyset,\emptyset}&[\exists x\in \Lambda_R:A_4^\square (x,r,R)]\\
&+2 {\bf P}_{\mathbb Z^2,\mathbb Z^2}^{\emptyset,\emptyset}[\exists x\in \Lambda_R:A_4^{\blacksquare\, \rm odd} (x,r,R)]
\end{align*}
and it suffices to bound the probability on the last line. Fix $\ep>0$. Either the quantity is bounded by $\ep$ and we are done, or by 
a combination of Lemmata~\ref{lem:claim 2} and \ref{lem:claim 3} we may assume that 
\begin{equation} \label{eq:h2}
(\mathcal H_\ep)\qquad\qquad\forall j\le J,\qquad {\bf P}_{\mathbb Z^2,\mathbb Z^2}^{\emptyset,\emptyset}[F_\delta(r_j)]\ge c_{\rm exists}(\ep),
\end{equation}
where $c_{\rm exists}(\ep)$ is given by Lemma~\ref{lem:claim 2}.
We now work under the assumption that $(\mathcal H_\ep)$ holds true. (Let us remark that a posteriori it is not true, and that we are therefore assuming something wrong.) 

Under $(\mathcal H_\ep)$, we will show that conditioned on having at least one point for which $A^\blacksquare_4(x,r,R)$ occurs, there are in fact many other places where the (translate of the) event does as well. 
 The underlying idea is reminiscent of the upper bound of Theorem~\ref{thm:crossing free}: we already know that the expected number of $x\in \Lambda_R\cap r\mathbb Z^2$ for which $A_4^{\blacksquare}(x,2r,R/2)$ occurs\footnote{The fact that we consider $R/2$ instead of $R$ is due to the following observation: the existence of $x_0\in \Lambda_R$ such that $A_4^\square(x_0,r,R)$ occurs implies the existence of $x\in \Lambda_R\cap r\mathbb Z^2$ such that $A_4^\square(x,2r,R/2)$ does.} is of order 1, so we only need to prove that the probability that there is such an $x$, but not too many other $y$ satisfying $A_4^{\blacksquare}(y,2r,R/4)$ \footnote{The fact that we consider $R/4$ instead of $R/2$ is due to the fact that in the construction of Claim 1 below, it will be useful to have $R/4$ instead of $R/2$.}, is quite small. More precisely, we will show that conditioned on $A_4^\blacksquare(x,2r,R/2)$, the number of diadic scales around $x$ that contain some $y\in r\mathbb Z^2$ with $A_4^\blacksquare(y,2r,R/4)$ occurring is typically large.

Introduce the event 
\[
H_j:=\{\exists x\in r\mathbb Z^2\cap \mathrm{Ann}(r_{j},2r_{j})\text{ such that }A_4^{\blacksquare\, \rm odd} (x, 2r ,R/4)\text{ occurs}\}
\] and the random variables
\begin{align*}
\mathbf N&:=|\{ x\in r\mathbb Z^2\cap \Lambda_R:A_4^{\blacksquare\, \rm odd}(x,2r,R/4)\text{ occurs}\}|,\\
\mathbf M&:=|\{j\le J:H_j \text{ occurs}\}|.\end{align*} 
The random variable $\mathbf N$ ``counts'' the number of places where $A_4^{\blacksquare\, \rm odd}(x,2r,R/4)$ occurs, while $\mathbf M$ is centred on the vertex 0 and counts the number of scales in which there is a vertex $y$ such that $A_4^{\blacksquare\, \rm odd}(y,2r,R/4)$ occurs. We also introduce $\mathbf M_x$ to be the translate of $\mathbf M$ by $x$.

Markov's inequality (like in the proof of Theorem~\ref{thm:crossing free}) implies that
\begin{align}
{\bf P}_{\mathbb Z^2,\mathbb Z^2}^{\emptyset,\emptyset}[\exists x\in \Lambda_R:A_4^{\blacksquare\, \rm odd}(x,r,R)]
&\le \ep{\bf E}_{\mathbb Z^2,\mathbb Z^2}^{\emptyset,\emptyset}[\mathbf N]+\sum_{x\in r\mathbb Z^2\cap \Lambda_R}{\bf P}_{\mathbb Z^2,\mathbb Z^2}^{\emptyset,\emptyset}[A_4^{\blacksquare\, \rm odd}(x,2r,R/2),\mathbf N\le \tfrac1\ep]\nonumber\\
&\le C\ep+\sum_{x\in r\mathbb Z^2\cap \Lambda_R}{\bf P}_{\mathbb Z^2,\mathbb Z^2}^{\emptyset,\emptyset}[A_4^{\blacksquare\, \rm odd}(x,2r,R/2),\mathbf M_x\le \tfrac1\ep]\nonumber\\
&\le C\ep+C(\tfrac R r)^2{\bf P}_{\mathbb Z^2,\mathbb Z^2}^{\emptyset,\emptyset}[A_4^{\blacksquare\, \rm odd}(2r,R/2),\mathbf M\le \tfrac1\ep],\label{eq:fund1}\end{align}
where in the second inequality we used Lemma~\ref{prop:open pivotal} for $1\le r\le R$ to bound the expectation of $\mathbf N$ and we used that $\mathbf M_x\le \mathbf N$, and in the last one we invoked the invariance under translation. 

It only remains to prove that for $R/r$ large enough (how large it must be depends on $\ep$),
\[
{\bf P}_{\mathbb Z^2,\mathbb Z^2}^{\emptyset,\emptyset}[A_4^{\blacksquare\, \rm odd}(2r,R/2),\mathbf M\le \tfrac1\ep]\le \ep(r/R)^2.
\]
In order to do that, we implement the following reasoning, sometimes referred to as a ``multimap principle'', or ``energy-entropy comparison'', that we first present in a generic context. Assume that one wishes to bound the probability of the event $\mathbf E$. Then, one may try to find two constants $\mathbf C,\mathbf K$, a set $\mathbf I$ and a family of events $(\mathbf E_\mathbf i)_{\mathbf i\in \mathbf I}$ included in an event $\mathbf F$ such that
\begin{itemize}[noitemsep,nolistsep]
\item[(i)] for every $\mathbf i\in \mathbf I$, $\mathbb P[\mathbf E]\le \mathbf C\mathbb P[\mathbf E_\mathbf i
]$;
\item[(ii)] the maximal number of $\mathbf i\in \mathbf I$ to which a given element of $\mathbf F$ can belong to is bounded by $\mathbf K$.
\end{itemize}
Then, we get the bound
\[
\mathbb P [\mathbf E]\le \frac{\mathbf C\mathbf K}{|\mathbf I|}\mathbb P [\mathbf F]
\]
from the chain of straightforward inequalities
\[
|\mathbf I|\mathbb P [\mathbf E]\le \sum_{\mathbf i\in \mathbf I}\mathbb P [\mathbf E]\le C\sum_{\mathbf i\in \mathbf I}\mathbb P[\mathbf E_{\mathbf i}]\le C\mathbb E[\#\{\mathbf i\in \mathbf I:\omega\in \mathbf E_{\mathbf i}\}\mathbb I_{\omega\in \mathbf F}]\le CK\mathbb P[\mathbf F].
\]

In our context, we will take 
\begin{align*}
\mathbf E&:=\{A_4^{\blacksquare\, \rm odd}(2r,R/2),\mathbf M\le \tfrac1\ep\},\\
\mathbf F&:=\{A_4^{\blacksquare\, \rm odd}(2r,R/2),\mathbf M\le \tfrac4\ep\},\\
\mathbf I&:=\{(j_1,\dots,j_k)\in\mathbb N^k:1<j_1<\dots<j_k<J\}\qquad\text{ with }k:=\lfloor1/\ep\rfloor,\\
\mathbf E_{\mathbf i}&:=\{A_4^{\blacksquare\, \rm odd}( 2r ,R/2),\mathbf M\le \tfrac4\ep,H_{j_1},\dots,H_{j_k}\}.
\end{align*}
By construction, one sees that (ii) occurs with
\[
\mathbf K:=\binom{\lfloor 4/\ep\rfloor}{k}
\] 
since for any $\omega$ with $\mathbf M\le 4/\ep$, there are at most $4/\ep$ indexes $j_i$ for which $H_{j_i}$ occurs.
Then, the following claim will be the equivalent of Property (i) (we state the estimate in a slightly more general context).

\paragraph{Claim 1} {\em Fix $\ep>0$ small enough. There exists $C_0=C_0(\ep)>0$ such that if we assume $(\mathcal H_\ep)$, then for every $\ell>0$ and every collection of integers $1<j_1<\dots<j_k<J$,}
\begin{equation}\label{eq:h5}
 {\bf P}_{\mathbb Z^2,\mathbb Z^2}^{\emptyset,\emptyset}[A_4^{\blacksquare\, \rm odd}( 2r ,R/2),\mathbf M\le \ell]\le C_0^k\,{\bf P}_{\mathbb Z^2,\mathbb Z^2}^{\emptyset,\emptyset}[A_4^{\blacksquare\, \rm odd}( 2r ,R/2),\mathbf M\le 3k+\ell,H_{j_1},\dots,H_{j_k}].
\end{equation}

By choosing $k=\ell$, the claim enables us to pick 
\[
\mathbf C:=C_0(\ep)^k.
\]
Overall, we deduce that for $r/R\le \eta$ small enough
\begin{align*}
 {\bf P}_{\mathbb Z^2,\mathbb Z^2}^{\emptyset,\emptyset}[A_4^{\blacksquare\, \rm odd}(2r,R/2),\mathbf M\le \tfrac1\ep]&\le \frac{C_0(\ep)^k\binom{\lfloor 4/\ep\rfloor}{k}}{\binom{J}{k}} {\bf P}_{\mathbb Z^2,\mathbb Z^2}^{\emptyset,\emptyset}[A_4^{\blacksquare\, \rm odd}(2r,R/2),\mathbf M\le \tfrac4\ep]\\
 &\le \frac{\ep}{16C} {\bf P}_{\mathbb Z^2,\mathbb Z^2}^{\emptyset,\emptyset}[A_4^{\blacksquare\, \rm odd}(2r,R/2)]\le \ep(r/R)^2,
\end{align*}
where the constant $C$ is defined in Lemma~\ref{prop:open pivotal}, the second inequality is due to the fact that $J\ge \lfloor \log_2(1/\eta)\rfloor-1$ and the last one to Lemma~\ref{prop:open pivotal}. This concludes the proof.
It only remains to prove Claim 1. The proof proceeds in a slightly similar way to proofs of arm-separation. We identify good scales at which clusters are well-separated (i.e.~at which ${\rm Sep}_\delta(r_j)$ and ${\rm Sep}_\delta(2r_j)$ occur). Then, we forget what happens outside of these scales to reconstruct four arms with the further requirement that $H_j$ occurs at the scales $j$ that we want.

\begin{proof}[Proof of Claim 1]
Fix $\ep>0$ small enough. We choose
\[
\delta:=\min\{\delta_1(\ep),\delta_0(\ep)\},
\] where $\delta_1(\ep)$ is given by Lemma~\ref{lem:claim 2} and $\delta_0(\ep)$ by Corollary~\ref{cor:sep}.
In the whole proof, fix  a collection of integers $1<j_1<\dots<j_k<J$ (recall that $J:=\lfloor \log_2(R/r)\rfloor-1$). 

Call an integer $j\in(0,J)$ {\em good} if ${\rm Sep}_\delta(r_j)$ and ${\rm Sep}_\delta(2r_j)$ occur. By definition, we decide that $0$ and $J$ are automatically good. We say that $j$ is {\em bad} when it is not good. Note that the choice of $\delta$ implies that the probability of being bad is smaller than $2\ep$ by Corollary~\ref{cor:sep}.

Decomposing on the first good integers strictly above and below each $j_i$ (they may be the same for different $j_i$), we get that
\begin{align}{\bf P}_{\mathbb Z^2,\mathbb Z^2}^{\emptyset,\emptyset}&[A_4^{\blacksquare\, \rm odd}( 2r ,R/2),\mathbf M\le\ell]\label{eq:hi0}\le \sum_{\vec j^\pm}\underbrace{{\bf P}_{\mathbb Z^2,\mathbb Z^2}^{\emptyset,\emptyset}[A_4^{\blacksquare\,  \rm odd}( 2r ,R/2),\mathbf M\le \ell,E_1(\vec j^\pm)]}_{(U_1)},\nonumber\end{align}
where the sum runs over the set of $\vec j^\pm:=(j_v^\pm:1\le v\le u)$ such that
\begin{itemize}[noitemsep]
\item
$j_1^-< j_1^+\le j_2^-< j_2^+\le \dots \le j_u^-< j_u^+$, 
\item for every $1\le v\le u$, there is no $i\in[1,k]$ such that $j_v^+<j_i<j_{v+1}^-$,
\item there is no $i\in[1,k]$ with $j_i<j_1^-$ or $j_i>j_u^+$,
\end{itemize}
and 
\[
E_1(\vec j^\pm):=\bigcap_{1\le v\le u} \{j_v^{\pm}\text{ is good},j\text{ is bad for every }j_v^-<j<j_v^+\}.
\]
More generally, for $1\le u'\le u$, set
\[
E_{u'}(\vec j^\pm):=\bigcap_{u'\le v\le u} \{j_v^{\pm}\text{ is good},j\text{ is bad for every }j_v^-<j<j_v^+\}.
\]
We will now bound $(U_1)$. Set $n:=r_{j_1^-}$ and $N:=r_{j_1^+}$. Also, let $s_1$ be the number of $i\in[1,k]$ such that $j_1^-\le j_i\le j_1^+$. 

Consider the event $F$ (see Fig.~\ref{fig:reconnect} for an illustration) that 
\begin{itemize}
\item $A_4^{\blacksquare\,  \rm odd}(2r,n)$ and $A_4^{\blacksquare\,  \rm odd}(2N ,R/2)$ occur,
\item the number of $j\ge j_1^+$ or $j\le j_1^-$ such that $H_j$ occurs is smaller than $\ell$, 
\item there exist vertices $x_1,y_1,x_2,y_2,x_3,y_3,x_4,y_4$  found in a counterclockwise order around $\partial\Lambda_n$ at a distance at least $\delta n$ of each other, such that each $\mathrm{Ann}(2r,n)$-cluster  crossing $\mathrm{Ann}(2r,n)$ intersects exactly one of the $(x_iy_i)$, and if $T_i^0$ denotes the union of these clusters intersecting $(x_iy_i)$, then $T_i^0$ has odd (resp.~even) $\n_1-$ and $\n_2-$fluxes for $i=1,3$ (resp.~$i=2,4$).
\item there exist vertices $x_1',y_1',x_2',y_2',x_3',y_3',x_4',y_4'$  found in a counterclockwise order around  $\partial\Lambda_{2N}$ at a distance at least $\delta N$ of each other, such that each cluster in $\mathrm{Ann}(2N,R/2)$ crossing $\mathrm{Ann}(2N,R/2)$ intersects exactly one of the $(x_i'y_i')$, and if $T_i^{s_1+1}$ denotes the union of these clusters intersecting $(x_iy_i)$, then $T_i^{s_1+1}$ has odd (resp.~even) $\n_1-$ and $\n_2-$fluxes for $i=1,3$ (resp.~$i=2,4$).
\item $E_2(\vec j^\pm)$ occurs.
\end{itemize}
Note that the event $F$ depends only on the state of edges inside $\Lambda_n$ and outside of $\Lambda_{2N}$. Using the definition of good integers, the mixing property \eqref{eq:mixing} and Corollary~\ref{cor:sep}, we get that\begin{align}\label{eq:hi1}
(U_1)\le C_{\rm mix}^{\lfloor (j_1^+-j_1^--1)/3\rfloor+1}\,(2\ep)
^{\lfloor (j_1^+-j_1^--1)/3\rfloor}\, {\bf P}_{\mathbb Z^2,\mathbb Z^2}^{\emptyset,\emptyset}[F].\end{align}
Now, define $F'_j$ to be the translate by the vector $z_j:=(\tfrac32r_j,0)$ of the event $F_\delta(r_{j-3})$ (see the definition just above Lemma~\ref{lem:claim 2}). Notice that $F_j'$ depends on the edges in the box $\Lambda_{r_j/4}(z_j)$ only. One may therefore use the mixing property \eqref{eq:mixing} ($s_1$ times) and $(\mathcal H_\ep)$  to get 
\[
[c_{\rm mix}c_{\rm exists}(\ep)]^{s_1}{\bf P}_{\mathbb Z^2,\mathbb Z^2}^{\emptyset,\emptyset}[F]
\le {\bf P}_{\mathbb Z^2,\mathbb Z^2}^{\emptyset,\emptyset}[F_{j_1}',\dots,F_{j_{s_1}}',F].
\]
Now, condition on the set of edges
\begin{itemize}[noitemsep]
\item in $\Lambda_n$ that are connected to $\Lambda_{n/2}$ in $\n_1+\n_2$;
\item outside $\Lambda_{2N}$ that are connected to $\mathbb Z^2\setminus\Lambda_{4N}$ in $\n_1+\n_2$;
\item in $\Lambda_{r_j/4}(z_j)$ that are connected to $\Lambda_{r_j/8}(z_j)$ for every $1\le j\le s_1$.
\end{itemize}
The definition of the events $F$ and $F'_j$ guarantees the existence, on $F\cap F'_{j_1}\cap\dots\cap F_{j_{s_1}}'$, of 
\begin{itemize}[noitemsep]
\item $(T_1^0,T_2^0,T_3^0,T_4^0)$ as above;
\item $(T_1^{s_1+1},T_2^{s_1+1},T_3^{s_1+1},T_4^{s_1+1})$ as above;
\item for every $1\le j\le s_1$, vertices $x_1^j,y_1^j,x_2^j,y_2^j,x_3^j,y_3^j,x_4^j,y_4^j$ found in a counterclockwise order around $\partial\Lambda_{r_j/4}(z_j)$ at a distance at least $\delta r_j/4$ of each other, such that each $\mathrm{Ann}(z_j,r_j/8,r_j/4)$-cluster  crossing $\mathrm{Ann}(z_j,r_j/8,r_j/4)$ intersects exactly one of the $(x_i^jy_i^j)$, and if $T_i^j$ denotes the union of these clusters intersecting $(x_i^jy_i^j)$, then $T_i^j$ has odd (resp.~even) $\n_1-$ and $\n_2-$fluxes for $i=1,3$ (resp.~$i=2,4$).
\end{itemize}
We may now use successive applications of Corollary~\ref{cor:upper bound crossing close 1} (this construction is tedious but fairly straightforward, see the caption of Fig.~\ref{fig:reconnect} for some details) to guarantee that with probability bounded from below by $C(\delta)^{-(s_1+2)} C_0^{-(j_1^+-j_1^-)}$, where $C_0>0$ is independent of everything else,
\begin{itemize}[noitemsep]
\item the only connections between some $T_i^j$ are from $T_1^j$ to $T_1^{j+1}$ and from $T_3^j$ to $T_3^{j+1}$ for some $0\le j\le s_1$,
\item the only possible places where $H_j$ occurs are $j\in\{j_1,\dots,j_{s_1}\}\cup\{j_1^-,j_1^+\}$.
\end{itemize}
We now use the crucial fact that we are working with $A_4^{\blacksquare \,\rm odd}$: the source parity on each $T_1^j$ and $T_3^j$ guarantees that in our case $T_1^j$ is connected to $T_1^{j+1}$ and $T_3^j$ to $T_3^{j+1}$ for every $0\le j\le s_1$. 
 We deduce  that the events $H_{j_i}$ occur for $1\le i\le s_1$  (we used that the constants for $A_4^\square$ around $x$ and $y$ are respectively taken to be $R/2$ and $R/4$). Overall, we get from the whole construction that
\begin{align*}
&{\bf P}_{\mathbb Z^2,\mathbb Z^2}^{\emptyset,\emptyset}[F_{j_1}',\dots,F_{j_{s_1}}',F]
\\
&\le C(\delta)^{s_1+2} C_0^{j_1^+-j_1^-}\underbrace{{\bf P}_{\mathbb Z^2,\mathbb Z^2}^{\emptyset,\emptyset}[A_4^{\blacksquare\,\rm odd}(2r,R/2),E_2(\vec j^\pm),\mathbf M\le s_1+2+\ell,H_{j_1},\dots,H_{j_{s_1}}]}_{(U_2)}.
\end{align*}
The crucial observation here is that $C_0$ is independent of everything, including $\delta$. 

We now set
\[
c(\ep):=C_0^3C_{\rm mix}2\ep\qquad\text{and}\qquad C(\ep):=\frac{C(\delta)^3}{c_{\rm mix}c_{\rm exists}(\ep)}.
\]
The three previous displayed equations lead to
\[
(U_1)
\le c(\ep)^{\lfloor (j_1^+-j_1^--1)/3\rfloor} C(\ep)^{s_1} (U_2).
\]
One may proceed similarly for $(U_2)$, $(U_3)$, etc and by induction obtain that 
\[
(U_1)
\le c(\ep)^{\sum \lfloor (j_i^+-j_i^--1)/3\rfloor} C(\ep)^k{\bf P}_{\mathbb Z^2,\mathbb Z^2}^{\emptyset,\emptyset}[A_4^{\blacksquare\, \rm odd}( 2r ,R/2),\mathbf M\le 3k+\ell,H_{j_1},\dots,H_{j_k}].
\]
It only remains to consider $\ep$ sufficiently small that $c(\ep)\le1/2$ and to sum over all possible values for $\vec j^\pm$ to get the result.
\end{proof}
All of this concludes the proof of our theorem.\end{proof}

\begin{figure}
		\begin{center}
			\includegraphics[scale=0.6]{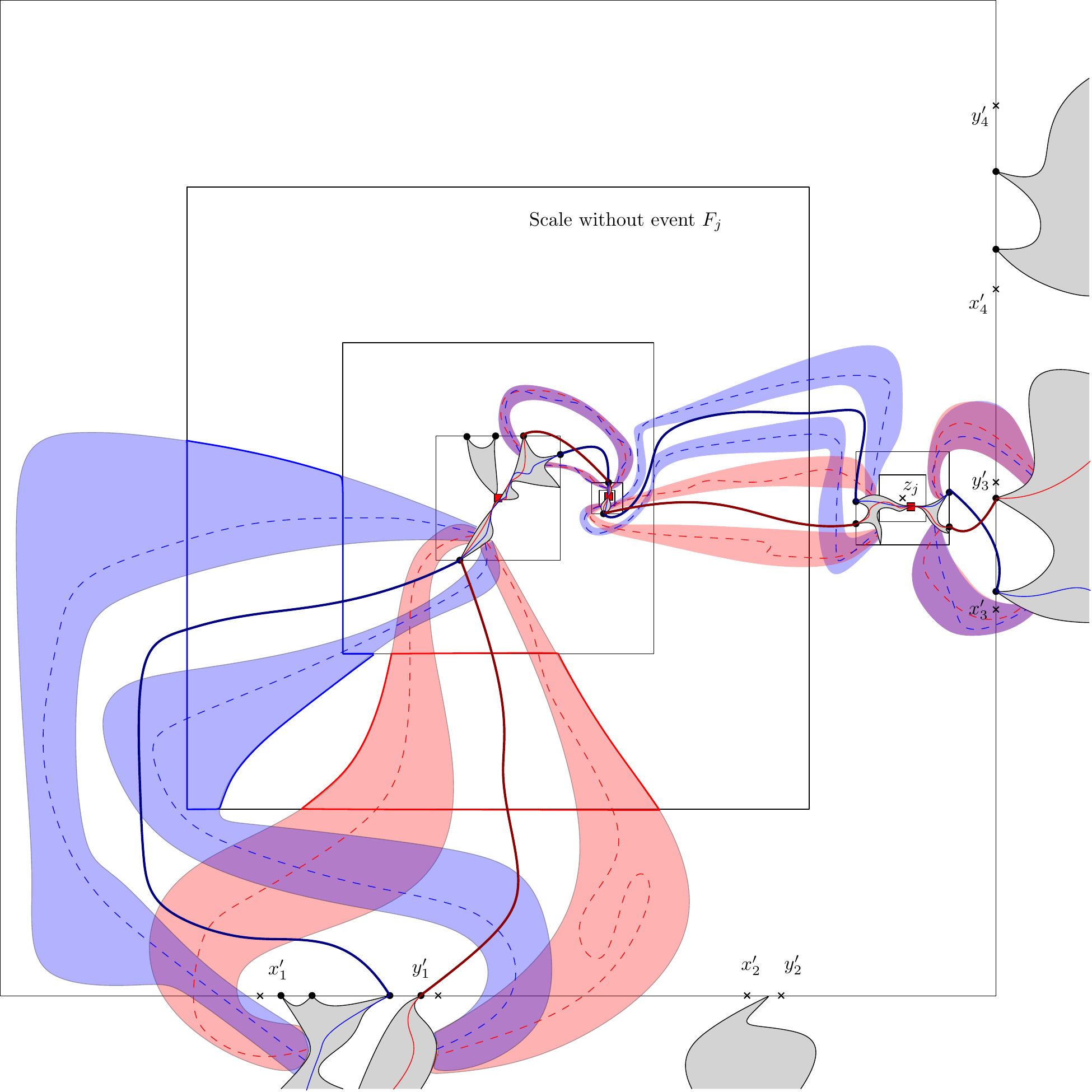}  
		\end{center}
		\caption{In blue, odd paths of $\n_1$, and in red, odd paths of $\n_2$. The sources of $\n_1$ and $\n_2$ are depicted with bullets. One enforces that none of the blue quads is crossed transversally by a path of $\n_1>0$, and similarly none of the red ones by a path of $\n_2>0$ (we depicted the absence of paths by dashed paths of the corresponding color). Note that some quads appear purple as they are both red and blue. The pictures for $\n_1$ and $\n_2$ are completely independent. Since one may find quads that are well-separated, one can prove that the cost of enforcing no crossing is of order constant to the number of scales. The non-existence of crossings forces the sources to be connected (for instance in the scale without event $F_j$ in the picture, there is a blue path going from the inner side of the blue quad to the outer side of it, and similarly for the red one); we depicted these paths in bold. To ensure that there is no other scale where there could be a four-arm event, we choose the quads in such a way that they force the existence of paths between sources in $\n_1$ and $\n_2$ that are macroscopically far from each other, like in the scale without event $F_j$. As a consequence of this choice, there cannot be any box in these scales that is intersecting both the path forced by the sources of the first current, and the one forced by the sources of the second current. Note that the scale without event $F_j$ could encompass a number of diadic scales.
		}	\label{fig:reconnect}
\end{figure}

\section{Proof of Theorem~\ref{prop:tight number crossings}}\label{sec:tightness}

In this section, we prove an Aizenman-Burchard \cite{AizBur99} criterion for the double random current model.

Roughly speaking, the proof of the theorem will consist in conditioning on ${\rm Ann}(r,R)$-clusters crossing ${\rm Ann}(r,R)$ and to show that conditioned on having $k$ such clusters, having an additional one  crossing ${\rm Ann}(r,R)$  costs at least $(r/R)^c$ using crossing estimates for the double random current. The problem with this strategy is that ${\rm Ann}(r,R)$-clusters  crossing ${\rm Ann}(r,R)$  may create sources on $\partial\Lambda_r$ and $\partial\Lambda_R$ that could force the existence of additional ${\rm Ann}(r,R)$-clusters (imagine for instance that after conditioning on the first ${\rm Ann}(r,R)$-cluster  crossing ${\rm Ann}(r,R)$, $\n_1$ or $\n_2$ has an odd number of sources on $\partial\Lambda_r$ and therefore also on $\partial\Lambda_R$). To go around this difficulty, we will first bound the probability of having many clusters of odd current in each $\n_i$ crossing ${\rm Ann}(r,s)$ and ${\rm Ann}(S,R)$ with $s$ and $S$ two intermediate integers that are chosen in such a way that the ratio $R/S$, $S/s$ and $s/r$ are roughly the same. The bound on the probability of having a certain number of such crossing clusters will be based on the interpretation of the odd part of a current as the low-temperature of a critical Ising model on the dual graph, and will not rely on crossing estimates for the double random current. Once it is proved that there are not too many crossing clusters of odd current with good probability, one may explore all the connected components of odd currents intersecting $\partial\Lambda_r$ and $\partial\Lambda_R$. Then, ${\rm Ann}(r,R)$-clusters crossing ${\rm Ann}(r,R)$ are of two types: either they intersect one of the clusters of odd current crossing ${\rm Ann}(r,s)$ or ${\rm Ann}(S,R)$, or if they do not they must contain a crossing of ${\rm Ann}(s,S)$ in the complement of what was explored, which in this case will necessarily be exempt of sources.  

We start with a lemma. Consider the set $\eta=\eta(\n)$ of odd edges of $\n$ and the event $B_{2k}(r,R)$ that there exist $k$ disjoint $\mathrm{Ann}(r,R)$-clusters in $\eta$ crossing ${\rm Ann}(r,R)$. 
\begin{lemma}[Aizenman-Burchard criterion for the odd part of one current]\label{lem:AB odd current}
There exist $C_0,\lambda_0>0$ such that for every $k\ge0$, every $\Omega$, and every $r,R$ with $1\le r\le R$,
\begin{align}\label{eq:ex}
{\bf P}_{\Omega}^{\emptyset}[B_{2k}(r,R)]\le (C_0\tfrac rR)^{k\lambda_0}.
\end{align}
\end{lemma}

\begin{proof}
For the purpose of the proof, let $B_0(r,R)$ be the full even
. The claim follows from the bound, for every $k\ge0$,
\begin{align}\label{eq:ex}
{\bf P}_{\Omega}^{\emptyset}[B_{2k+4}(r,R)|B_{2k}(r,R)]\le C_0(\tfrac rR)^{\lambda_0}.
\end{align}

In order to prove \eqref{eq:ex}, we work with the Ising model on $\Omega^*$.
Recall that, by Kramers-Wannier's duality (Proposition~\ref{prop:coupling Kramers Wannier}), $\eta$ can be seen as the low-temperature expansion of this Ising model. Below, spins refer to Ising spins of the Ising model on the dual graph. 

Order the vertices on $\partial\Lambda_R$ in counterclockwise order starting from $(R,0)$ and index the clusters in $\eta$ intersecting $\partial\Lambda_R$ according to the smallest vertex of $\partial\Lambda_R$ it contains. Condition on the first $k$ ${\rm Ann}(r,R)$-clusters crossing ${\rm Ann}(r,R)$ and let $\pmb\Omega$ be the set of unexplored vertices. In the dual Ising model, the boundary conditions on $\pmb\Omega^*$ are monochromatic on the two arcs of $\partial\pmb\Omega^*$ strictly inside ${\rm Ann}(r,R)$. We are now facing two possibilities:

\begin{itemize} 
\item If the arcs receive the same spin, say minus, then by comparison between boundary conditions for the Ising model, the probability that there is a path of minuses between these two arcs, and therefore no additional cluster of $\eta$ crossing ${\rm Ann}(r,R)$ from inside to outside, is bounded from below by the probability that there exists a circuit of minuses surrounding the origin in the Ising model on ${\rm Ann}(r,R)^*$ with plus boundary conditions. We saw in the footnote preceding \eqref{eq:uhug} that this probability is bounded by $1-C_0(\tfrac rR)^{\lambda_0}$.

\item If the two arcs receive different spins, say minus and plus, then explore the cluster of minuses connected to the minus arc. The boundary of this cluster is an additional path in $\eta$ crossing ${\rm Ann}(r,R)$. Yet, the exterior boundary of this cluster is a star-connected path of pluses in the Ising model and one can apply the same argument as in the previous item to prove the existence of a path of pluses connecting this arc to the plus arc, thus proving that there is no additional crossing of ${\rm Ann}(r,R)$ in $\eta$. Overall, we deduce that the probability that there are two disjoint clusters in $\eta$ crossing ${\rm Ann}(r,R)$ is bounded by $C_0(\tfrac rR)^{\lambda_0}$ again.
\end{itemize}
Overall, this gives \eqref{eq:ex} for $k\ge2$. For $k=1$ one can use the estimate in the footnote preceding \eqref{eq:uhug} directly.
\end{proof}

\begin{proof}[Proof of Theorem~\ref{prop:tight number crossings}]
Fix two intermediary integers $s,S$ satisfying that $R/S$, $S/s$, and $s/r$ are all larger than $\lfloor (R/r)^{1/3}\rfloor $. We recommend to take a look at Fig.~\ref{fig:karmevent}.
Assume that both $\n_1$ and $\n_2$ do not belong to $B_{2k}(r,s)\cup B_{2k}(S,R)$. 
Let $\mathsf C_i=\mathsf C_i(\eta_i)$ be the union of the ${\rm Ann}(r,R)$-clusters in $\eta_i$ crossing $\mathrm{Ann}(r,s)$ or $\mathrm{Ann}(S,R)$. Also, let $\mathsf C=\mathsf C(\n_1+\n_2)$ be the union of all the $\mathrm{Ann}(r,R)$-clusters in $\n_1+\n_2$ intersecting $\mathsf C_1\cup\mathsf C_2$. By definition, the sets $\mathsf C_1$ and $\mathsf C_2$ each contain at most $k$ clusters crossing $\mathrm{Ann}(r,s)$ and $k$ clusters crossing $\mathrm{Ann}(S,R)$ and therefore $\mathsf C$ contains at most $4k$ disjoint clusters in $\n_1+\n_2$ crossing $\mathrm{Ann}(r,R)$. 

Condition on all the connected components in $\eta_1$ and $\eta_2$ intersecting $\partial\mathrm{Ann}(r,R)$, and then on the $\mathrm{Ann}(r,R)$-clusters of $\n_1+\n_2$ intersecting $\mathsf C_1\cup\mathsf C_2$. Let $E_i$ be the set of vertices of $\Omega$ that are not connected to $\partial\mathrm{Ann}(r,R)$ in $\eta_i$ and do not belong to $\mathsf C(\n_1+\n_2)$ (note that $E_1$ and $E_2$ are not necessarily coinciding, but that their intersections with ${\rm Ann}(s,S)$ are). The currents $\n'_i$ on $E_i$ are sourceless currents (as the conditioning on the edges incident to a vertex in $E_i$ and one outside is imposing that the current $\n_i$ is either 0 or even -- when it is incident to $\partial\mathrm{Ann}(r,R)$ for instance, see Fig.~\ref{fig:karmevent}).
For the conditioned measure, the probability that there exist $k$ crossing  $\mathrm{Ann}(r,R)$-clusters in $\n_1+\n_2$ which are not in $\mathsf C$ is bounded by the $\mathbf P^{\emptyset,\emptyset}_{E_1,E_2}$-probability that there are $k$ clusters in $\n_1'+\n_2'$ that are crossing $\mathrm{Ann}(s,S)$. 

Consider the event $F_\ell$ that there are $\ell$ clusters of $\n_1'+\n'_2$ crossing $\mathrm{Ann}(s,S)$. Conditioning on $F_\ell$ and exploring the clusters of $\n_1'+\n_2'$ intersecting $\partial\Lambda_S$ one-by-one by going counterclockwise around $\partial \Lambda_S$, we deduce that
\[
\mathbf P^{\emptyset,\emptyset}_{E_1,E_2}[F_{\ell+1}|F_\ell]\le \mathbf E^{\emptyset,\emptyset}_{E_1,E_2}\big[\mathbf P^{\emptyset,\emptyset}_{\mathbf E_1(\ell),\mathbf E_2(\ell)}[\partial\Lambda_s\stackrel{\n_1+\n_2}\longleftrightarrow\partial\Lambda_S]~\big|~F_\ell\big],
\]
where $\mathbf E_i(\ell)$ is obtained from $E_i$  by removing the edges with at least one endpoint in the first $\ell\ge1$ clusters. Since $\mathbf E_1(\ell)$ and $\mathbf E_2(\ell)$ coincide on ${\rm Ann}(s,S)$, Lemma~\ref{lem:monotonicity} and Corollary~\ref{cor:upper bound crossing}  gives that a.s.
\[
\mathbf P^{\emptyset,\emptyset}_{\mathbf E_1(\ell),\mathbf E_2(\ell)}[\partial\Lambda_s\stackrel{\n_1+\n_2}\longleftrightarrow\partial\Lambda_S]\le (C_1\tfrac rR)^{\lambda_1}
\] which when averaged over $F_\ell$ gives that for every $\ell\ge1$,
\[
\mathbf P^{\emptyset,\emptyset}_{E_1,E_2}[F_{\ell+1}|F_\ell]\le(C_1\tfrac rR)^{\lambda_1}.
\]
We conclude that
\[
\mathbf P^{\emptyset,\emptyset}_{E_1,E_2}[F_{k}]\le (C_1\tfrac rR)^{k\lambda_1}.
\]

In conclusion, since any crossing is either part of the $\mathrm{Ann}(r,R)$-clusters of $\n_1+\n_2$ intersecting $\mathsf C_1\cup\mathsf C_2$, or contains a crossing of $\n'_1+\n'_2$ (since they need to connect the green parts in Fig.~\ref{fig:karmevent}), we get that
\begin{align}
{\bf P}_{\Omega,\Omega
}^{\emptyset,\emptyset}[A_{5k}(r,R)]\le (C_0\tfrac rR)^{k\lambda_0}+(C_1\tfrac rR)^{k\lambda_1}.\end{align}
The result follows readily for $5k$ instead of $2k$. 

If one wants the result for $k\le 2$, simply use that there must be a path in $\n_1+\n_2$ from $\partial\Lambda_r$ to $\partial\Lambda_R$. Lemma~\ref{lem:monotonicity} and Corollary~\ref{cor:upper bound crossing} directly imply that the probability is bounded by $C_0(\tfrac rR)^{\lambda_0}$ in this case.
\end{proof}

	\begin{figure}[t] 
		\begin{center}
			\includegraphics[scale=1]{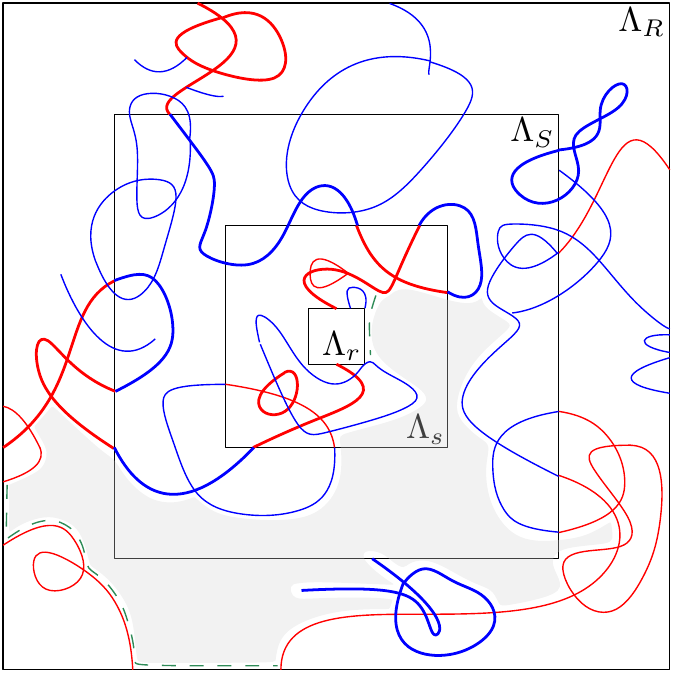}  
		\end{center}
		\caption{The depiction of the clusters of $\eta_1$ and $\eta_2$ intersecting $\partial\mathrm{Ann}(r,R)$, as well as the cluster in $\n_1+\n_2$ of the clusters of $\eta_1$ and $\eta_2$ that cross $\mathrm{Ann}(r,s)$ or $\mathrm{Ann}(S,R)$. To distinguish between $\n_1$ and $\n_2$, we depicted the former in bold. In light grey,  one connected component of the set $E_1$. The red corresponds to the events $B_{2k}(r,s)\cup B_{2k}(S,R)$ and the blue to the rest of the currents. Note that the edges in dashed green are not necessarily equal to 0 in $\n_1$, but that necessarily $\n_1$ is even. In particular, they do not impact the current $\n_1'$  in the grey area, and the current can therefore be considered as a sourceless current. The same is true of $E_2$ and $\n'_2$. Then, $\n'_1+\n_2'$ should still contain crossings of $\mathrm{Ann}(s,S)$ to add new potential $\mathrm{Ann}(r,R)$-clusters crossing $\mathrm{Ann}(r,R)$, as the part of the boundary of $E_1\cap \mathrm{Ann}(r,R)=E_2\cap\mathrm{Ann}(r,R)$ that is strictly inside $\mathrm{Ann}(s,S)$ is made of zero currents.}	\label{fig:karmevent}
\end{figure}

We conclude this paper by listing a straightforward yet important consequence of Theorem~\ref{prop:tight number crossings}.

\begin{corollary}[Tightness of the number of clusters crossing a rectangle]\label{tight rectangles}
There exist $c,C>0$ such that for every $k\ge C$, every $2R\times R$ rectangle $D$, and every domain $\Omega$ (not necessarily containing $D$),
\[
{\bf P}_{\Omega,\Omega}^{\emptyset,\emptyset}[\exists k\text{ $D$-clusters crossing $D$}]\le (1-c)^{k},
\]
\end{corollary}

\begin{proof}Fix $r=R/(2C)$ where $C$ is the constant from the previous proposition and consider $x_0,\dots, x_s$ (with $s=O(C^2)$) such that the boxes $\Lambda_r(x_i)$ cover $D$ for $0\le i\le s$. For $k$ crossings to exist, there must be one of the annuli $\mathrm{Ann}(x_i,r,R)$ that contains $k/(2C)$ clusters crossing from outside to inside. As a consequence, we deduce the result immediately from Theorem~\ref{prop:tight number crossings} by choosing $k$ large enough that $\lambda k>2C$.
\end{proof}

\begin{remark}[double random-current with wired boundary conditions]
Let us mention a result for the double random current with wired boundary conditions, meaning the double random current on the graph  $\Omega^{\mathfrak g}$ obtained from $\Omega$ by adding a {\em ghost} vertex $\mathfrak g\notin \Omega$ connected to all the vertices on $\partial\Omega$ by an edge. We also get the Aizenman-Burchard criterion for this model, with the small point that the ${\rm Ann}(r,R)$-clusters of $\mathfrak g$ are counted as a single cluster. Indeed, if $\Lambda_R\subset \Omega$ then the mixing property enables to deduce the result from the result for free boundary conditions. When $\Lambda_R$ is not contained in $\Omega$, the result still holds as one can first explore the ${\rm Ann}(r,R)$-clusters of $\mathfrak g$ in $\Omega$, and then use the same argument as for free boundary conditions in the remaining domain. \end{remark}

\bibliographystyle{amsplain}

\begin{bibdiv}
\begin{biblist}

\bib{Aiz82}{article}{
      author={Aizenman, M.},
       title={Geometric analysis of {$\varphi ^{4}$} fields and {I}sing models.
  {I}, {II}},
        date={1982},
        ISSN={0010-3616},
     journal={Comm. Math. Phys.},
      volume={86},
      number={1},
       pages={1\ndash 48},
         url={http://projecteuclid.org/getRecord?id=euclid.cmp/1103921614},
}

\bib{AizBarFer87}{article}{
      author={Aizenman, M.},
      author={Barsky, D.~J.},
      author={Fern{{\'a}}ndez, R.},
       title={The phase \mbox{transition} in a general class of {I}sing-type
  models is sharp},
        date={1987},
        ISSN={0022-4715},
     journal={J. Statist. Phys.},
      volume={47},
      number={3-4},
       pages={343\ndash 374},
         url={http://dx.doi.org/10.1007/BF01007515},
}

\bib{AizBur99}{article}{
  title={H{\"o}lder regularity and dimension bounds for random curves},
  author={Aizenman, M.},
  author={Burchard, A.},
  journal={Duke mathematical journal},
  volume={99},
  number={3},
  pages={419--453},
  year={1999},
}

\bib{AizDum19}{article}{
  title={Marginal triviality of the scaling limits of critical 4D Ising and $\phi_4^4$ models},
  author={Aizenman, M.},
  author ={ Duminil-Copin, H.},
  journal={arXiv:1912.07973},
  year={2019}
}

\bib{AizDumSid15}{article}{
      author={Aizenman, M.},
      author={{Duminil-Copin}, H.},
      author={Sidoravicius, V.},
       title={Random {C}urrents and {C}ontinuity of {I}sing {M}odel's
  {S}pontaneous {M}agnetization},
        date={2015},
     journal={Communications in Mathematical Physics},
      volume={334},
       pages={719\ndash 742},
}

\bib{ADTW}{article}{
      author={Aizenman, Michael},
      author={Duminil-Copin, Hugo},
      author={Tassion, Vincent},
      author={Warzel, Simone},
       title={Emergent planarity in two-dimensional Ising models with
  finite-range interactions},
        date={2019},
     journal={Inventiones mathematicae},
      volume={216},
      number={3},
       pages={661\ndash 743},
         url={https://doi.org/10.1007/s00222-018-00851-4},
}



\bib{CheDumHon16}{article}{
  title={Crossing probabilities in topological rectangles for the critical planar FK-Ising model},
  author={Chelkak, D.},
  author={ Duminil-Copin, H.},
  author={Hongler, C.},
  journal={Electronic Journal of Probability},
  volume={21},
  year={2016},

}




\bib{Dum17a}{unpublished}{
      author={Duminil-Copin, H.},
       title={Lectures on the {I}sing and {P}otts models on the hypercubic
  lattice},
        note={arXiv:1707.00520},
}

\bib{Dum16}{unpublished}{
      author={{Duminil-Copin}, H.},
       title={Random currents expansion of the {I}sing model},
        date={2016},
        note={arXiv:1607:06933},
}

\bib{DumGosRao18}{article}{
author={Duminil-Copin, H.},
author={Goswami, S.},
author={Raoufi, A.},
title={Exponential Decay of Truncated Correlations for the Ising Model in any Dimension for all but the Critical Temperature},
date={2020},
number={2},
volume={374},
journal={Communications in Mathematical Physics},
pages={891-921},
}

\bib{DumLis}{article}{
      author={Duminil-Copin, H.},
      author={Lis, M.},
       title={On the double random current nesting field},
        date={2019},
     journal={Probability Theory and Related Fields},
      volume={175},
      number={3-4},
       pages={937\ndash 955},
}




\bib{DumHonNol11}{article}{
  title={Connection probabilities and RSW-type bounds for the two-dimensional FK Ising model},
  author={Duminil-Copin, H.},
  author ={Hongler, C.},
  author ={Nolin, P.},
  journal={Communications on pure and applied mathematics},
  volume={64},
  number={9},
  pages={1165--1198},
  year={2011},
}

\bib{DumLisQia21}{article}{
     author={Duminil-Copin, H.},
      author={Lis, M.},
author={Qian, W.},
title={Conformal invariance of double random currents and the XOR-Ising model I: identification of the limit},
date={2021},
        note={preprint},
}


\bib{DumManTas20}{article}{
  title={Planar random cluster model: fractal properties of the critical phase},
  author={Duminil-Copin, H.},
  author ={Manolescu, I.},
  author ={Tassion, V.},
  journal={arXiv:2007.14707},
  year={2020}
}

\bib{DumSidTas17}{article}{
  title={Continuity of the Phase Transition for Planar Random cluster and Potts Models with $1\le q\le4$},
  author={Duminil-Copin, H.},
  author={Sidoravicius, V.},
  author={Tassion, V.},
  journal={Communications in Mathematical Physics},
  volume={349},
  number={1},
  pages={47--107},
  year={2017}
}


\bib{DumTas16}{article}{
  title={RSW and box-crossing property for planar percolation},
  author={Duminil-Copin, H.},
  author={Tassion, V.},
  booktitle={Proceedings of the international congress of Mathematical Physics},
  year={2016}
}


\bib{DumTas15}{article}{
      author={Duminil-Copin,~H.},
      author={Tassion, V.},
       title={A new proof of the sharpness of the phase transition for
  {B}ernoulli percolation and the {I}sing model},
        date={2016},
     journal={Communications in {M}athematical {P}hysics},
      volume={343},
      number={2},
       pages={725\ndash 745},
}

\bib{EdwSok}{article}{
  title = {Generalization of the Fortuin-Kasteleyn-Swendsen-Wang representation and Monte Carlo algorithm},
  author = {Edwards, Robert G.}
  author=  {Sokal, Alan D.},
  journal = {Phys. Rev. D},
  volume = {38},
  issue = {6},
  pages = {2009--2012},
  numpages = {0},
  year = {1988},
  month = {Sep},
  publisher = {American Physical Society},
  doi = {10.1103/PhysRevD.38.2009},
  url = {https://link.aps.org/doi/10.1103/PhysRevD.38.2009}
}



\bib{GHS}{article}{
      author={Griffiths, R.~B.},
      author={Hurst, C.~A.},
      author={Sherman, S.},
       title={{Concavity of Magnetization of an Ising Ferromagnet in a Positive
  External Field}},
        date={1970},
     journal={Journal of Mathematical Physics},
      volume={11},
      number={3},
       pages={790\ndash 795},
  url={http://scitation.aip.org/content/aip/journal/jmp/11/3/10.1063/1.1665211},
}



\bib{KemSmi17}{article}{
  title={Random curves, scaling limits and Loewner evolutions},
  author={Kemppainen, A.},
  author={Smirnov, S.},
  journal={The Annals of Probability},
  volume={45},
  number={2},
  pages={698--779},
  year={2017},
}



\bib{KraWan41}{article}{
  title={Statistics of the two-dimensional ferromagnet. Part I},
  author={Kramers, H.},
  author={Wannier, G.},
  journal={Physical Review},
  volume={60},
  number={3},
  pages={252},
  year={1941},
  publisher={APS}
}



\bib{LisT}{article}{
      author={Lis, Marcin},
       title={The planar {I}sing model and total positivity},
        date={2017},
        ISSN={0022-4715},
     journal={J. Stat. Phys.},
      volume={166},
      number={1},
       pages={72\ndash 89},
         url={https://doi.org/10.1007/s10955-016-1690-x},
}



\bib{LupWer}{article}{
      author={Lupu, T.},
      author={Werner, W.},
       title={{A note on Ising random currents, Ising-FK, loop-soups and the
  Gaussian free field}},
        date={2016},
     journal={Electron. Commun. Probab.},
      volume={21},
       pages={7 pp.},
       }




\bib{Pei36}{article}{
      author={Peierls, R.},
       title={On {I}sing's model of ferromagnetism.},
        date={1936},
     journal={Math. Proc. Camb. Phil. Soc.},
      volume={32},
       pages={477\ndash 481},
}


\bib{Rao17}{article}{
author = {Raoufi, Aran},
title = {{Translation-invariant Gibbs states of the Ising model: General setting}},
volume = {48},
journal = {The Annals of Probability},
number = {2},
publisher = {Institute of Mathematical Statistics},
pages = {760 -- 777},
keywords = {amenable graphs, Gibss states, Ising model, percolation},
year = {2020},
doi = {10.1214/19-AOP1374},
URL = {https://doi.org/10.1214/19-AOP1374}
}


\bib{Rus78}{article}{
  title = {A note on percolation},
  author = {Russo, Lucio},
  journal = {Zeitschrift f\"ur Wahrscheinlickkeitstheorie},
  volume = {43},
  number = {1},
  pages = {39 à 48},
  year = {1978},
  editor = {springer},
}



\bib{SeyWel78}{incollection}{
  title={Percolation probabilities on the square lattice},
  author={Seymour, Paul D},
  author={Welsh, Dominic JA},
  booktitle={Annals of Discrete Mathematics},
  volume={3},
  pages={227--245},
  year={1978},
  publisher={Elsevier},
}




\bib{Wilson}{article}{
AUTHOR={Wilson, D.B.},
TITLE={{XOR}-{I}sing loops and the {G}aussian free field},
JOURNAL={arXiv:1102.3782},
}



\end{biblist}
\end{bibdiv}

\end{document}